\documentclass[12pt,reqno]{amsart} 
\usepackage{amsthm}

\usepackage{amsmath,amssymb}
\usepackage{ascmac}
\usepackage{mathrsfs}

\usepackage[abbrev]{amsrefs}

\usepackage{graphicx,xcolor}

\usepackage{hyperref}

\hypersetup{
    colorlinks=true,     
    linkcolor=red,      
    citecolor=green,      
    urlcolor=blue,       
    bookmarks=true,
    bookmarksnumbered=true,
    bookmarksdepth=tocdepth
}

\allowdisplaybreaks[1]
\theoremstyle{definition}
\newtheorem{dfn}{Definition}
\newtheorem{thm}{Theorem}[section]
\newtheorem*{thm*}{Theorem}
\newtheorem{cor}[thm]{Corollary}
\newtheorem{prop}[thm]{Proposition}
\newtheorem{lem}[thm]{Lemma}
\newtheorem{rem}{Remark}[section]

\newtheorem*{nota}{Notation}

\newcommand{\eqsp}[1]{{\begin{equation}\begin{split}#1\end{split}\end{
equation}}}

\makeatletter
\newcommand{\subscripts}[3]{%
  \@mathmeasure\z@\displaystyle{#2}%
  \global\setbox\@ne\vbox to\ht\z@{}\dp\@ne\dp\z@
  \setbox\tw@\box\@ne
  \@mathmeasure4\displaystyle{\copy\tw@_{#1}}%
  \@mathmeasure6\displaystyle{{#2}_{#3}}%
  \dimen@-\wd6 \advance\dimen@\wd4 \advance\dimen@\wd\z@
  \hbox to\dimen@{}\mathop{\kern-\dimen@\box4\box6}%
}
\makeatother

\begin{document}

\begin{center}\LARGE \bf
  On the stabilizing effects of rotation on stationary Navier--Stokes equations in the half spaces 
  in scale-critical spaces 
\end{center}
   
\footnote[0]
{


* Taiki Okazaki -- E-mail: okazaki.taiki.r5@dc.tohoku.ac.jp 

}
\vskip5mm

\begin{center}
  {\large Taiki Okazaki*} 
  \vskip2mm
  {\large Mathematical Institute, Tohoku University}\\
  {\large Sendai 980-8578 Japan}
 \end{center}

\vskip5mm

\begin{center}
\begin{minipage}{135mm}
\footnotesize
{{\sc Abstract.} }
In this paper, we study the three-dimensional incompressible stationary Navier--Stokes equations with the Coriolis force in the half space with the inhomogeneous Dirichlet boundary condition. 
We show the existence of a unique small solution in the scale-critical Besov spaces. 
In particular, for any boundary data with a vanishing normal component, 
and for any external forces, we establish the well-posedness of the inhomogeneous Dirichlet boundary value problem when the speed of rotation is sufficiently large. 
In the proof, to exploit the stabilizing effect of the Coriolis term, 
we employ two explicit formulas depending on the relation between the horizontal frequency and the magnitude of the rotation speed. 
\end{minipage}
\end{center}

\section{Introduction}
In this paper, we consider the three-dimensional incompressible stationary Navier--Stokes equations 
with the Coriolis force in the half space 
\begin{equation*}
  \mathbb{R}_{+}^{3} := \{x = (x', x_{3})|\ x' = (x_{1}, x_{2})\in \mathbb{R}^{2}, x_{3}>0\}
\end{equation*}
with the inhomogeneous Dirichlet boundary condition. 
\begin{equation}\label{SNSC}
  \begin{cases}
    -\nu\Delta {\bf u} + \Omega(e_{3}\times{\bf u}) + ({\bf u}\cdot \nabla){\bf u} + \nabla p = \text{div } {\bf F}, 
          &\quad  \   x\in \mathbb{R}_{+}^{3}, \\
    \text{div } {\bf u} = 0,
          &\quad  \   x\in \mathbb{R}_{+}^{3}, \\
    {\bf u}(x', 0) = {\bf a}(x'),
          &\quad  \ x'\in{\mathbb{R}}^2, 
  \end{cases}
\end{equation}
where ${\bf u}(x) = (u_{1}(x), u_{2}(x), u_{3}(x))$ denotes the unknown velocity field of a fluid 
and $p(x)$ denotes the unknown pressure of the fluid, which acts to enforce the incompressibility condition $\text{div }{\bf u} = 0$, 
while the external force is assumed to be given in divergence form for some given matrix-valued function ${\bf F}(x) = \{F_{m,n}\}_{1\leq m,n \leq 3}$, that is, 
\begin{equation*}
  \text{div }{\bf F} 
  = \left(\sum_{m = 1}^{3}\partial_{m}F_{m, 1}, \sum_{m = 1}^{3}\partial_{m}F_{m, 2}, \sum_{m = 1}^{3}\partial_{m}F_{m, 3}\right), 
\end{equation*}
and ${\bf a}(x') = (a_{1}(x'), a_{2}(x'), a_{3}(x'))$ denotes the given boundary data. 
The constant $\nu > 0$ is the kinematic viscosity, and 
the constant $\Omega \in \mathbb{R}$ is the speed of rotation around the vertical unit vector $e_{3} = (0, 0, 1)$. 
The aim of this paper is to study the effect of the Coriolis force on \eqref{SNSC} from the viewpoint of the existence of solutions in scale-critical spaces. 
In particular, for any boundary data with $a_{3} = 0$ 
and for any ${\bf F}$, we show the existence of a unique small solution of \eqref{SNSC} when $\Omega$ is sufficiently large. 

We begin by recalling several known results. 
We first consider the case $\Omega = 0$. 
There is a vast literature on the stationary Navier--Stokes equations in the whole space $\mathbb{R}^{d}$ $(d \geq 3)$, 
\begin{equation}\label{0621-1}
  \begin{cases}
    -\nu\Delta {\bf u} + ({\bf u}\cdot \nabla){\bf u} + \nabla p = \text{div } {\bf F}, 
          &\quad  \   x\in \mathbb{R}^{d}, \\
    \text{div } {\bf u} = 0,
          &\quad  \   x\in \mathbb{R}^{d}. 
  \end{cases}
\end{equation}
Here, we only discuss works that are directly relevant to the present study. 
For further details, see the references below and the references therein. 
We briefly discuss the scale-critical spaces of \eqref{0621-1}. 
Suppose that $({\bf u}, p)$ is a solution of \eqref{0621-1}.
For $\lambda > 0$, we set 
\begin{equation}\label{0621-2}
  {\bf u}_{\lambda}(x) := \lambda {\bf u}(\lambda x),\quad 
  p_{\lambda}(x) := \lambda^{2}p(\lambda x) 
  \quad 
  \text{and} 
  \quad 
  {\bf F}_{\lambda}(x) := \lambda^{2}{\bf F}(\lambda x). 
\end{equation}
Then for given ${\bf F}_{\lambda}$, $({\bf u}_{\lambda}, p_{\lambda})$ is also a solution of \eqref{0621-1}. 
That is, \eqref{0621-1} is invariant under this scaling. 
A function space is said to be a scale-critical space for \eqref{0621-1} if its norm is invariant under the scaling \eqref{0621-2}. 
For instance, in the framework of Lebesgue spaces, the scale-critical spaces of \eqref{0621-1} are ${\bf F} \in L^{\frac{d}{2}}(\mathbb{R}^{d})$ and ${\bf u}\in L^{d}(\mathbb{R}^{d})$. 
Also, in the Besov space setting, the scale-critical spaces of \eqref{0621-1} are ${\bf F}\in \dot{B}_{p,q}^{\frac{d}{p} - 2}(\mathbb{R}^{d})$ and ${\bf u}\in \dot{B}_{p,q}^{\frac{d}{p} - 1}(\mathbb{R}^{d})$. 
It is known that scale-critical spaces plays an important role in the study of solvability of partial differential equations. 
Indeed, for small ${\bf F}\in L^{\frac{d}{2}}(\mathbb{R}^{d})$, Chen~\cite{Ch_1993} proved the existence of a unique small solution ${\bf u} \in L^{d}(\mathbb{R}^{d})$, 
and Kaneko, Kozono and Shimizu~\cite{Ka_Ko_Sh_2019} showed that for small ${\bf F}\in \dot{B}_{p,q}^{\frac{d}{p} - 2}(\mathbb{R}^{d})$ with $1 \leq p < d$ and $1 \leq q \leq \infty$, 
there exists a unique small solution ${\bf u}\in \dot{B}_{p,q}^{\frac{d}{p} - 1}(\mathbb{R}^{d})$, 
whereas if $d < p \leq \infty$ and $1\leq q \leq \infty$ or $p = d$ and $2 < q \leq \infty$, then the ill-posedness of \eqref{0621-1} in Besov spaces in the sense that the solution map is discontinuous was established by Tsurumi~\cite{Tsu_2019, Tsu_2019_2}. 
Recently, Fujii~\cite{Fu_2026} constructed non-unique solutions of \eqref{0621-1} in $\dot{B}_{p,q}^{\frac{d}{p} - 1}(\mathbb{R}^{d})$ with $d < p \leq \infty$ and $1\leq q \leq \infty$ or $p = d$ and $2 < q \leq \infty$ in the case ${\bf F} \equiv 0$, 
and then obtained non-unique solutions to the non-stationary problem by considering perturbations around them. 
When the domain is the half space $\mathbb{R}_{+}^{d}$, that is, when $\Omega = 0$ in \eqref{SNSC} and the spatial dimension is generalized to $d \geq 3$, 
Fujii~\cite{Fu_2025} derived an integral equation formulation of the problem by viewing the equation as an evolution equation in the normal direction. 
Moreover, based on the analysis of this formulation, \cite{Fu_2025} proved the existence of a unique small solution which belongs to the following scale-critical space: 
\begin{equation}\label{0621-3}
  \widetilde{L}_{x_{d}}^{\infty}(0, \infty ; \dot{B}_{p,q}^{\frac{d - 1}{p} - 1}(\mathbb{R}^{d - 1})), 
\end{equation}
for small ${\bf a} \in \dot{B}_{p,q}^{\frac{d - 1}{p} - 1}(\mathbb{R}^{d - 1})$ and ${\bf F} \in \widetilde{L}_{x_{d}}^{q}(0, \infty ; \dot{B}_{p,q}^{\frac{d - 1}{p} + \frac{1}{q} - 2}(\mathbb{R}^{d - 1}))$. 
A detailed definition of \eqref{0621-3} will be given below (Definition~\ref{0623-4}). 
Note that the fact that \eqref{0621-3} is a scale-critical space of the incompressible stationary Navier--Stokes equations in $\mathbb{R}_{+}^{d}$ follows immediately from \eqref{0621-2}, 
expressed as 
\begin{equation*}
  {\bf u}_{\lambda}(x) 
  = \lambda{\bf u}(\lambda x', \lambda x_{d}), 
\end{equation*}
where $x' := (x_{1}, x_{2}, \cdots , x_{d-1}) \in \mathbb{R}^{d-1}$ and $x_{d} \in (0, \infty)$. 
Also, the scaling of ${\bf a}$ corresponding to \eqref{0621-2} is 
\begin{equation*}
  {\bf a}_{\lambda}(x') 
  := \lambda{\bf a}(\lambda x'). 
\end{equation*}
In this paper, we study \eqref{SNSC} within the same framework as that of \cite{Fu_2025}. 

We next review previous studies in the case $\Omega \neq 0$. 
We begin by discussing previous results on the non-stationary problem in the whole space $\mathbb{R}^{3}$, 
\begin{equation}\label{0621-4}
  \begin{cases}
    \partial_{t}{\bf u} -\nu\Delta {\bf u} + \Omega(e_{3}\times{\bf u}) + ({\bf u}\cdot \nabla){\bf u} + \nabla p = \text{div } {\bf F}, 
          &\quad  \   t > 0, x\in \mathbb{R}^{3}, \\
    \text{div } {\bf u} = 0,
          &\quad  \   t > 0, x\in \mathbb{R}^{3}, \\
    {\bf u}(0, x) = {\bf u}_{0}(x),
          &\quad  \ x\in{\mathbb{R}}^{3}, 
  \end{cases}
\end{equation}
where ${\bf u}_{0}$ is the given initial data. 
In equation \eqref{0621-4}, the term $e_{3}\times {\bf u}$ represents the effect of rotation in large-scale geophysical flows, 
where $\Omega$ denotes the angular velocity of rotation. 
As suggested by the Taylor--Proudman theorem, when rotational effects are dominant, the flow described by \eqref{0621-4} is expected to be quasi-two-dimensional. 
For a detailed physical discussion, we refer the reader to the books \cite{Gr_1980, Pe_1979}. 
From a mathematical point of view, Chemin, Desjardins, Gallagher and Grenier~\cite{Ch_De_Ga_Gr_2002} proved that 
for any given initial data ${\bf u}_{0} = {\bf v}_{0} + {\bf w}_{0} \in L^{2}(\mathbb{R}^{2})^{3} + H^{\frac{1}{2}}(\mathbb{R}^{3})^{3}$ with $\text{div } {\bf u}_{0} = 0$, there exists a constant $\Omega_{0}$ such that 
for any $\Omega \in \mathbb{R}$ with $|\Omega| \geq \Omega_{0}$, there exists a unique global solution of \eqref{0621-4}. 
Moreover, \cite{Ch_De_Ga_Gr_2002} showed that the above solution converges, in the limit $|\Omega| \to \infty$, to the solution of the two-dimensional Navier--Stokes equations with initial data ${\bf v}_{0}$. 
Although \eqref{0621-4} does not possess the same scaling invariance as in the case $\Omega = 0$, 
the scale-critical spaces for the case $\Omega = 0$ are often still referred to as scale-critical spaces for \eqref{0621-4}, 
and have been extensively studied. 
We refer the reader to \cite{Gi_I_Ma_Ma_2005, Gi_I_Ma_Sa_2008, Hi_Sh_2010, Ko_Yo_2011, I_Ta_2014} for well-posedness results that are independent of $\Omega$. 
These results can be viewed as generalizations of the corresponding results for the case $\Omega = 0$ (see, e.g., \cite{Ka_1984, Ko_Ta_2001}) to the setting with the parameter $\Omega$. 
We investigate the effect of the Coriolis term on the solutions to \eqref{0621-4}. 
It is known that the following Fourier multiplier operator appears in the solution formula for the linear problem corresponding to \eqref{0621-4}, 
\begin{equation}\label{0621-5}
  \mathcal{F}_{\mathbb{R}^{3}}^{-1}\left[e^{-\nu t|\xi|^{2}}e^{\pm i\Omega t\frac{\xi_{3}}{|\xi|}}\mathcal{F}_{\mathbb{R}^{3}}[f](\xi)\right]. 
\end{equation}
Here, $\mathcal{F}_{\mathbb{R}^{3}}$ and $\mathcal{F}_{\mathbb{R}^{3}}^{-1}$ denote the Fourier transform and the inverse Fourier transform on $\mathbb{R}^{3}$, respectively. 
For an explicit formula for solutions to the linear problem associated with \eqref{0621-4} and its derivation, see, e.g., \cite{Gi_I_Ma_Ma_2006, Hi_Sh_2010}. 
From \eqref{0621-5}, the Coriolis term is expected to induce anisotropic dispersion. 
Indeed, in the non-viscous setting, a dispersive estimate for the semigroup generated by the Coriolis term was obtained by Dutrifoy~\cite{Du_2005}, 
and subsequently improved by Iwabuchi and Takada~\cite{I_Ta_2015} and Koh, Lee and Takada~\cite{Ko_Le_Ta_2014}. 
Using the dispersive estimate, \cite{Ko_Le_Ta_2014} proved that if ${\bf u}_{0} \in \dot{H}^{s}(\mathbb{R}^{3})$ with $s>1/2$, i.e., in the scale-sub-critical case, 
then there exists a constant $\Omega_{0} > 0$ such that for any $\Omega \in \mathbb{R}$ with $|\Omega| > \Omega_{0}$, 
there exists a unique global solution of \eqref{0621-4}. 
More importantly, $\Omega_{0}$ is explicitly determined by the $\dot{H}^{s}$ norm of ${\bf u}_{0}$. 

The effect of the Coriolis term also appears in the stationary problem associated with \eqref{0621-4}. 
Konieczny and Yoneda~\cite{Ko_Yo_2011} showed that for any ${\bf F}$ belonging to a scale-critical Fourier--Besov space $\dot{FB}_{p,p}^{1 - \frac{3}{p}}(\mathbb{R}^{3})$ $(3 < p < \infty)$, 
there exists a constant $\Omega_{0} > 0$ such that for any $\Omega \in \mathbb{R}$ with $|\Omega| \geq \Omega_{0}$, 
there exists a unique solution ${\bf u}\in \dot{FB}_{p,p}^{2 - \frac{3}{p}}(\mathbb{R}^{3})$ of the stationary problem associated with \eqref{0621-4}. 
It should be noted that the well-posedness results for the stationary Navier--Stokes equations \eqref{0621-1} discussed above are obtained under a smallness assumption on ${\bf F}$. 

We now consider \eqref{0621-4} in domains other than the whole space. 
For the torus $\mathbb{T}^{3} = [0, 2\pi]^{3}$ with periodic boundary conditions, results similar to those in the $\mathbb{R}^{3}$ setting were obtained by Babin, Mahalov and Nicolaenko~\cite{Ba_Ma_Ni_1997, Ba_Ma_Ni_1999, Ba_Ma_Ni_2001}. 
Also, for results on a three-dimensional layer $\mathbb{R}^{2} \times \mathbb{T}$ with periodic boundary conditions, 
we refer the reader to Ohyama and Takada~\cite{Oh_Ta_2021}. 
In contrast to these results, the study of \eqref{0621-4} in domains with Dirichlet boundary conditions cannot be treated in the same manner. 
One reason for this is that boundary layer effects have to be taken into account. 
As is well known from the Prandtl boundary layer for $\Omega = 0$, 
boundary layer phenomena play an important role in singular limits of viscous flows subject to Dirichlet boundary conditions, including the limits $\nu \to 0$ and $|\Omega| \to +\infty$. 
Such boundary layer effects are typically significant near the boundary but become negligible away from it. 
In viscous rotating flow, this boundary layer is known as the Ekman boundary layer~\cite{Ek_1905}. 
This boundary layer has been extensively studied in the literature; see, for instance, \cite{Ch_De_Ga_Gr_2006, Ge_2005, Gr_1980, Gr_Ma_1997, Ma_2000, Pe_1979}. 
In particular, its mathematical justification has been investigated in \cite{Gr_Ma_1997, Ma_2000, Ch_De_Ga_Gr_2002_2}. 
In the half-space $\mathbb{R}_{+}^{3}$ with homogeneous Dirichlet boundary conditions, 
Giga, Inui, Mahalov, Matsui, and Saal~\cite{Gi_In_Ma_Ma_Sa_2007} obtained a local well-posedness result around the Ekman spiral, 
which is a time-independent exact solution. 
All of the above results are obtained for domains with flat boundaries. 
Recently, Chemin, Fanelli and Gallagher~\cite{Ch_Fa_Ga_2026} established analogous results for domains with general non-flat boundaries arising from bottom topographies. 
We also mention the work of Koba~\cite{Ko_2015} on global well-posedness and asymptotic stability for related systems with small initial data in general domains, including the half-space. 
The present work does not address a singular limit problem, 
but rather focuses on the stationary problem for sufficiently large rotation parameters. 
We refer to \cite{Da_Pr_2014, Da_Ge_2017} for well posedness results concerning the stationary problem in a half-space with a rough surface, 
which is related to but distinct from the problem considered here. 
In particular, Dalibard and Prange~\cite{Da_Pr_2014} derived an explicit formula for \eqref{SNSC} with $\Omega = 1$ and ${\bf F}\equiv 0$. 
This formula will play an important role in our analysis; see Subsection~\ref{0623-2} and Subsection~\ref{0623-3}. 

In this work, in the same framework as that of \cite{Fu_2025}, we derive and analyze an explicit formula for \eqref{SNSC}, 
and investigate the effect of the Coriolis term on \eqref{SNSC} from the viewpoint of well-posedness in scale-critical spaces. 
In what follows, we assume that 
\begin{equation*}
  \nu = 1 
  \quad \text{and}\quad a_{3} = 0. 
\end{equation*}
We first present a simplified statement of the main theorem of this paper, 
which is the central result we wish to emphasize. 
The following result illustrates the stabilizing effect of the Coriolis term on \eqref{SNSC}. 
\vskip2mm
\noindent 
{\bf Simplified statement of Theorem~\ref{0609-4}.}
For any boundary data ${\bf a}$ with $a_{3} = 0$ and any matrix-valued function ${\bf F}$ in appropriate scale-critical spaces for \eqref{SNSC}, 
there exists a sufficiently large constant $\Omega_{0} > 0$ such that for any $\Omega \in \mathbb{R}\setminus\{0\}$ with $|\Omega| \geq \Omega_{0}$, 
there exists a unique small solution ${\bf u}$ of \eqref{SNSC}. 
\vskip2mm
To give precise statements of our main results, we first introduce some notation. 
We begin with the functional setting used throughout this paper. 
Let $d \in \mathbb{N}$. 
Let $\mathcal{S}(\mathbb{R}^{d})$ be the Schwartz space on $\mathbb{R}^{d}$, 
and let $\mathcal{S}'(\mathbb{R}^{d})$ be the space of tempered distributions on $\mathbb{R}^{d}$. 
For $f \in \mathcal{S}(\mathbb{R}^{d})$, we define the Fourier transform and the inverse Fourier transform on $\mathbb{R}^{d}$ by 
\begin{equation*}
  \mathcal{F}_{\mathbb{R}^{d}}[f](\xi) 
  := \frac{1}{(2\pi)^{\frac{d}{2}}}\int_{\mathbb{R}^{d}}e^{-ix\cdot\xi}f(x)~{\rm d}x 
  \quad 
  \text{and} 
  \quad 
  \mathcal{F}_{\mathbb{R}^{d}}^{-1}[f](x) 
  := \frac{1}{(2\pi)^{\frac{d}{2}}}\int_{\mathbb{R}^{d}}e^{ix\cdot\xi}f(\xi)~{\rm d}\xi. 
\end{equation*}
In particular, we denote 
\begin{equation*}
  \xi' = (\xi_{1}, \xi_{2})\in \mathbb{R}^{2} 
  \quad 
  \text{and} 
  \quad 
  \widehat{f}(\xi') := \mathcal{F}_{\mathbb{R}^{2}}[f](\xi'). 
\end{equation*}
For $f,g\in \mathcal{S}(\mathbb{R}^{2})$, we also define the convolution on $\mathbb{R}^{2}$ by 
\begin{equation*}
  f*g (x') 
  := \int_{\mathbb{R}^{2}}f(x' - y')g(y')~{\rm d}y' 
  = 2\pi\mathcal{F}_{\mathbb{R}^{2}}^{-1}[\widehat{f}(\xi')\widehat{g}(\xi')](x'). 
\end{equation*}
With the above notation, let us recall the definition of homogeneous Besov spaces $\dot{B}_{p,q}^{s}(\mathbb{R}^{2})$. 
We refer to the book by Triebel~\cite{Tr_1983}. 
\begin{dfn}
  Let $\{\phi_{j}\}_{j \in \mathbb{Z}}$ be such that 
  \begin{equation*}
    \text{supp } \widehat{\phi}_{j}\subset\{\xi'\in\mathbb{R}^{2}\ |\ 2^{j-1}\leq |\xi'| \leq 2^{j+1}\} 
    \text{ for any }j\in\mathbb{Z}, 
  \end{equation*}
  \begin{equation*}
    \sum_{j\in\mathbb{Z}}\widehat{\phi_j}(\xi')=1 \text{ for any }\xi'\in\mathbb{R}^{2}\backslash\{0\}. 
  \end{equation*}
  For $s \in \mathbb{R}$ and $1\leq p, q \leq \infty$, we define the homogeneous Besov spaces as follows. 
  \begin{equation*}
    \dot{B}_{p,q}^{s} 
    = \dot{B}_{p,q}^{s}(\mathbb{R}^{2}):=\{f\in \mathcal{S}'(\mathbb{R}^{2})/ \mathcal{P}(\mathbb{R}^{2}) |\ {\|f\|}_{B^s_{p,q}}<\infty\}, 
  \end{equation*}
  where $\mathcal{P}(\mathbb{R}^{2})$ is the set of all polynomials of $2$ real variables and 
  \begin{equation}\label{0625-15}
    {\|f\|}_{\dot{B}_{p,q}^{s}} 
    := {\left\|\left\{2^{sj}{\|\phi_j*f\|}_{L^p}\right\}_{j\in\mathbb{Z}}\right\|}_{l^q}. 
  \end{equation}
\end{dfn}
\begin{rem}
  In the above definition, note that $f$ is chosen as a representative of the equivalence class. 
  However, if $s < 2/p$ or $s = 2/p$ and $q = 1$, then the following isomorphism follows (see, e.g., \cite{Ko_Ya_1994}). 
  \begin{equation*}
    \dot{B}_{p,q}^{s} 
    \simeq \Big\{f \in \mathcal{S}'(\mathbb{R}^{2}) |\ {\|f\|}_{B^s_{p,q}}<\infty, f = \sum_{j \in \mathbb{Z}}\phi_{j}*f \text{ in }\mathcal{S}'(\mathbb{R}^{2}) \Big\}. 
  \end{equation*}
  All the spaces used in this study satisfy the condition $s < 2/p$. 
\end{rem}

We interpret \eqref{SNSC} as an evolution equation in the normal direction. 
Accordingly, we use the Chemin--Lerner spaces $\widetilde{L}_{x_{3}}^{r}(0, \infty ; \dot{B}_{p,q}^{s}(\mathbb{R}^{2}))$. 
The Chemin--Lerner spaces were first introduced in \cite{Ch_Le_1995}. 
One of the advantages of this framework is that it provides maximal regularity estimates for evolution operators (see e.g., \cite{Da_2000}). 
Let $\mathcal{D}((0, \infty)\times \mathbb{R}^{2})$ be the set of all test functions on $(0, \infty)\times \mathbb{R}^{2}$, 
and let $\mathcal{D}'((0, \infty)\times \mathbb{R}^{2})$ be the topological dual of $\mathcal{D}((0, \infty)\times \mathbb{R}^{2})$. 
\begin{dfn}\label{0623-4}
  Let $s\in\mathbb{R}$ and $1\leq p,q,r\leq\infty$. 
  We define the Chemin-Lerner space as follows. 
  \begin{equation*}
    \widetilde{L}_{x_{3}}^{r}(0,\infty;\dot{B}_{p,q}^{s})=\widetilde{L}_{x_{3}}^{r}(0,\infty;\dot{B}_{p,q}^{s}(\mathbb{R}^{2}))
    :=\left\{f\in \mathcal{D}'((0,\infty)\times\mathbb{R}^{2})|\ {\|f\|}_{\widetilde{L}_{x_{3}}^{r}(0,\infty;\dot{B}_{p,q}^s)}<\infty\right\}, 
  \end{equation*}
  where 
  \begin{equation*}
    {\|f\|}_{\widetilde{L}_{x_{3}}^{r}(0,\infty;\dot{B}_{p,q}^{s})}:=
    {\left\|\left\{2^{sj}{\|\phi_j*f\|}_{L_{x_{3}}^{r}(0,\infty;L^{p})}\right\}_{j\in\mathbb{Z}}\right\|}_{l^{q}}. 
  \end{equation*}
\end{dfn}
As in \cite{Fu_2025}, we decompose the solution ${\bf u}$ of \eqref{SNSC} into a solution of a nonlinear equation defined on $\mathbb{R}^{3}$, 
which possesses symmetry with respect to $x_{3}$, 
and a solution of a linear equation that corrects the boundary conditions. 
In addition, we employ different equations depending on the relationship between the horizontal frequency and the size of the parameter $\Omega$. 
Assume that ${\bf u}$ is a smooth solution of \eqref{SNSC} 
for boundary data ${\bf a}$ with $a_{3} = 0$ and matrix-valued function ${\bf F}$. 
Let $\bar{\bf u} = (\bar{u}_{1}, \bar{u}_{2}, \bar{u}_{3})$ be an extension of ${\bf u}$ to $\mathbb{R}^{3}$ defined as 
\begin{equation}\label{0625-6}
  \bar{\bf u}_{h}(x', x_{3}) 
  := 
  \begin{cases}
    {\bf u}_{h}(x', x_{3}), 
    & \ x_{3} \geq 0,\\ 
    {\bf u}_{h}(x', -x_{3}), 
    & \ x_{3} < 0, 
  \end{cases}
  \quad 
  \text{and} 
  \quad 
  \bar{u}_{3}(x', x_{3}) 
  := 
  \begin{cases}
    u_{3}(x', x_{3}) 
    & \ x_{3} \geq 0,\\ 
    - u_{3}(x', -x_{3}) 
    & \ x_{3} < 0. 
  \end{cases}
\end{equation}
Here ${\bf u}_{h} := (u_{1}, u_{2})$. 
Namely, $\bar{\bf u}_{h}$ is an even extension of ${\bf u}_{h}$ and $\bar{u}_{3}$ is an odd extension of $u_{3}$. 
Note that since $a_{3} = 0$, $\bar{\bf u}$ satisfies the incompressibility condition $\text{div } \bar{\bf u} = 0$ on $\mathbb{R}^{3}$. 
Also, for the nonlinear term $({\bf u}\cdot \nabla){\bf u} = \text{div }({\bf u}\otimes{\bf u})$, 
$(\text{div }(\bar{\bf u}\otimes\bar{\bf u}))_{1}$ and $(\text{div }(\bar{\bf u}\otimes\bar{\bf u}))_{2}$ are even in $x_{3}$, 
and $(\text{div }(\bar{\bf u}\otimes\bar{\bf u}))_{3}$ is odd in $x_{3}$. 
\noindent 
Similarly, let $\bar{\bf F} = \{\bar{F}_{m,n}\}$ be an extension of ${\bf F}$ to $\mathbb{R}^{3}$ defined as 
for $1 \leq m,n \leq 2$, 
\begin{equation}\label{0625-7}
  \bar{F}_{m,n}(x', x_{3}) 
  := 
  \begin{cases}
    F_{m,n}(x', x_{3}), 
    & \ x_{3} \geq 0,\\ 
    F_{m,n}(x', -x_{3}), 
    & \ x_{3} < 0, 
  \end{cases}
  \quad 
  \bar{F}_{3,3}(x', x_{3}) 
  := 
  \begin{cases}
    F_{3,3}(x', x_{3}) 
    & \ x_{3} \geq 0,\\ 
    F_{3,3}(x', -x_{3}) 
    & \ x_{3} < 0, 
  \end{cases}
\end{equation}
and 
\begin{equation}\label{0625-8}
  \bar{F}_{m,3}(x', x_{3}) 
  := 
  \begin{cases}
    F_{m,3}(x', x_{3}), 
    & \ x_{3} \geq 0,\\ 
    - F_{m,3}(x', -x_{3}), 
    & \ x_{3} < 0, 
  \end{cases}
  \quad 
  \bar{F}_{3,n}(x', x_{3}) 
  := 
  \begin{cases}
    F_{3,n}(x', x_{3}) 
    & \ x_{3} \geq 0,\\ 
    - F_{3,n}(x', -x_{3}) 
    & \ x_{3} < 0. 
  \end{cases}
\end{equation}
Then $(\text{div }\bar{\bf F})_{1}$ and $(\text{div }\bar{\bf F})_{2}$ are even in $x_{3}$, 
and $(\text{div }\bar{\bf F})_{3}$ is odd in $x_{3}$. 

We decompose ${\bf u}$ in two different ways. 
First, we regard the Coriolis term as a perturbation. 
For $({\bf u}, {\bf F})$, let $({\bf w}^{(1)}, \pi^{(1)})$ satisfy 
\begin{equation*}
  \begin{cases}
  -\Delta{\bf w}^{(1)} + \nabla \pi^{(1)} = \text{div }(\bar{\bf F} - \bar{\bf u}\otimes\bar{\bf u})
    - \Omega\left(e_{3}\times\bar{\bf u}\right), 
  &\quad \  x\in\mathbb{R}^{3},\\
  \text{div } {\bf w}^{(1)} =0,
  &\quad \  x\in\mathbb{R}^{3}. 
  \end{cases}
\end{equation*}
We set 
\begin{equation*}
  {\bf v}^{(1)} 
  := {\bf u} - {\bf w}^{(1)}|_{\mathbb{R}_{+}^{3}} 
  \quad 
  \text{and} 
  \quad 
  q^{(1)} 
  := p - \pi^{(1)}|_{\mathbb{R}_{+}^{3}}, 
\end{equation*}
then $({\bf v}^{(1)}, q^{(1)})$ should solve 
\begin{equation*}
  \begin{cases}
    -\Delta {\bf v}^{(1)} + \nabla q^{(1)} = 0, 
          &\quad  \   x\in \mathbb{R}_{+}^{3}, \\
    \text{div } {\bf v}^{(1)} = 0,
          &\quad  \   x\in \mathbb{R}_{+}^{3}, \\
    {\bf v}^{(1)}(x', 0) =  {\bf a}(x') - {\bf w}^{(1)}(x', 0),
          &\quad  \ x'\in{\mathbb{R}}^{2}
  \end{cases}
\end{equation*}
Therefore, we can decompose ${\bf u} = {\bf w}^{(1)}|_{\mathbb{R}_{+}^{3}} + {\bf v}^{(1)}$. 

\begin{rem}
  The assumption $a_{3} = 0$ is used only to ensure that the extension $\bar{\bf u}$ employed in the present work satisfies the incompressibility condition on $\mathbb{R}^{3}$. 
  This property is used in estimating the perturbation term $\Omega\left(e_{3}\times\bar{\bf u}\right)$ (see Subsection~\ref{0612-2} and Subsection~\ref{1029-1}). 
  We do not address the case $a_{3} \neq 0$ in the present paper. 
\end{rem}

In another decomposition of ${\bf u}$, 
we regard the Coriolis term as the leading term. 
That is, let $({\bf w}^{(2)}, \pi^{(2)})$ satisfy 
\begin{equation}\label{0624-1}
  \begin{cases}
  -\Delta{\bf w}^{(2)} + \Omega(e_{3}\times{\bf w}^{(2)}) + \nabla \pi^{(2)} = \text{div }(\bar{\bf F} - \bar{\bf u}\otimes\bar{\bf u}), 
  &\quad \  x\in\mathbb{R}^{3},\\
  \text{div } {\bf w}^{(2)} =0,
  &\quad \  x\in\mathbb{R}^{3}. 
  \end{cases}
\end{equation}
Then ${\bf v}^{(2)} := {\bf u} - {\bf w}^{(2)}|_{\mathbb{R}_{+}^{3}}$ and $q^{(2)} := p - \pi^{(2)}|_{\mathbb{R}_{+}^{3}}$ should solve 
\begin{equation}\label{0624-2}
  \begin{cases}
    -\Delta {\bf v}^{(2)} + \Omega(e_{3}\times{\bf v}^{(2)}) + \nabla q^{(2)} = 0, 
          &\quad  \   x\in \mathbb{R}_{+}^{3}, \\
    \text{div } {\bf v}^{(2)} = 0,
          &\quad  \   x\in \mathbb{R}_{+}^{3}, \\
    {\bf v}^{(2)}(x', 0) =  {\bf a}(x') - {\bf w}^{(2)}(x', 0),
          &\quad  \ x'\in{\mathbb{R}}^2. 
  \end{cases}
\end{equation}
Thus, we obtain ${\bf u} = {\bf w}^{(2)}|_{\mathbb{R}_{+}^{3}} + {\bf v}^{(2)}$. 

We present the reason for using the two decompositions $({\bf w}^{(1)}, {\bf v}^{(1)})$ and $({\bf w}^{(2)}, {\bf v}^{(2)})$. 
In order to study the effect of the Coriolis term in \eqref{SNSC}, the decomposition $({\bf w}^{(2)}, {\bf v}^{(2)})$ is suitable. 
However, as shown in Subsection~\ref{0623-3}, 
for the decomposition $({\bf w}^{(2)}, {\bf v}^{(2)})$, different estimates are derived depending on the relationship between the horizontal frequency $|\xi'|$ and the magnitude of $|\Omega|^{\frac{1}{2}}$. 
In fact, the leading-order terms of the eigenvalues of the linear operators appearing in the explicit formulas for $({\bf w}^{(2)}, {\bf v}^{(2)})$ differ depending on the relationship between $|\xi'|$ and $|\Omega|^{\frac{1}{2}}$. 
This correspondence can be viewed as being related to the balance between the viscous term and the Coriolis term in \eqref{0624-1} and \eqref{0624-2}. 
Since our aim is to study the effect of sufficiently rapid rotation, we focus on the decomposition $({\bf w}^{(2)}, {\bf v}^{(2)})$ in the frequency region where $|\Omega|^{\frac{1}{2}}$ is too large compared with $|\xi'|$. 
On the other hand, when $|\xi'|$ is sufficiently larger than $|\Omega|^{\frac{1}{2}}$, 
it is more convenient to use decomposition $({\bf w}^{(1)}, {\bf v}^{(1)})$, in which the Coriolis term is treated as a perturbation, 
since the required estimates can be derived more directly. 
In this case, in the explicit formula for ${\bf w}^{(1)}$, the parameter $\Omega$ appearing in the perturbation term can be replaced by a second-order derivative (see Subsection~\ref{0624-3}). 
In Section~\ref{0624-4}, we provide explicit formulas for $({\bf w}^{(1)}, {\bf v}^{(1)})$ and $({\bf w}^{(2)}, {\bf v}^{(2)})$, 
together with their estimates, within the present framework. 

We introduce the following frequency cutoff operators. 
\begin{dfn}\label{0625-14}
  Let $\Omega \in \mathbb{R}$, $C\geq 1$ and $j\in \mathbb{Z}$. 
  For $f \in \mathcal{S}'(\mathbb{R}^{2})/ \mathcal{P}(\mathbb{R}^{2})$, we define 
  \begin{equation*}
    P_{2^{j} > C|\Omega|^{\frac{1}{2}}}f 
    := \sum_{2^{j} > C|\Omega|^{\frac{1}{2}}}\phi_{j}*f, 
    \quad 
    P_{C^{-1}|\Omega|^{\frac{1}{2}} \leq 2^{j} \leq C|\Omega|^{\frac{1}{2}}}f 
    := \sum_{C^{-1}|\Omega|^{\frac{1}{2}} \leq 2^{j} \leq C|\Omega|^{\frac{1}{2}}}\phi_{j}*f 
  \end{equation*}
  and 
  \begin{equation*}
    P_{2^{j} < C^{-1}|\Omega|^{\frac{1}{2}}}f 
    := \sum_{2^{j} < C^{-1}|\Omega|^{\frac{1}{2}}}\phi_{j}*f. 
  \end{equation*}
  Note that we have 
  \begin{equation*}
    \left(P_{2^{j} > C|\Omega|^{\frac{1}{2}}} + P_{C^{-1}|\Omega|^{\frac{1}{2}} \leq 2^{j} \leq C|\Omega|^{\frac{1}{2}}} + P_{2^{j} < C^{-1}|\Omega|^{\frac{1}{2}}}\right)f 
    = f. 
  \end{equation*}
\end{dfn}

We now state our main results. 
The first theorem establishes a well-posedness result that is independent of $\Omega$ under the assumption that $a_{3} = 0$. 
\begin{thm}\label{0609-3}
  Let $\Omega \in \mathbb{R}$, 
  and let $1\leq p,r_{0}, r < \infty$
  and $1\leq q < \infty$ with 
  \begin{equation}\label{0625-12}
    p < 2r' 
    \quad \text{and} \quad 
    \max\{2, r_{0}\} \leq r, 
  \end{equation}
  where $r':= r/(r - 1 )$ is the H\"{\rm o}lder conjugate of $r$. 
  Then there exists a sufficiently large constant $C_{0} \geq 1$ independent of $\Omega$ such that 
  there exist $\varepsilon_{0}, \delta_{0}> 0$ independent of $\Omega$ 
  such that for any boundary data ${\bf a} \in \dot{B}_{p,q}^{\frac{2}{p} - 1}(\mathbb{R}^{2})$ with $a_{3} = 0$
  and any matrix-valued function ${\bf F} \in \widetilde{L}_{x_{3}}^{r_{0}}(0,\infty; \dot{B}_{p,q}^{\frac{2}{p} + \frac{1}{r_{0}} - 2}(\mathbb{R}^{2}))$ with
  \begin{equation}\label{0609-5}
    {\|{\bf a}\|}_{\dot{B}_{p,q}^{\frac{2}{p} - 1}},\ 
    {\|{\bf F}\|}_{\widetilde{L}_{x_{3}}^{r_{0}}(0,\infty; \dot{B}_{p,q}^{\frac{2}{p} + \frac{1}{r_{0}} - 2})} 
    \leq \delta_{0}, 
  \end{equation}
  there exists a solution ${\bf u} \in C_{x_{3}}([0,\infty) ; \dot{B}_{p,q}^{\frac{2}{p} - 1}(\mathbb{R}^{2}))\cap \widetilde{L}_{x_{3}}^{r}(0,\infty;\dot{B}_{p,q}^{\frac{2}{p} + \frac{1}{r} - 1}(\mathbb{R}^{2}))$ of \eqref{SNSC} 
  satisfying 
  \begin{equation}\label{0609-1}
    {\bf u} = P_{2^{j}> C_{0}|\Omega|^{\frac{1}{2}}}({\bf w}^{(1)}|_{\mathbb{R}_{+}^{3}} + {\bf v}^{(1)}) 
    + \left(P_{C_{0}^{-1}|\Omega|^{\frac{1}{2}} \leq 2^{j} \leq C_{0}|\Omega|^{\frac{1}{2}}} + P_{2^{j} < C_{0}^{-1}|\Omega|^{\frac{1}{2}}}\right)({\bf w}^{(2)}|_{\mathbb{R}_{+}^{3}} + {\bf v}^{(2)})
  \end{equation}
  in $C_{x_{3}}([0,\infty) ; \dot{B}_{p,q}^{\frac{2}{p} - 1}(\mathbb{R}^{2}))\cap \widetilde{L}_{x_{3}}^{r}(0,\infty;\dot{B}_{p,q}^{\frac{2}{p} + \frac{1}{r} - 1}(\mathbb{R}^{2}))$, 
  and ${\|{\bf u}\|}_{\widetilde{L}_{x_{3}}^{r}(0,\infty;\dot{B}_{p,q}^{\frac{2}{p} + \frac{1}{r} - 1})} \leq \varepsilon_{0}$. 
  Moreover, the solution is unique in the following function space, 
  \begin{equation}\label{0609-2}
    \left\{{\bf u} \in \widetilde{L}_{x_{3}}^{r}(0,\infty;\dot{B}_{p,q}^{\frac{2}{p} + \frac{1}{r} - 1}(\mathbb{R}^{2})) |\ {\|{\bf u}\|}_{\widetilde{L}_{x_{3}}^{r}(0,\infty;\dot{B}_{p,q}^{\frac{2}{p} + \frac{1}{r} - 1})} \leq \varepsilon_{0}\right\}. 
  \end{equation}
\end{thm}

Here, $C_{x_{3}}([0,\infty) ; \dot{B}_{p,q}^{\frac{2}{p} - 1}(\mathbb{R}^{2}))$ denotes the space of $\dot{B}_{p,q}^{\frac{2}{p} - 1}(\mathbb{R}^{2})$-valued functions that are continuous with respect to $x_{3}$. 

\begin{rem}
  The size requirement on the constant $C_{0} \geq 1$ appearing in the assumptions of Theorem~\ref{0609-3} is needed to justify the asymptotic expansions of the eigenvalues of the linear operator appearing in the explicit formulas for $(\phi_{j}*{\bf w}^{(2)}, \phi_{j}*{\bf v}^{(2)})$, 
  for $j\in \mathbb{Z}$ satisfying $2^{j} < C_{0}^{-1}|\Omega|^{\frac{1}{2}}$ (see Lemma~\ref{0613-7} \eqref{0613-4}). 
  The same applies throughout the remainder of the paper. 
\end{rem}

By restricting the exponents in the Chemin--Lerner spaces, we obtain the following result by taking $|\Omega|$ sufficiently large depending on any horizontal components of the boundary data and ${\bf F}$. 
\begin{thm}\label{0609-4}
  Let $1\leq p,q,r_{0}, r < \infty$ with 
  \begin{equation*}
    p < 2r' 
    \quad \text{and} \quad 
    \max\{2, r_{0}\} < r. 
  \end{equation*}
  Then there exists a sufficiently large constant $C_{0} \geq 1$ such that 
  there exists a constant $\varepsilon_{0} > 0$ such that 
  for any boundary data ${\bf a} \in \dot{B}_{p,q}^{\frac{2}{p} - 1}(\mathbb{R}^{2})$ with $a_{3} \equiv 0$ and matrix-valued function ${\bf F} \in \widetilde{L}_{x_{3}}^{r_{0}}(0,\infty; \dot{B}_{p,q}^{\frac{2}{p} + \frac{1}{r_{0}} - 2}(\mathbb{R}^{2}))$, 
  there exists a sufficiently large constant $\Omega_{0} > 0$ such that 
  for any $\Omega \in \mathbb{R}\setminus\{0\}$ with $|\Omega| \geq \Omega_{0}$, 
  there exists a solution ${\bf u} \in C_{x_{3}}([0,\infty) ; \dot{B}_{p,q}^{\frac{2}{p} - 1}(\mathbb{R}^{2}))\cap \widetilde{L}_{x_{3}}^{r}(0,\infty;\dot{B}_{p,q}^{\frac{2}{p} + \frac{1}{r} - 1}(\mathbb{R}^{2}))$ of \eqref{SNSC} 
  satisfying \eqref{0609-1}
  and ${\|{\bf u}\|}_{\widetilde{L}_{x_{3}}^{r}(0,\infty;\dot{B}_{p,q}^{\frac{2}{p} + \frac{1}{r} - 1})} \leq \varepsilon_{0}$. 
  Moreover, the solution is unique in the function space \eqref{0609-2}. 
\end{thm}

We present several remarks concerning the our main theorems. 
\begin{rem}\label{0624-5}
  \begin{enumerate}
    \item Theorem~\ref{0609-3} generalizes the result of \cite{Fu_2025} in the case $a_{3} = 0$ with respect to $\Omega$. 
    \item Within the framework of our analysis, the Coriolis term can be viewed as altering the horizontal diffusion effect induced by the Laplacian. 
    For more details on this point, see Subsection~\ref{0623-3}. 
    We also remark that, for the nonstationary problem in the whole space, 
    the diffusion due to the Laplacian and the dispersion induced by the Coriolis term can be separated at the level of the evolution operator \eqref{0621-5}. 
    \item The solution ${\bf u}$ is continuous with respect to $x_{3}$ up to the boundary 
    and attains the prescribed Dirichlet boundary data ${\bf a}$ in the strong topology of $\dot{B}_{p,q}^{\frac{2}{p} - 1}(\mathbb{R}^{2})$. 
    \item For the interpolation index $q = \infty$, we can construct a solution ${\bf u}$ of \eqref{SNSC} belonging to $\widetilde{L}_{x_{3}}^{\infty}(0,\infty ; \dot{B}_{p,\infty}^{\frac{2}{p} - 1}(\mathbb{R}^{2}))\cap \widetilde{L}_{x_{3}}^{r}(0,\infty;\dot{B}_{p,\infty}^{\frac{2}{p} + \frac{1}{r} - 1}(\mathbb{R}^{2}))$ 
    under the same assumptions and by the same argument as in Theorem~\ref{0609-3}. 
  \end{enumerate}
\end{rem}

This paper is organized as follows. 
In Section~\ref{0624-4}, we derive explicit formulas for ${\bf w}^{(1)}$, ${\bf v}^{(1)}$, ${\bf w}^{(2)}$ and ${\bf v}^{(2)}$, 
and establish their estimates. 
In addition, we establish auxiliary lemmas on estimates for frequency cutoff functions and bilinear estimates in Chemin--Lerner spaces required for the proof of the main theorem. 
Finally, in Section~\ref{0625-5}, based on the estimates established in the previous section, 
we prove the main theorems via Banach's fixed point theorem. 

\section{Preliminaries}\label{0624-4}
As discussed in the previous section, we study \eqref{SNSC} using two different decompositions of the solution, $({\bf w}^{(1)}, {\bf v}^{(1)})$ and $({\bf w}^{(2)}, {\bf v}^{(2)})$. 
This is because the estimates for the decomposition $({\bf w}^{(2)}, {\bf v}^{(2)})$ depend on the relationship between the horizontal frequency $|\xi'|$ and the magnitude of $|\Omega|^{\frac{1}{2}}$ (see Subsection~\ref{0623-3}). 
In particular, when $|\xi'|$ is sufficiently larger than $|\Omega|^{\frac{1}{2}}$, 
it is more convenient to treat the Coriolis term as a perturbation. 
This motivates us to use the decomposition $({\bf w}^{(1)}, {\bf v}^{(1)})$ in this regime. 
\subsection{Explicit formulas for ${\bf w}^{(1)}$ and ${\bf v}^{(1)}$}\label{0612-2}
In this subsection, we consider 
\begin{equation}\label{0610-5}
  \begin{cases}
  -\Delta{\bf w}^{(1)} + \nabla \pi^{(1)} = \text{div }(\bar{\bf F} - \bar{\bf u}\otimes\bar{\bf u})
    - \Omega\left(e_{3}\times\bar{\bf u}\right), 
  &\quad \  x\in\mathbb{R}^{3},\\
  \text{div } {\bf w}^{(1)} =0,
  &\quad \  x\in\mathbb{R}^{3}, 
  \end{cases}
\end{equation}
and 
\begin{equation}\label{0611-3}
  \begin{cases}
    -\Delta {\bf v}^{(1)} + \nabla q^{(1)} = 0, 
          &\quad  \   x\in \mathbb{R}_{+}^{3}, \\
    \text{div } {\bf v}^{(1)} = 0,
          &\quad  \   x\in \mathbb{R}_{+}^{3}, \\
    {\bf v}^{(1)}(x', 0) =  {\bf a}(x') - {\bf w}^{(1)}(x', 0),
          &\quad  \ x'\in{\mathbb{R}}^2, 
  \end{cases}
\end{equation}
where ${\bf u}$ and ${\bf F}$ are given functions on $\mathbb{R}_{+}^{3}$, 
and $\bar{\bf u}$ and $\bar{\bf F}$ are the extensions of ${\bf u}$ and ${\bf F}$ to $\mathbb{R}^{3}$ 
defined by \eqref{0625-6}, \eqref{0625-7} and \eqref{0625-8}. 
Also, ${\bf a}$ is given boundary data with $a_{3} = 0$. 

First, we derive an explicit formula for ${\bf w}^{(1)}$. 
In the case $\Omega = 0$, an explicit formula for ${\bf w}^{(1)}$ adapted to the framework employed in the present estimates 
was given by Fujii~\cite{Fu_2025}. 
The derivation for $\Omega \neq 0$ follows the same procedure. 
For the sake of a unified notation, 
we briefly present the derivation here. 
To simplify the notation, let 
\begin{equation*}
  {\bf G} := {\bf F} - {\bf u}\otimes{\bf u}. 
\end{equation*}
Then $\bar{\bf G} = \bar{\bf F} - \bar{\bf u}\otimes\bar{\bf u}$. 
Let 
\begin{equation}\label{0610-8}
  \mathbb{P}_{\mathbb{R}^{3}} := Id + \nabla(-\Delta)^{-1}\text{div }
\end{equation}
denote the Helmholtz projection on $\mathbb{R}^{3}$. 
Applying $\mathbb{P}_{\mathbb{R}^{3}}$ to both sides of \eqref{0610-5}, we obtain 
\begin{equation*}
  -\Delta{\bf w}^{(1)} = \mathbb{P}_{\mathbb{R}^{3}}(\text{div }\bar{\bf G}) - \Omega\mathbb{P}_{\mathbb{R}^{3}}\left(e_{3}\times\bar{\bf u}\right). 
\end{equation*}
Moreover, since $\text{div }\bar{\bf u} = 0$,  it follows that 
\begin{equation*}
  \mathbb{P}_{\mathbb{R}^{3}}\left(e_{3}\times\bar{\bf u}\right) 
  = - \partial_{3}(-\Delta)^{-1}(\nabla\times\bar{\bf u}). 
\end{equation*}
Thus, we have 
\begin{equation}\label{0610-6}
  -\Delta{\bf w}^{(1)} = \mathbb{P}_{\mathbb{R}^{3}}(\text{div }\bar{\bf G}) + \Omega\partial_{3}(-\Delta)^{-1}(\nabla\times\bar{\bf u}). 
\end{equation}
Taking the Fourier transform on $\mathbb{R}^{3}$ of both sides of \eqref{0610-6}, we get 
\begin{equation*}
  |\xi|^{2}\mathcal{F}_{\mathbb{R}^{3}}[{\bf w}^{(1)}] 
  = \mathcal{F}_{\mathbb{R}^{3}}\left[\mathbb{P}_{\mathbb{R}^{3}}(\text{div }\bar{\bf G})\right] 
  + \Omega \frac{i\xi_{3}}{|\xi|^{2}}(i\xi\times (\mathcal{F}_{\mathbb{R}^{3}}\left[\bar{\bf u}\right])). 
\end{equation*}
Hence it follows that 
\begin{equation}\label{0610-7}
  \mathcal{F}_{\mathbb{R}^{3}}[{\bf w}^{(1)}] 
  = \frac{1}{|\xi|^{2}}\mathcal{F}_{\mathbb{R}^{3}}\left[\mathbb{P}_{\mathbb{R}^{3}}(\text{div }\bar{\bf G})\right] 
  + \Omega \frac{i\xi_{3}}{|\xi|^{4}}(i\xi\times(\mathcal{F}_{\mathbb{R}^{3}}\left[\bar{\bf u}\right])). 
\end{equation}
Let $\xi = (\xi', \xi_{3})\in \mathbb{R}^{3}$, where $\xi'\in \mathbb{R}^{2}$, and let 
\begin{align*}
  \mathbb{P}_{\mathbb{R}^{3}}(\text{div }\bar{\bf G}) 
  &= ((\mathbb{P}_{\mathbb{R}^{3}}(\text{div }\bar{\bf G}))_{1}, (\mathbb{P}_{\mathbb{R}^{3}}(\text{div }\bar{\bf G}))_{2}, (\mathbb{P}_{\mathbb{R}^{3}}(\text{div }\bar{\bf G}))_{3})\\ 
  &=: ((\mathbb{P}_{\mathbb{R}^{3}}(\text{div }\bar{\bf G}))_{h}, (\mathbb{P}_{\mathbb{R}^{3}}(\text{div }\bar{\bf G}))_{3}). 
\end{align*}
We now take the inverse Fourier transform of \eqref{0610-7} separately in the tangential and normal directions. 
By the definition of $\mathbb{P}_{\mathbb{R}^{3}}$ \eqref{0610-8}, we can write 
\begin{align}\label{0612-3}
  \frac{1}{|\xi|^{2}}\left(\mathcal{F}_{\mathbb{R}^{3}}\left[\mathbb{P}_{\mathbb{R}^{3}}(\text{div }\bar{\bf G})\right]\right)_{h}
  =& \frac{1}{|\xi|^{2}}\sum_{m = 1}^{3}i\xi_{m}\mathcal{F}_{\mathbb{R}^{3}}\left[(\bar{G}_{m, 1}, \bar{G}_{m, 2})\right]\notag\\ 
  &+ \frac{i\xi'}{|\xi|^{4}}\sum_{n = 1}^{3}\sum_{m = 1}^{3}i\xi_{n}i\xi_{m}\mathcal{F}_{\mathbb{R}^{3}}\left[\bar{G}_{m, n}\right]\notag\\ 
  =& \frac{1}{|\xi'|^{2} + \xi_{3}^{2}}\sum_{m = 1}^{2}i\xi_{m}\mathcal{F}_{\mathbb{R}^{3}}\left[(\bar{G}_{m, 1}, \bar{G}_{m, 2})\right] \notag\\ 
  &+ \frac{i\xi_{3}}{|\xi'|^{2} + \xi_{3}^{2}}\mathcal{F}_{\mathbb{R}^{3}}\left[(\bar{G}_{3, 1}, \bar{G}_{3, 2})\right] \notag\\ 
  &+ \frac{i\xi'}{(|\xi'|^{2} + \xi_{3}^{2})^{2}}\sum_{m, n = 1}^{2}i\xi_{m}i\xi_{n}\mathcal{F}_{\mathbb{R}^{3}}\left[\bar{G}_{m,n}\right]\notag\\ 
  &+ \frac{i\xi'i\xi_{3}}{(|\xi'|^{2} + \xi_{3}^{2})^{2}}\Big(\sum_{m = 1}^{2}i\xi_{m}\mathcal{F}_{\mathbb{R}^{3}}\left[\bar{G}_{m,3}\right] + \sum_{n = 1}^{2}i\xi_{n}\mathcal{F}_{\mathbb{R}^{3}}\left[\bar{G}_{3,n}\right]\Big)\notag\\ 
  &+ \frac{i\xi'(i\xi_{3})^{2}}{(|\xi'|^{2} + \xi_{3}^{2})^{2}}\mathcal{F}_{\mathbb{R}^{3}}\left[\bar{G}_{3,3}\right], \notag\\ 
  \frac{1}{|\xi|^{2}}\left(\mathcal{F}_{\mathbb{R}^{3}}\left[\mathbb{P}_{\mathbb{R}^{3}}(\text{div }\bar{\bf G})\right]\right)_{3} 
  =& \frac{1}{|\xi|^{2}}\sum_{m = 1}^{3}i\xi_{m}\mathcal{F}_{\mathbb{R}^{3}}\left[\bar{G}_{m, 3}\right] 
  + \frac{i\xi_{3}}{|\xi|^{4}}\sum_{n = 1}^{3}\sum_{m = 1}^{3}i\xi_{n}i\xi_{m}\mathcal{F}_{\mathbb{R}^{3}}\left[\bar{G}_{m, n}\right]\notag\\ 
  =& \frac{|\xi'|^{2}}{(|\xi'|^{2} + \xi_{3}^{2})^{2}}\sum_{m = 1}^{2}i\xi_{m}\mathcal{F}_{\mathbb{R}^{3}}\left[\bar{G}_{m, 3}\right] 
  + \frac{i\xi_{3}|\xi'|^{2}}{(|\xi'|^{2} + \xi_{3}^{2})^{2}}\mathcal{F}_{\mathbb{R}^{3}}\left[\bar{G}_{3,3}\right]\notag\\ 
  &+ \frac{i\xi_{3}}{(|\xi'|^{2} + \xi_{3}^{2})^{2}}\sum_{m,n = 1}^{2}i\xi_{m}i\xi_{n}\mathcal{F}_{\mathbb{R}^{3}}\left[\bar{G}_{m,n}\right]\notag\\ 
  &+  \frac{(i\xi_{3})^{2}}{(|\xi'|^{2} + \xi_{3}^{2})^{2}}\sum_{n = 1}^{2}i\xi_{n}\mathcal{F}_{\mathbb{R}^{3}}\left[\bar{G}_{3,n}\right]. 
\end{align}
Similarly, it follows that 
\begin{align*}
  \frac{i\xi_{3}}{|\xi|^{4}}\left(\mathcal{F}_{\mathbb{R}^{3}}\left[\nabla\times\bar{\bf u}\right]\right)_{h}
  &= \frac{(i\xi_{3})^{2}}{(|\xi'|^{2} + \xi_{3}^{2})^{2}}\mathcal{F}_{\mathbb{R}^{3}}[\bar{\bf u}_{h}^{\perp}] 
  - \frac{i\xi_{3}}{(|\xi'|^{2} + \xi_{3}^{2})^{2}}(i\xi')^{\perp}\mathcal{F}_{\mathbb{R}^{3}}[\bar{u}_{3}],\\ 
  \frac{i\xi_{3}}{|\xi|^{4}}\left(\mathcal{F}_{\mathbb{R}^{3}}\left[\nabla\times\bar{\bf u}\right]\right)_{3} 
  &= \frac{i\xi_{3}}{(|\xi'|^{2} + \xi_{3}^{2})^{2}}(i\xi')^{\perp}\cdot(\mathcal{F}_{\mathbb{R}^{3}}[\bar{\bf u}_{h}]), 
\end{align*}
where ${\bf u}_{h} := (u_{1}, u_{2})$, ${\bf u}_{h}^{\perp} := (- u_{2}, u_{1})$ and $(\xi')^{\perp} := (-\xi_{2}, \xi_{1})$. 
We observe that the following holds. 
\begin{align}\label{0611-1}
  \mathcal{F}_{\mathbb{R}_{\xi_{3}}}^{-1}\left[\frac{1}{|\xi'|^{2} + \xi_{3}^{2}}\right](\xi', x_{3}) 
  &= \left(\frac{\pi}{2}\right)^{\frac{1}{2}}\frac{1}{|\xi'|}e^{-|\xi'||x_{3}|} 
  =: K_{1}^{(1)}(\xi', x_{3}),\notag\\ 
  \mathcal{F}_{\mathbb{R}_{\xi_{3}}}^{-1}\left[\frac{i\xi_{3}}{|\xi'|^{2} + \xi_{3}^{2}}\right](\xi', x_{3}) 
  &= \partial_{3}\mathcal{F}_{\mathbb{R}_{\xi_{3}}}^{-1}\left[\frac{1}{|\xi'|^{2} + \xi_{3}^{2}}\right](\xi', x_{3})\notag\\ 
  &= - \left(\frac{\pi}{2}\right)^{\frac{1}{2}}\text{sgn}(x_{3})e^{-|\xi'||x_{3}|}
  =: K_{2}^{(1)}(\xi', x_{3}),\notag\\ 
  \mathcal{F}_{\mathbb{R}_{\xi_{3}}}^{-1}\left[\frac{1}{(|\xi'|^{2} + \xi_{3}^{2})^{2}}\right](\xi', x_{3}) 
  &= \left(\frac{\pi}{2}\right)^{\frac{1}{2}}\frac{1}{2|\xi'|^{3}}(1 + |\xi'||x_{3}|)e^{-|\xi'||x_{3}|}
  =: K_{3}^{(1)}(\xi', x_{3}),\notag\\ 
  \mathcal{F}_{\mathbb{R}_{\xi_{3}}}^{-1}\left[\frac{i\xi_{3}}{(|\xi'|^{2} + \xi_{3}^{2})^{2}}\right](\xi', x_{3}) 
  &= - \left(\frac{\pi}{2}\right)^{\frac{1}{2}}\frac{x_{3}}{2|\xi'|}e^{-|\xi'||x_{3}|}
  =: K_{4}^{(1)}(\xi', x_{3}),\notag\\ 
  \mathcal{F}_{\mathbb{R}_{\xi_{3}}}^{-1}\left[\frac{(i\xi_{3})^{2}}{(|\xi'|^{2} + \xi_{3}^{2})^{2}}\right](\xi', x_{3}) 
  &= \left(\frac{\pi}{2}\right)^{\frac{1}{2}}\frac{1}{2|\xi'|}(|\xi'||x_{3}| - 1)e^{-|\xi'||x_{3}|}
  =: K_{5}^{(1)}(\xi', x_{3}). 
\end{align}
Therefore, we obtain 
\begin{align*}
  \widehat{\bf w}_{h}^{(1)}(\xi', x_{3}) 
  =& \sum_{m = 1}^{2}i\xi_{m}\int_{-\infty}^{+\infty}K_{1}^{(1)}(\xi', x_{3} - y_{3})(\widehat{\bar{G}}_{m, 1}, \widehat{\bar{G}}_{m, 2})(\xi', y_{3})~{\rm d}y_{3}\\ 
  &+ \int_{-\infty}^{+\infty}K_{2}^{(1)}(\xi', x_{3} - y_{3})(\widehat{\bar{G}}_{3, 1}, \widehat{\bar{G}}_{3, 2})(\xi', y_{3})~{\rm d}y_{3}\\ 
  &+ i\xi'\sum_{m,n = 1}^{2}i\xi_{m}i\xi_{n}\int_{-\infty}^{+\infty}K_{3}^{(1)}(\xi', x_{3} - y_{3})\widehat{\bar{G}}_{m,n}(\xi', y_{3})~{\rm d}y_{3}\\ 
  &+ i\xi'\sum_{m = 1}^{2}i\xi_{m}\int_{-\infty}^{+\infty}K_{4}^{(1)}(\xi', x_{3} - y_{3})(\widehat{\bar{G}}_{m, 3} + \widehat{\bar{G}}_{3,m})(\xi', y_{3})~{\rm d}y_{3}\\ 
  &+ i\xi'\int_{-\infty}^{+\infty}K_{5}^{(1)}(\xi', x_{3} - y_{3})\widehat{\bar{G}}_{3,3}(\xi', y_{3})~{\rm d}y_{3}\\ 
  &+ \Omega\int_{-\infty}^{+\infty}K_{5}^{(1)}(\xi', x_{3} - y_{3})\widehat{\bar{\bf u}}_{h}^{\perp}(\xi', y_{3})~{\rm d}y_{3}\\ 
  &- \Omega (i\xi')^{\perp}\int_{-\infty}^{+\infty}K_{4}^{(1)}(\xi', x_{3} - y_{3})\widehat{\bar{u}}_{3}(\xi', y_{3})~{\rm d}y_{3}. 
\end{align*}
Here, we recall that $\widehat{f} = \mathcal{F}_{\mathbb{R}^{2}}[f]$. 
Thanks to 
\begin{equation*}
  \widehat{\bar{\bf u}} = \bar{\widehat{\bf u}} 
  \quad 
  \text{and}
  \quad 
  \widehat{\bar{\bf F}} = \bar{\widehat{\bf F}}, 
\end{equation*}
and the definitions of $\bar{\bf u}$ and $\bar{\bf F}$ (\eqref{0625-6}, \eqref{0625-7} and \eqref{0625-8}), 
we can write 
\begin{align*}
  &\widehat{\bf w}_{h}^{(1)}(\xi', x_{3})\\ 
  =& \sum_{m = 1}^{2}i\xi_{m}\int_{0}^{+\infty}\left(K_{1}^{(1)}(\xi', x_{3} - y_{3}) + K_{1}^{(1)}(\xi', x_{3} + y_{3})\right)(\widehat{G}_{m, 1}, \widehat{G}_{m, 2})(\xi', y_{3})~{\rm d}y_{3}\\ 
  &+ \int_{0}^{+\infty}\left(K_{2}^{(1)}(\xi', x_{3} - y_{3}) - K_{2}^{(1)}(\xi', x_{3} + y_{3})\right)(\widehat{G}_{3, 1}, \widehat{G}_{3, 2})(\xi', y_{3})~{\rm d}y_{3}\\ 
  &+ i\xi'\sum_{m,n = 1}^{2}i\xi_{m}i\xi_{n}\int_{0}^{+\infty}\left(K_{3}^{(1)}(\xi', x_{3} - y_{3}) + K_{3}^{(1)}(\xi', x_{3} + y_{3})\right)\widehat{G}_{m,n}(\xi', y_{3})~{\rm d}y_{3}\\ 
  &+ i\xi'\sum_{m = 1}^{2}i\xi_{m}\int_{0}^{+\infty}\left(K_{4}^{(1)}(\xi', x_{3} - y_{3}) - K_{4}^{(1)}(\xi', x_{3} + y_{3})\right)(\widehat{G}_{m,3} + \widehat{G}_{3,m})(\xi', y_{3})~{\rm d}y_{3}\\ 
  &+ i\xi'\int_{0}^{+\infty}\left(K_{5}^{(1)}(\xi', x_{3} - y_{3}) + K_{5}^{(1)}(\xi', x_{3} + y_{3})\right)\widehat{G}_{3,3}(\xi', y_{3})~{\rm d}y_{3}\\ 
  &+ \Omega\int_{0}^{+\infty}\left(K_{5}^{(1)}(\xi', x_{3} - y_{3}) + K_{5}^{(1)}(\xi', x_{3} + y_{3})\right)\widehat{\bf u}_{h}^{\perp}(\xi', y_{3})~{\rm d}y_{3}\\ 
  &- \Omega (i\xi')^{\perp}\int_{0}^{+\infty}\left(K_{4}^{(1)}(\xi', x_{3} - y_{3}) - K_{4}^{(1)}(\xi', x_{3} + y_{3})\right)\widehat{u}_{3}(\xi', y_{3})~{\rm d}y_{3}. 
\end{align*}
For a function $f = f(x',x_{3}):\mathbb{R}_{+}^{3}\to \mathbb{R}$ that is locally integrable with respect to $x_{3}$, we define 
\begin{align}\label{0611-2}
  &T_{K_{l}^{(1)}}^{(\pm)}[f](x', x_{3})\notag\\ 
  :=& \frac{1}{\sqrt{2\pi}}\int_{0}^{+\infty}\mathcal{F}_{\mathbb{R}^{2}}^{-1}\left[\left(K_{l}^{(1)}(\xi', x_{3} - y_{3}) \pm K_{l}^{(1)}(\xi', x_{3} + y_{3})\right)\widehat{f}(\xi', y_{3})\right](x')~{\rm d}y_{3}, 
\end{align}
for $l = 1, 2, 3, 4, 5$. 
Therefore, we obtain 
\begin{align*}
  {\bf w}_{h}^{(1)}(x', x_{3}) 
  =& \sum_{m = 1}^{2}\partial_{m}T_{K_{1}^{(1)}}^{(+)}\left[(G_{m, 1}, G_{m, 2})\right](x', x_{3}) 
  + T_{K_{2}^{(1)}}^{(-)}\left[(G_{3, 1}, G_{3, 2})\right](x', x_{3})\\ 
  &+ \nabla'\Big(\sum_{m,n = 1}^{2}\partial_{m}\partial_{n}T_{K_{3}^{(1)}}^{(+)}\left[G_{m,n}\right]\Big)(x', x_{3})\\ 
  &+ \nabla'\Big(\sum_{m = 1}^{2}\partial_{m}T_{K_{4}^{(1)}}^{(-)}\left[G_{m, 3} + G_{3, m}\right]\Big)(x', x_{3}) 
  + \nabla'\big(T_{K_{5}^{(1)}}^{(+)}\left[G_{3,3}\right]\big)(x', x_{3})\\ 
  &+ \Omega T_{K_{5}^{(1)}}^{(+)}\left[{\bf u}_{h}^{\perp}\right](x', x_{3})
  - \Omega (\nabla')^{\perp}\big(T_{K_{4}^{(1)}}^{(-)}\left[u_{3}\right]\big)(x', x_{3}), 
\end{align*}
where $\nabla' := (\partial_{1}, \partial_{2})$ and $(\nabla')^{\perp} := (-\partial_{2}, \partial_{1})$. 
By a similar argument, it follows that 
\begin{align*}
  w_{3}^{(1)}(x', x_{3}) 
  =& -\Delta_{\mathbb{R}^{2}}\Big(\sum_{m = 1}^{2}\partial_{m}T_{K_{3}^{(1)}}^{(-)}\left[G_{m,3}\right]\Big)(x', x_{3}) 
  -\Delta_{\mathbb{R}^{2}}\big(T_{K_{4}^{(1)}}^{(+)}\left[G_{3,3}\right]\big)(x', x_{3})\\ 
  &+ \sum_{m,n = 1}^{2}\partial_{m}\partial_{n}T_{K_{4}^{(1)}}^{(+)}\left[G_{m,n}\right](x', x_{3}) 
  + \sum_{m = 1}^{2}\partial_{m}T_{K_{5}^{(1)}}^{(-)}\left[G_{3,m}\right](x', x_{3})\\ 
  &+ \Omega(\nabla')^{\perp}\cdot\big(T_{K_{4}^{(1)}}^{(+)}\left[{\bf u}_{h}\right]\big)(x', x_{3}), 
\end{align*}
where $\Delta_{\mathbb{R}^{2}} := \partial_{1}^{2} + \partial_{2}^{2}$. 
By expanding the notation ${\bf G} = {\bf F} - {\bf u}\otimes{\bf u}$, we obtain the following proposition. 
\begin{prop}
  Let $\Omega \in \mathbb{R}$. 
  Then the solution ${\bf w}^{(1)} = (w_{1}^{(1)}, w_{2}^{(1)}, w_{3}^{(1)}) := ({\bf w}_{h}^{(1)}, w_{3}^{(1)})$ of \eqref{0610-5} is given by 
\begin{align}\label{0611-4}
  {\bf w}_{h}^{(1)}(x', x_{3}) 
  =& \sum_{m = 1}^{2}\partial_{m}T_{K_{1}^{(1)}}^{(+)}\left[(F_{m, 1}, F_{m, 2}) - u_{m}{\bf u}_{h}\right](x', x_{3}) 
  + T_{K_{2}^{(1)}}^{(-)}\left[(F_{3, 1}, F_{3, 2}) - u_{3}{\bf u}_{h}\right](x', x_{3})\notag\\ 
  &+ \nabla'\Big(\sum_{m,n = 1}^{2}\partial_{m}\partial_{n}T_{K_{3}^{(1)}}^{(+)}\left[F_{m,n} - u_{m}u_{n}\right]\Big)(x', x_{3})\notag\\ 
  &+ \nabla'\Big(\sum_{m = 1}^{2}\partial_{m}T_{K_{4}^{(1)}}^{(-)}\left[F_{m,3} + F_{3, m} - 2u_{m}u_{3}\right]\Big)(x', x_{3})\notag\\ 
  &+ \nabla'\big(T_{K_{5}^{(1)}}^{(+)}\left[F_{3,3} - u_{3}u_{3}\right]\big)(x', x_{3})\notag\\ 
  &+ \Omega T_{K_{5}^{(1)}}^{(+)}\left[{\bf u}_{h}^{\perp}\right](x', x_{3})
  - \Omega (\nabla')^{\perp}\big(T_{K_{4}^{(1)}}^{(-)}\left[u_{3}\right]\big)(x', x_{3}),\notag\\ 
  w_{3}^{(1)}(x', x_{3}) 
  =& -\Delta_{\mathbb{R}^{2}}\Big(\sum_{m = 1}^{2}\partial_{m}T_{K_{3}^{(1)}}^{(-)}\left[F_{m,3} - u_{m}u_{3}\right]\Big)(x', x_{3})\notag\\ 
  &-\Delta_{\mathbb{R}^{2}}\Big(T_{K_{4}^{(1)}}^{(+)}\left[F_{3,3} - u_{3}u_{3}\right]\Big)(x', x_{3})\notag\\ 
  &+ \sum_{m,n = 1}^{2}\partial_{m}\partial_{n}T_{K_{4}^{(1)}}^{(+)}\left[F_{m,n} - u_{m}u_{n}\right](x', x_{3})\notag\\ 
  &+ \sum_{m = 1}^{2}\partial_{m}T_{K_{5}^{(1)}}^{(-)}\left[F_{3,m} - u_{3}u_{m}\right](x', x_{3})\notag\\ 
  &+ \Omega(\nabla')^{\perp}\cdot\big(T_{K_{4}^{(1)}}^{(+)}\left[{\bf u}_{h}\right]\big)(x', x_{3}), 
\end{align}
where ${\bf u}_{h} := (u_{1}, u_{2})$, and $K_{l}^{(1)}$ and $T_{K_{l}^{(1)}}^{(\pm)}$ ($l = 1, 2, 3, 4, 5$) are given by \eqref{0611-1} and \eqref{0611-2}, respectively. 
\end{prop}

\begin{rem}
  From the definition of $K_{l}^{(1)}$ \eqref{0611-1} and $T_{K_{l}^{(1)}}^{(\pm)}$ \eqref{0611-2} ($l = 1,2,3,4,5$), 
  we readily check that ${\bf w}_{h}^{(1)}(x', x_{3})$ is even in $x_{3}$ and $w_{3}^{(1)}(x', x_{3})$ is odd in $x_{3}$. 
  Namely, ${\bf w}^{(1)}$ preserves the symmetries of ${\bf u}$ and ${\bf F}$, and in particular $w_{3}^{(1)}(x', 0) = 0$ holds. 
\end{rem}

Next, we consider an explicit formula for the solution ${\bf v}^{(1)}$ of \eqref{0611-3}. 
Since $\Omega$ does not appear in \eqref{0611-3}, the explicit formula for ${\bf v}^{(1)}$ is the same as that in \cite{Fu_2025}. 
Therefore, we omit the derivation and state only the resulting formula. 
Note that 
\begin{equation*}
  a_{3}(\cdot) - w_{3}^{(1)}(\cdot, 0) = 0 
\end{equation*}
holds. 

\begin{prop}[\cite{Fu_2025}]
  The solution ${\bf v}^{(1)} = (v_{1}^{(1)}, v_{2}^{(1)}, v_{3}^{(1)}) =: ({\bf v}_{h}^{(1)}, v_{3}^{(1)})$ of \eqref{0611-3} is given by 
  \begin{align}\label{0611-5}
    {\bf v}_{h}^{(1)}(x', x_{3}) 
    =& \left(Id + x_{3}\nabla'(-\Delta_{\mathbb{R}^{2}})^{-\frac{1}{2}}\nabla'\cdot\right)\mathcal{F}_{\mathbb{R}^{2}}^{-1}\left[e^{-x_{3}|\xi'|}\left(\widehat{\bf a}_{h}(\xi') - (\widehat{\bf w}^{(1)})_{h}(\xi', 0)\right)\right](x'), \notag\\ 
    v_{3}^{(1)}(x', x_{3}) 
    =& - x_{3}\nabla'\cdot\left(\mathcal{F}_{\mathbb{R}^{2}}^{-1}\left[e^{-x_{3}|\xi'|}\left(\widehat{\bf a}_{h}(\xi') - (\widehat{\bf w}^{(1)})_{h}(\xi', 0)\right)\right]\right)(x'), 
  \end{align}
  where $(-\Delta_{\mathbb{R}^{2}})^{\frac{1}{2}}f := \mathcal{F}_{\mathbb{R}^{2}}^{-1}[|\xi'|\widehat{f}(\xi')]$. 
\end{prop}

\subsection{Estimates for ${\bf w}^{(1)}$ and ${\bf v}^{(1)}$}\label{1029-1}
In this subsection, we estimate ${\bf w}^{(1)}$ and ${\bf v}^{(1)}$ in $\widetilde{L}_{x_{3}}^{r}(0,\infty; \dot{B}_{p,q}^{s}(\mathbb{R}^{2}))$ 
using the explicit formulas \eqref{0611-4} and \eqref{0611-5} derived in the previous subsection. 

We first state some basic lemmas concerning frequency-localized functions. 
\begin{lem}[e.g., \cite{Ba_Ch_Da_2011}]\label{0611-6}
  Let $1\leq p \leq \infty$. 
  Then there exists a constant $C>0$ such that for any $j\in \mathbb{Z}$ and $f$ with $\phi_j*f \in L^p$, we have 
  \begin{equation*}
    {\|\nabla(\phi_j*f)\|}_{L^p}\leq C 2^{j}{\|\phi_{j}*f\|}_{L^p} 
    \text{ and } 
    {\|\Lambda^{-1}(\phi_j*f)\|}\leq C 2^{-j}{\|\phi_{j}*f\|}_{L^p}. 
  \end{equation*}
\end{lem}

Since $e^{-x_{3}|\xi'|}$ is the Fourier multiplier associated with the Poisson semigroup, the following lemma holds. 
\begin{lem}[\cite{I_2015}]\label{0611-10}
  Let $x_{3} > 0$ and $1\leq p\leq \infty$. 
  Then there exist constants $C,c>0$ such that for any $j\in\mathbb{Z}$ and $f$ with $\phi_j*f\in L^{p}(\mathbb{R}^{2})$, we have 
  \begin{equation*}
    {\left\|\mathcal{F}_{\mathbb{R}^{2}}\left[e^{-x_{3}|\xi'|}\widehat{\phi}_{j}(\xi')\widehat{f}(\xi')\right]\right\|}_{L^{p}}
    \leq Ce^{-cx_{3}2^{j}}{\|\phi_j*f\|}_{L^{p}}. 
  \end{equation*}
\end{lem}

The following result follows immediately from these lemmas and the explicit formula for the multiplier $K_{l}^{(1)}$ $(l = 1,2,3,4,5)$ given by \eqref{0611-1}. 
\begin{lem}\label{0611-7}
  Let $x_{3} > 0$ and $1\leq p \leq \infty$. 
  Then there exist constants $C,c>0$ such that for any $j\in\mathbb{Z}$ and $f$ with $\phi_j*f\in L^{p}(\mathbb{R}^{2})$, we have 
  \begin{align*}
    {\left\|\mathcal{F}_{\mathbb{R}^{2}}^{-1}\left[\widehat{\phi}_{j}(\xi')K_{1}^{(1)}(\xi', x_{3})\widehat{f}(\xi')\right]\right\|}_{L^{p}} 
    \leq& C 2^{-j}e^{-cx_{3}2^{j}}{\|\phi_j*f\|}_{L^{p}},\\ 
    {\left\|\mathcal{F}_{\mathbb{R}^{2}}^{-1}\left[\widehat{\phi}_{j}(\xi')K_{2}^{(1)}(\xi', x_{3})\widehat{f}(\xi')\right]\right\|}_{L^{p}} 
    \leq& C e^{-cx_{3}2^{j}}{\|\phi_j*f\|}_{L^{p}},\\ 
    {\left\|\mathcal{F}_{\mathbb{R}^{2}}^{-1}\left[\widehat{\phi}_{j}(\xi')K_{3}^{(1)}(\xi', x_{3})\widehat{f}(\xi')\right]\right\|}_{L^{p}} 
    \leq& C 2^{-3j}(1 + x_{3}2^{j})e^{-cx_{3}2^{j}}{\|\phi_j*f\|}_{L^{p}},\\ 
    {\left\|\mathcal{F}_{\mathbb{R}^{2}}^{-1}\left[\widehat{\phi}_{j}(\xi')K_{4}^{(1)}(\xi', x_{3})\widehat{f}(\xi')\right]\right\|}_{L^{p}} 
    \leq& C x_{3}2^{-j}e^{-cx_{3}2^{j}}{\|\phi_j*f\|}_{L^{p}},\\ 
    {\left\|\mathcal{F}_{\mathbb{R}^{2}}^{-1}\left[\widehat{\phi}_{j}(\xi')K_{5}^{(1)}(\xi', x_{3})\widehat{f}(\xi')\right]\right\|}_{L^{p}} 
    \leq& C 2^{-j}(1 + x_{3}2^{j})e^{-cx_{3}2^{j}}{\|\phi_j*f\|}_{L^{p}}. 
  \end{align*}
\end{lem}

By the following proposition, we obtain an estimate for ${\bf w}^{(1)}$. 
\begin{prop}\label{0611-11}
  Let $\Omega \in \mathbb{R}$, $s\in \mathbb{R}$ and $1\leq p,q,r,r_{0},r_{1} \leq \infty$ with $2 \leq r_{1}$ and $\max\{r_{0}, r_{1}\}\leq r$. 
  Let ${\bf u}$ and ${\bf F}$ be smooth functions on $\mathbb{R}_{+}^{3}$. 
  Then there exists a constant $C>0$ such that for ${\bf w}^{(1)}$ given by \eqref{0611-4}, we have 
  \begin{align}\label{0611-8}
    {\|{\bf w}^{(1)}\|}_{\widetilde{L}_{x_{3}}^{r}(0, \infty ; \dot{B}_{p,q}^{s + \frac{1}{r}})} 
    \leq& C {\|{\bf F}\|}_{\widetilde{L}_{x_{3}}^{r_{0}}(0, \infty ; \dot{B}_{p,q}^{s + \frac{1}{r_{0}} - 1})} 
    + C {\|{\bf u}\otimes{\bf u}\|}_{\widetilde{L}_{x_{3}}^{\frac{r_{1}}{2}}(0, \infty ; \dot{B}_{p,q}^{s + \frac{2}{r_{1}} - 1})}\notag\\ 
    &+ C |\Omega|{\|{\bf u}\|}_{\widetilde{L}_{x_{3}}^{r_{1}}(0, \infty ; \dot{B}_{p,q}^{s + \frac{1}{r_{1}} - 2})}. 
  \end{align}
\end{prop}
\begin{proof}
  Using Lemma~\ref{0611-6}, for each $j\in \mathbb{Z}$, we get 
  \begin{align*}
    {\|\phi_{j}*{\bf w}^{(1)}\|}_{L^{p}} 
    \leq& C 2^{j}{\left\|\phi_{j}*T_{K_{1}^{(1)}}^{(+)}\left[{\bf F} - {\bf u}\otimes{\bf u}\right]\right\|}_{L^{p}} 
    + C{\left\|\phi_{j}*T_{K_{2}^{(1)}}^{(-)}\left[{\bf F} - {\bf u}\otimes{\bf u}\right]\right\|}_{L^{p}}\\ 
    &+ C2^{3j}{\left\|\phi_{j}*T_{K_{3}^{(1)}}^{(\pm)}\left[{\bf F} - {\bf u}\otimes{\bf u}\right]\right\|}_{L^{p}} 
    + C2^{2j}{\left\|\phi_{j}*T_{K_{4}^{(1)}}^{(\pm)}\left[{\bf F} - {\bf u}\otimes{\bf u}\right]\right\|}_{L^{p}}\\ 
    &+ C2^{j}{\left\|\phi_{j}*T_{K_{5}^{(1)}}^{(\pm)}\left[{\bf F} - {\bf u}\otimes{\bf u}\right]\right\|}_{L^{p}} 
    + C|\Omega|2^{j}{\left\|\phi_{j}*T_{K_{4}^{(1)}}^{(\pm)}\left[{\bf u}\right]\right\|}_{L^{p}}\\ 
    &+ C|\Omega|{\left\|\phi_{j}*T_{K_{5}^{(1)}}^{(+)}\left[{\bf u}\right]\right\|}_{L^{p}}. 
  \end{align*}
  By the definition of $T_{K_{l}^{(1)}}^{(\pm)}$ $(l = 1,2,3,4,5)$ \eqref{0611-2}, it follows that 
  \begin{align}\label{0616-2}
    {\left\|\phi_{j}*T_{K_{l}^{(1)}}^{(\pm)}[f]\right\|}_{L^{p}}
    \leq& \int_{0}^{+\infty}{\left\|\mathcal{F}_{\mathbb{R}^{2}}^{-1}\left[\widehat{\phi}_{j}(\xi')K_{l}^{(1)}(\xi', x_{3} - y_{3})\widehat{f}(\xi', y_{3})\right]\right\|}_{L^{p}}~{\rm d}y_{3}\notag\\ 
    &+ \int_{0}^{+\infty}{\left\|\mathcal{F}_{\mathbb{R}^{2}}^{-1}\left[\widehat{\phi}_{j}(\xi')K_{l}^{(1)}(\xi', x_{3} + y_{3})\widehat{f}(\xi', y_{3})\right]\right\|}_{L^{p}}~{\rm d}y_{3}. 
  \end{align}
  Thus, from Lemma~\eqref{0611-7}, we have 
  \begin{align*}
    {\left\|\phi_{j}*T_{K_{1}^{(1)}}^{(\pm)}[f]\right\|}_{L^{p}}
    \leq& C2^{-j}\int_{0}^{+\infty}(e^{-c|x_{3} - y_{3}|2^{j}} + e^{-c(x_{3} + y_{3})2^{j}}){\|\phi_{j}*f(\cdot, y_{3})\|}_{L^{p}}~{\rm d}y_{3}\\ 
    \leq& C2^{-j}\int_{0}^{+\infty}e^{-c|x_{3} - y_{3}|2^{j}}{\|\phi_{j}*f(\cdot, y_{3})\|}_{L^{p}}~{\rm d}y_{3}, \\ 
    {\left\|\phi_{j}*T_{K_{2}^{(1)}}^{(\pm)}[f]\right\|}_{L^{p}}
    \leq& C\int_{0}^{+\infty}e^{-c|x_{3} - y_{3}|2^{j}}{\|\phi_{j}*f(\cdot, y_{3})\|}_{L^{p}}~{\rm d}y_{3}, \\ 
    {\left\|\phi_{j}*T_{K_{3}^{(1)}}^{(\pm)}[f]\right\|}_{L^{p}}
    \leq& C2^{-3j}\int_{0}^{+\infty}(1 + |x_{3} - y_{3}|2^{j})e^{-c|x_{3} - y_{3}|2^{j}}{\|\phi_{j}*f(\cdot, y_{3})\|}_{L^{p}}~{\rm d}y_{3}\\ 
    &+ C2^{-3j}\int_{0}^{+\infty}(1 + (x_{3} + y_{3})2^{j})e^{-c(x_{3} + y_{3})2^{j}}{\|\phi_{j}*f(\cdot, y_{3})\|}_{L^{p}}~{\rm d}y_{3}\\ 
    \leq& C2^{-3j}\int_{0}^{+\infty}e^{-c|x_{3} - y_{3}|2^{j}}{\|\phi_{j}*f(\cdot, y_{3})\|}_{L^{p}}~{\rm d}y_{3}, \\ 
    {\left\|\phi_{j}*T_{K_{4}^{(1)}}^{(\pm)}[f]\right\|}_{L^{p}}
    \leq& C2^{-2j}\int_{0}^{+\infty}e^{-c|x_{3} - y_{3}|2^{j}}{\|\phi_{j}*f(\cdot, y_{3})\|}_{L^{p}}~{\rm d}y_{3}, \\ 
    {\left\|\phi_{j}*T_{K_{5}^{(1)}}^{(\pm)}[f]\right\|}_{L^{p}}
    \leq& C2^{-j}\int_{0}^{+\infty}e^{-c|x_{3} - y_{3}|2^{j}}{\|\phi_{j}*f(\cdot, y_{3})\|}_{L^{p}}~{\rm d}y_{3}. 
  \end{align*}
  Here, we use 
  \begin{equation*}
    x_{3}2^{j}e^{-x_{3}2^{j}} 
    \leq Ce^{-cx_{3}2^{j}}. 
  \end{equation*}
  Therefore, we obtain 
  \begin{align*}
    {\|\phi_{j}*{\bf w}^{(1)}\|}_{L^{p}} 
    \leq& C \int_{0}^{+\infty}e^{-c|x_{3} - y_{3}|2^{j}}\left({\left\|\phi_{j}*{\bf F}(\cdot, y_{3})\right\|}_{L^{p}} + {\left\|\phi_{j}*({\bf u}\otimes{\bf u})(\cdot, y_{3})\right\|}_{L^{p}}\right)~{\rm d}y_{3}\\ 
    &+ C|\Omega|2^{-j} \int_{0}^{+\infty}e^{-c|x_{3} - y_{3}|2^{j}}{\left\|\phi_{j}*{\bf u}(\cdot, y_{3})\right\|}_{L^{p}}~{\rm d}y_{3}. 
  \end{align*}
  By Young's inequality (note that $\max\{r_{0}, r_{1}\} \leq r$), we have 
  \begin{align}\label{0611-9}
    {\|\phi_{j}*{\bf w}^{(1)}\|}_{L^{r}(0, \infty; L^{p})} 
    \leq& C 2^{\left(-1 - \frac{1}{r} + \frac{1}{r_{0}}\right)j}{\left\|\phi_{j}*{\bf F}\right\|}_{L^{r_{0}}(0, \infty; L^{p})}\notag\\ 
    &+ C 2^{\left(-1 - \frac{1}{r} + \frac{2}{r_{1}}\right)j}{\left\|\phi_{j}*({\bf u}\otimes{\bf u})\right\|}_{L^{\frac{r_{1}}{2}}(0, \infty; L^{p})}\notag\\ 
    &+ C|\Omega|2^{\left(-2 - \frac{1}{r} + \frac{1}{r_{1}}\right)j}{\left\|\phi_{j}*{\bf u}\right\|}_{L^{r_{1}}(0, \infty; L^{p})}. 
  \end{align}
  Multiplying $2^{(s + \frac{1}{r})j}$ to \eqref{0611-9} and taking the $l^{q}(\mathbb{Z})$ norm with respect to $j$, 
  we get \eqref{0611-8}. 
\end{proof}

By the same argument, we can also derive an estimate for ${\bf v}^{(1)}$. 
\begin{prop}\label{0611-13}
  Let $\Omega \in \mathbb{R}$, $s\in \mathbb{R}$ and $1\leq p,q,r,r_{0},r_{1} \leq \infty$ with $2\leq r_{1}$. 
  Let ${\bf u}$ and ${\bf F}$ be smooth functions on $\mathbb{R}_{+}^{3}$ and ${\bf a}$ be a smooth function on $\mathbb{R}^{2}$ with $a_{3} = 0$. 
  Then there exists a constant $C>0$ such that for ${\bf v}^{(1)}$ given by \eqref{0611-5}, we have 
  \begin{align*}
    {\|{\bf v}^{(1)}\|}_{\widetilde{L}_{x_{3}}^{r}(0, \infty ; \dot{B}_{p,q}^{s + \frac{1}{r}})} 
    \leq& C {\|{\bf a}\|}_{\dot{B}_{p,q}^{s}}
    + C {\|{\bf F}\|}_{\widetilde{L}_{x_{3}}^{r_{0}}(0, \infty ; \dot{B}_{p,q}^{s + \frac{1}{r_{0}} - 1})} 
    + C {\|{\bf u}\otimes{\bf u}\|}_{\widetilde{L}_{x_{3}}^{\frac{r_{1}}{2}}(0, \infty ; \dot{B}_{p,q}^{s + \frac{2}{r_{1}} - 1})}\notag\\ 
    &+ C |\Omega|{\|{\bf u}\|}_{\widetilde{L}_{x_{3}}^{r_{1}}(0, \infty ; \dot{B}_{p,q}^{s + \frac{1}{r_{1}} - 2})}. 
  \end{align*}
\end{prop}
\begin{proof}
  Thanks to Lemma~\eqref{0611-10}, for each $j\in \mathbb{Z}$, we have 
  \begin{equation}\label{1023-1}
    {\|\phi_{j}*{\bf v}^{(1)}\|}_{L^{r}(0, \infty; L^{p})} 
    \leq C 2^{-\frac{1}{r}j}{\|\phi_j*{\bf a}\|}_{L^{p}} 
    + 2^{-\frac{1}{r}j}{\|\phi_{j}*{\bf w}^{(1)}(\cdot, 0)\|}_{L^{p}}. 
  \end{equation}
  We estimate ${\bf w}^{(1)}(\cdot, 0)$. 
  As in the proof of Proposition~\ref{0611-11}, by Lemma~\ref{0611-7}, it follows that 
  \begin{align*}
    {\|\phi_{j}*{\bf w}^{(1)}(\cdot, 0)\|}_{L^{p}} 
    \leq& C \int_{0}^{+\infty}e^{-cy_{3}2^{j}}\left({\left\|\phi_{j}*{\bf F}(\cdot, y_{3})\right\|}_{L^{p}} + {\left\|\phi_{j}*({\bf u}\otimes{\bf u})(\cdot, y_{3})\right\|}_{L^{p}}\right)~{\rm d}y_{3}\\ 
    &+ C|\Omega|2^{-j} \int_{0}^{+\infty}e^{-cy_{3}2^{j}}{\left\|\phi_{j}*{\bf u}(\cdot, y_{3})\right\|}_{L^{p}}~{\rm d}y_{3}. 
  \end{align*}
  Using H\"older's inequality, we have 
  \begin{align}\label{0611-12}
    {\|\phi_{j}*{\bf w}^{(1)}(\cdot, 0)\|}_{L^{p}} 
    \leq& C 2^{-\frac{1}{r_{0}'}j}{\|\phi_{j}*{\bf F}\|}_{L^{r_{0}}(0, \infty; L^{p})} 
    + C 2^{-\frac{1}{(r_{1}/2)'}j}{\|\phi_{j}*({\bf u}\otimes{\bf u})\|}_{L^{\frac{r_{1}}{2}}(0, \infty; L^{p})}\notag\\ 
    &+ C|\Omega|2^{-\left(1 + \frac{1}{r_{1}'}\right)j}{\|\phi_{j}*{\bf u}\|}_{L^{r_{1}}(0, \infty; L^{p})}\notag\\ 
    =& C 2^{\left(\frac{1}{r_{0}} - 1\right)j}{\|\phi_{j}*{\bf F}\|}_{L^{r_{0}}(0, \infty; L^{p})} 
    + C 2^{\left(\frac{2}{r_{1}} - 1\right)j}{\|\phi_{j}*({\bf u}\otimes{\bf u})\|}_{L^{\frac{r_{1}}{2}}(0, \infty; L^{p})}\notag\\ 
    &+ C|\Omega|2^{\left(\frac{1}{r_{1}} - 2\right)j}{\|\phi_{j}*{\bf u}\|}_{L^{r_{1}}(0, \infty; L^{p})}. 
  \end{align}
  Substituting \eqref{0611-12} into the right-hand side of \eqref{1023-1} and 
  multiplying \eqref{0611-12} by $2^{(s + \frac{1}{r})j}$ and taking the $l^{q}(\mathbb{Z})$ norm with respect to $j$, 
  we obtain the desired estimate. 
  This completes completes the proof of Proposition~\ref{0611-13}. 
\end{proof}

\subsection{Explicit formulas and estimates for ${\bf w}^{(2)}$ and ${\bf v}^{(2)}$}\label{0623-2}
In this subsection, we consider 
\begin{equation}\label{0611-14}
  \begin{cases}
  -\Delta{\bf w}^{(2)} + \Omega(e_{3}\times{\bf w}^{(2)}) + \nabla \pi^{(2)} = \text{div }(\bar{\bf F} - \bar{\bf u}\otimes\bar{\bf u}), 
  &\quad \  x\in\mathbb{R}^{3},\\
  \text{div } {\bf w}^{(2)} =0,
  &\quad \  x\in\mathbb{R}^{3}, 
  \end{cases}
\end{equation}
and 
\begin{equation}\label{0611-15}
  \begin{cases}
    -\Delta {\bf v}^{(2)} + \Omega(e_{3}\times{\bf v}^{(2)}) + \nabla q^{(2)} = 0, 
          &\quad  \   x\in \mathbb{R}_{+}^{3}, \\
    \text{div } {\bf v}^{(2)} = 0,
          &\quad  \   x\in \mathbb{R}_{+}^{3}, \\
    {\bf v}^{(2)}(x', 0) =  {\bf a}(x') - {\bf w}^{(2)}(x', 0),
          &\quad  \ x'\in{\mathbb{R}}^2, 
  \end{cases}
\end{equation}
where $\Omega \in \mathbb{R}\setminus\{0\}$, ${\bf u}$ and ${\bf F}$ are given functions on $\mathbb{R}_{+}^{3}$, 
and $\bar{\bf u}$ and $\bar{\bf F}$ are the extensions of ${\bf u}$ and ${\bf F}$ to $\mathbb{R}^{3}$ 
defined by \eqref{0625-6}, \eqref{0625-7} and \eqref{0625-8}, 
and ${\bf a}$ is a given boundary data with $a_{3} = 0$. 
We derive explicit formulas for ${\bf w}^{(2)}$ and ${\bf v}^{(2)}$ by the same procedure as for ${\bf w}^{(1)}$ and ${\bf v}^{(1)}$. 
Since \eqref{0611-14} and \eqref{0611-15} contain the Coriolis term in their principal parts, 
their explicit formulas are more complicated than those for ${\bf w}^{(1)}$ and ${\bf v}^{(1)}$. 

We first consider \eqref{0611-14}. 
We again use the notation ${\bf G} = {\bf F} - {\bf u}\otimes{\bf u}$. 
By applying $\mathbb{P}_{\mathbb{R}^{3}}$ to both sides of \eqref{0611-14} and taking the Fourier transform on $\mathbb{R}^{3}$, 
we obtain 
\begin{equation}\label{0611-16}
  |\xi|^{2}\mathcal{F}_{\mathbb{R}^{3}}[{\bf w}^{(2)}] - \Omega \frac{i\xi_{3}}{|\xi|^{2}}i\xi\times(\mathcal{F}_{\mathbb{R}^{3}}[{\bf w}^{(2)}]) 
  = \mathcal{F}_{\mathbb{R}^{3}}\left[\mathbb{P}_{\mathbb{R}^{3}}(\text{div }\bar{\bf G})\right]. 
\end{equation}
It is known that ${\bf w}^{(2)}$ satisfying \eqref{0611-16} can be represented as follows (see~\cite{Ko_Yo_2011}). 
\begin{align*}
  \mathcal{F}_{\mathbb{R}^{3}}[{\bf w}^{(2)}](\xi) 
  &= \int_{0}^{+\infty}\cos\left(\Omega\frac{\xi_{3}}{|\xi|}t\right)e^{-t|\xi|^{2}}\mathcal{F}_{\mathbb{R}^{3}}\left[\mathbb{P}_{\mathbb{R}^{3}}(\text{div }\bar{\bf G})\right](\xi)~{\rm d}t\\ 
  +& \int_{0}^{+\infty}\frac{i}{|\xi|}\sin\left(\Omega\frac{\xi_{3}}{|\xi|}t\right)e^{-t|\xi|^{2}}\mathcal{F}_{\mathbb{R}^{3}}\left[\nabla\times\left(\mathbb{P}_{\mathbb{R}^{3}}(\text{div }\bar{\bf G})\right)\right](\xi)~{\rm d}t. 
\end{align*}
Note that by the properties of the Helmholtz projection, we have 
\begin{equation*}
  \nabla\times\left(\mathbb{P}_{\mathbb{R}^{3}}(\text{div }\bar{\bf G})\right) 
  = \nabla\times(\text{div }\bar{\bf G}). 
\end{equation*}
Also, we get 
\begin{align*}
  \int_{0}^{+\infty}\cos\left(\Omega\frac{\xi_{3}}{|\xi|}t\right)e^{-t|\xi|^{2}}~{\rm d}t 
  &= \frac{|\xi|^{4}}{|\xi|^{6} + |\Omega|^{2}\xi_{3}^{2}}, \\ 
  \int_{0}^{+\infty}\frac{i}{|\xi|}\sin\left(\Omega\frac{\xi_{3}}{|\xi|}t\right)e^{-t|\xi|^{2}}~{\rm d}t 
  &= \frac{\Omega i\xi_{3}}{|\xi|^{6} + |\Omega|^{2}\xi_{3}^{2}}. 
\end{align*}
Thus, we obtain 
\begin{equation}\label{0612-1}
  \mathcal{F}_{\mathbb{R}^{3}}[{\bf w}^{(2)}] 
  = \frac{|\xi|^{4}}{|\xi|^{6} + |\Omega|^{2}\xi_{3}^{2}}\mathcal{F}_{\mathbb{R}^{3}}\left[\mathbb{P}_{\mathbb{R}^{3}}(\text{div }\bar{\bf G})\right] 
  + \frac{\Omega i\xi_{3}}{|\xi|^{6} + |\Omega|^{2}\xi_{3}^{2}}i\xi\times (\mathcal{F}_{\mathbb{R}^{3}}\left[\text{div }\bar{\bf G}\right]). 
\end{equation}
As in Subsection~\ref{0612-2}, we take the inverse Fourier transform of \eqref{0612-1} with respect to $\xi_{3}$. 
The representation of $\mathcal{F}_{\mathbb{R}^{3}}\left[\mathbb{P}_{\mathbb{R}^{3}}(\text{div }\bar{\bf G})\right]$ has already been given in \eqref{0612-3}, 
and we can write 
\begin{align*}
  \left(i\xi\times\big(\mathcal{F}_{\mathbb{R}^{3}}\left[\text{div }\bar{\bf G}\right]\big)\right)_{h} 
  =& i\xi_{3}\left(\mathcal{F}_{\mathbb{R}^{3}}[\text{div }\bar{\bf G}]\right)_{h}^{\perp}
- (i\xi')^{\perp}\left(\mathcal{F}_{\mathbb{R}^{3}}[\text{div }\bar{\bf G}]\right)_{3}\\ 
  =& i\xi_{3}\sum_{m = 1}^{2}i\xi_{m}\mathcal{F}_{\mathbb{R}^{3}}[(\bar{G}_{m, 1}, \bar{G}_{m, 2})^{\perp}] 
  + (i\xi_{3})^{2}\mathcal{F}_{\mathbb{R}^{3}}[(\bar{G}_{3, 1}, \bar{G}_{3, 2})^{\perp}]\\ 
  &- (i\xi')^{\perp}\sum_{m = 1}^{2}i\xi_{m}\mathcal{F}_{\mathbb{R}^{3}}[\bar{G}_{m, 3}] 
  - (i\xi')^{\perp}i\xi_{3}\mathcal{F}_{\mathbb{R}^{3}}[\bar{G}_{3, 3}],\\ 
  \left(i\xi\times\big(\mathcal{F}_{\mathbb{R}^{3}}\left[\text{div }\bar{\bf G}\right]\big)\right)_{3} 
  =& (i\xi')^{\perp}\cdot\left(\mathcal{F}_{\mathbb{R}^{3}}[\text{div }\bar{\bf G}]\right)_{h}\\ 
  =& (i\xi')^{\perp}\cdot\Big(\sum_{m = 1}^{2}i\xi_{m}\mathcal{F}_{\mathbb{R}^{3}}[(\bar{G}_{m, 1}, \bar{G}_{m, 2})] 
  + i\xi_{3}\mathcal{F}_{\mathbb{R}^{3}}[(\bar{G}_{3, 1}, \bar{G}_{3, 2})]\Big). 
\end{align*}
Hence, it follows that 
\begin{align*}
  \mathcal{F}_{\mathbb{R}^{3}}[{\bf w}_{h}^{(2)}] 
  =& \frac{(|\xi'|^{2} + \xi_{3}^{2})^{2}}{(|\xi'|^{2} + \xi_{3}^{2})^{3} + |\Omega|^{2}\xi_{3}^{2}}\left(\sum_{m = 1}^{2}i\xi_{m}\mathcal{F}_{\mathbb{R}^{3}}\left[(\bar{G}_{m, 1}, \bar{G}_{m, 2})\right]\right. \\ 
  &+ i\xi_{3}\mathcal{F}_{\mathbb{R}^{3}}\left[(\bar{G}_{3, 1}, \bar{G}_{3, 2})\right]
  + \frac{i\xi'}{|\xi'|^{2} + \xi_{3}^{2}}\sum_{m, n = 1}^{2}i\xi_{m}i\xi_{n}\mathcal{F}_{\mathbb{R}^{3}}\left[\bar{G}_{m,n}\right]\\ 
  &+ \frac{i\xi'i\xi_{3}}{|\xi'|^{2} + \xi_{3}^{2}}\sum_{m = 1}^{2}i\xi_{m}\mathcal{F}_{\mathbb{R}^{3}}\left[\bar{G}_{m, 3} + \bar{G}_{3,m}\right] 
  + \left.\frac{i\xi'(i\xi_{3})^{2}}{|\xi'|^{2} + \xi_{3}^{2}}\mathcal{F}_{\mathbb{R}^{3}}\left[\bar{G}_{3,3}\right]\right) \\ 
  &+ \frac{\Omega i\xi_{3}}{(|\xi'|^{2} + \xi_{3}^{2})^{3} + |\Omega|^{2}\xi_{3}^{2}}\left(i\xi_{3}\sum_{m = 1}^{2}i\xi_{m}\mathcal{F}_{\mathbb{R}^{3}}[(\bar{G}_{m, 1}, \bar{G}_{m, 2})^{\perp}]\right.\\  
  &+ (i\xi_{3})^{2}\mathcal{F}_{\mathbb{R}^{3}}[(\bar{G}_{3, 1}, \bar{G}_{3, 2})^{\perp}]
  \left.- (i\xi')^{\perp}\sum_{m = 1}^{2}i\xi_{m}\mathcal{F}_{\mathbb{R}^{3}}[\bar{G}_{m, 3}] 
  - (i\xi')^{\perp}i\xi_{3}\mathcal{F}_{\mathbb{R}^{3}}[\bar{G}_{3, 3}]\right)\\ 
  =& \frac{1}{(|\xi'|^{2} + \xi_{3}^{2})^{3} + |\Omega|^{2}\xi_{3}^{2}}\left(\sum_{m = 1}^{2}(|\xi'|^{2} + \xi_{3}^{2})^{2}i\xi_{m}\mathcal{F}_{\mathbb{R}^{3}}\left[(\bar{G}_{m, 1}, \bar{G}_{m, 2})\right]\right. \\ 
  &+ i\xi_{3}(|\xi'|^{2} + \xi_{3}^{2})^{2}\mathcal{F}_{\mathbb{R}^{3}}\left[(\bar{G}_{3, 1}, \bar{G}_{3, 2})\right]
  + (|\xi'|^{2} + \xi_{3}^{2})i\xi'\sum_{m, n = 1}^{2}i\xi_{m}i\xi_{n}\mathcal{F}_{\mathbb{R}^{3}}\left[\bar{G}_{m,n}\right]\\ 
  &+ i\xi_{3}(|\xi'|^{2} + \xi_{3}^{2})i\xi'\sum_{m = 1}^{2}i\xi_{m}\mathcal{F}_{\mathbb{R}^{3}}\left[\bar{G}_{m, 3} + \bar{G}_{3,m}\right]\\ 
  &+ (i\xi_{3})^{2}(|\xi'|^{2} + \xi_{3}^{2})i\xi'\mathcal{F}_{\mathbb{R}^{3}}\left[\bar{G}_{3,3}\right] \\ 
  &+ \Omega(i\xi_{3})^{2}\sum_{m = 1}^{2}i\xi_{m}\mathcal{F}_{\mathbb{R}^{3}}[(\bar{G}_{m, 1}, \bar{G}_{m, 2})^{\perp}] 
  + \Omega(i\xi_{3})^{3}\mathcal{F}_{\mathbb{R}^{3}}[(\bar{G}_{3, 1}, \bar{G}_{3, 2})^{\perp}]\\ 
  &- \left.\Omega i\xi_{3}(i\xi')^{\perp}\sum_{m = 1}^{2}i\xi_{m}\mathcal{F}_{\mathbb{R}^{3}}[\bar{G}_{m, 3}] 
  - \Omega(i\xi_{3})^{2}(i\xi')^{\perp}\mathcal{F}_{\mathbb{R}^{3}}[\bar{G}_{3, 3}]\right), \\ 
  \mathcal{F}_{\mathbb{R}^{3}}[w_{3}^{(2)}] 
  =& \frac{1}{(|\xi'|^{2} + \xi_{3}^{2})^{3} + |\Omega|^{2}\xi_{3}^{2}}\left((|\xi'|^{2} + \xi_{3}^{2})|\xi'|^{2}\sum_{m = 1}^{2}i\xi_{m}\mathcal{F}_{\mathbb{R}^{3}}\left[\bar{G}_{m, 3}\right]\right.\\ 
  &+ i\xi_{3}(|\xi'|^{2} + \xi_{3}^{2})|\xi'|^{2}\mathcal{F}_{\mathbb{R}^{3}}\left[\bar{G}_{3,3}\right] 
  + i\xi_{3}(|\xi'|^{2} + \xi_{3}^{2})\sum_{m,n = 1}i\xi_{m}i\xi_{n}\mathcal{F}_{\mathbb{R}^{3}}\left[\bar{G}_{m,n}\right]\\ 
  &+ (i\xi_{3})^{2}(|\xi'|^{2} + \xi_{3}^{2})\sum_{m = 1}^{2}i\xi_{m}\mathcal{F}_{\mathbb{R}^{3}}\left[\bar{G}_{3,m}\right]\\ 
  &+ \Omega i\xi_{3}(i\xi')^{\perp}\cdot\Big(\sum_{m = 1}^{2}i\xi_{m}\mathcal{F}_{\mathbb{R}^{3}}[(\bar{G}_{m, 1}, \bar{G}_{m, 2})]\Big)\\ 
  &+ \Omega (i\xi_{3})^{2}(i\xi')^{\perp}\cdot(\mathcal{F}_{\mathbb{R}^{3}}[(\bar{G}_{3, 1}, \bar{G}_{3, 2})])\Bigg). 
\end{align*}
We consider 
\begin{equation*}
  \mathcal{F}_{\mathbb{R}_{\xi_{3}}}^{-1}\left[\frac{1}{(|\xi'|^{2} + \xi_{3}^{2})^{3} + |\Omega|^{2}\xi_{3}^{2}}\right](\xi', \xi_{3}). 
\end{equation*}
The following holds. 

\begin{lem}\label{0804-2}
  Let $\xi'\in \mathbb{R}^{2}\backslash\{0\}$ and $\Omega \in \mathbb{R}\backslash\{0\}$. 
  Then we have 
  \begin{equation*}
    \mathcal{F}_{\mathbb{R}_{\xi_{3}}}^{-1}\left[\frac{1}{(|\xi'|^{2} + \xi_{3}^{2})^{3} + |\Omega|^{2}\xi_{3}^{2}}\right](\xi', x_{3}) 
    = \left(\frac{\pi}{2}\right)^{\frac{1}{2}}\sum_{k = 1}^{3}\frac{B_{k}}{\lambda_{|\Omega|, k}}e^{-\lambda_{|\Omega|, k}|x_{3}|}, 
  \end{equation*}
  where $\lambda_{|\Omega|, k} = \lambda_{|\Omega|, k}(|\xi'|)$ $(k = 1,2,3)$ is a root of the following algebraic equation for a fixed $|\xi'|$ and $|\Omega|$: 
  \begin{equation}\label{0804-1}
    (|\xi'|^{2} - \lambda^{2})^{3} - |\Omega|^{2}\lambda^{2} = 0, 
  \end{equation}
  and 
  \begin{equation*}
    B_{k} := \left(\prod_{\alpha = 1, \alpha \ne k}^{3}(\lambda_{|\Omega|, \alpha}^{2}-\lambda_{|\Omega|, k}^2)\right)^{-1}. 
  \end{equation*}
\end{lem}

\begin{proof}
  We can write 
  \begin{align*}
    (|\xi'|^{2} + \xi_{3}^{2})^{3} + |\Omega|^{2}\xi_{3}^{2} 
    &= (|\xi'|^{2} - (i\xi_{3})^{2})^{3} - |\Omega|^{2}(i\xi_{3})^{2}\\ 
    &= (\lambda_{|\Omega|, 1}^{2} - (i\xi_{3})^{2})(\lambda_{|\Omega|, 2}^{2} - (i\xi_{3})^{2})(\lambda_{|\Omega|, 3}^{2} - (i\xi_{3})^{2}). 
  \end{align*}
  Note that $\lambda_{|\Omega|, k}$ cannot be purely imaginary, 
  since substituting $\lambda = i\eta$ with $\eta \in \mathbb{R}$ into \eqref{0804-1} yields 
  \begin{equation*}
    (|\xi'|^{2} + \eta^{2})^{3} + |\Omega|^{2}\eta^{2} = 0, 
  \end{equation*}
  which is impossible for $|\xi'| \neq 0$. 
  Also, without loss of generality, we can assume that $\Re(\lambda_{|\Omega|, k}) > 0$ $(k = 1,2,3)$. 
  A detailed expression for $\lambda_{|\Omega|, k}$ will be given later; see \eqref{0625-9}. 
  Therefore, by partial fraction decomposition, it follows that 
  \begin{align*}
    \frac{1}{(|\xi'|^{2} + \xi_{3}^{2})^{3} + |\Omega|^{2}\xi_{3}^{2}} 
    =& \frac{1}{(\lambda_{|\Omega|, 1}^{2} - (i\xi_{3})^{2})(\lambda_{|\Omega|, 2}^{2} - (i\xi_{3})^{2})(\lambda_{|\Omega|, 3}^{2} - (i\xi_{3})^{2})}\\ 
    =& \sum_{k = 1}^{3}\frac{B_{k}}{\lambda_{|\Omega|, k}^{2} - (i\xi_{3})^{2}}. 
  \end{align*}
  Now, we consider 
  \begin{align*}
    \mathcal{F}_{\mathbb{R}_{\xi_{3}}}^{-1}\left[\frac{1}{\lambda_{|\Omega|, k}^{2} - (i\xi_{3})^{2}}\right](\xi', x_{3}) 
    =& \frac{1}{(2\pi)^{\frac{1}{2}}}\int_{-\infty}^{+\infty}\frac{e^{i\xi_{3}x_{3}}}{\lambda_{|\Omega|, k}^{2} - (i\xi_{3})^{2}}~{\rm d}\xi_{3}\\ 
    =& \frac{1}{(2\pi)^{\frac{1}{2}}}\int_{-\infty}^{+\infty}\frac{e^{i\xi_{3}x_{3}}}{(\xi_{3} - i\lambda_{|\Omega|, k})(\xi_{3} + i\lambda_{|\Omega|, k})}~{\rm d}\xi_{3}. 
  \end{align*}
  Since the above integrand has simple poles at $\xi_{3} = \pm i\lambda_{|\Omega|, k}$ and $\Re(\lambda_{|\Omega|, k}) > 0$ $(k = 1, 2, 3)$, 
  using the residue theorem, if $x_{3} > 0$, then we obtain 
  \begin{equation*}
    \int_{-\infty}^{+\infty}\frac{e^{i\xi_{3}x_{3}}}{(\xi_{3} - i\lambda_{|\Omega|, k})(\xi_{3} + i\lambda_{|\Omega|, k})}~{\rm d}\xi_{3} 
    = 2\pi i\lim_{\xi_{3} \to +i\lambda_{|\Omega|, k}}\left(\frac{e^{i\xi_{3}x_{3}}}{(\xi_{3} + i\lambda_{|\Omega|, k})}\right) 
    = \frac{\pi}{\lambda_{|\Omega|, k}}e^{-\lambda_{|\Omega|, k}x_{3}}. 
  \end{equation*}
  Similarly, if $x_{3} < 0$, it follows that 
  \begin{equation*}
    \int_{-\infty}^{+\infty}\frac{e^{i\xi_{3}x_{3}}}{(\xi_{3} - i\lambda_{|\Omega|, k})(\xi_{3} + i\lambda_{|\Omega|, k})}~{\rm d}\xi_{3} 
    = -2\pi i\lim_{\xi_{3} \to -i\lambda_{|\Omega|, k}}\left(\frac{e^{i\xi_{3}x_{3}}}{(\xi_{3} - i\lambda_{|\Omega|, k})}\right) 
    = \frac{\pi}{\lambda_{|\Omega|, k}}e^{\lambda_{|\Omega|, k}x_{3}}. 
  \end{equation*}
  Therefore, we have 
  \begin{equation*}
    \mathcal{F}_{\mathbb{R}_{\xi_{3}}}^{-1}\left[\frac{1}{(|\xi'|^{2} + \xi_{3}^{2})^{3} + |\Omega|^{2}\xi_{3}^{2}}\right](\xi', x_{3}) 
    = \left(\frac{\pi}{2}\right)^{\frac{1}{2}}\sum_{k = 1}^{3}\frac{B_{k}}{\lambda_{|\Omega|, k}}e^{-\lambda_{|\Omega|, k}|x_{3}|}. 
  \end{equation*}
\end{proof}

\begin{nota}
  Applying Lemma~\ref{0804-2}, we obtain the following. 
  \begin{align}\label{0612-4}
    \mathcal{F}_{\mathbb{R}_{\xi_{3}}}^{-1}\left[\frac{i\xi_{3}}{(|\xi'|^{2} + \xi_{3}^{2})^{3} + |\Omega|^{2}\xi_{3}^{2}}\right](\xi', x_{3}) 
    =& \partial_{3}\mathcal{F}_{\mathbb{R}_{\xi_{3}}}^{-1}\left[\frac{1}{(|\xi'|^{2} + \xi_{3}^{2})^{3} + |\Omega|^{2}\xi_{3}^{2}}\right](\xi', x_{3})\notag\\ 
    =& -\left(\frac{\pi}{2}\right)^{\frac{1}{2}}\sum_{k = 1}^{3}\text{sgn}(x_{3})B_{k}e^{-\lambda_{|\Omega|, k}|x_{3}|} 
    =: K_{1}^{(2)}(\xi', x_{3}), \notag\\ 
    \mathcal{F}_{\mathbb{R}_{\xi_{3}}}^{-1}\left[\frac{(i\xi_{3})^{2}}{(|\xi'|^{2} + \xi_{3}^{2})^{3} + |\Omega|^{2}\xi_{3}^{2}}\right](\xi', x_{3}) 
    =& \partial_{3}K_{1}^{(2)}(\xi', x_{3})\notag\\ 
    =& \left(\frac{\pi}{2}\right)^{\frac{1}{2}}\sum_{k = 1}^{3}(- \partial_{x_{3}}\text{sgn}(x_{3}) + \lambda_{|\Omega|, k})B_{k}e^{-\lambda_{|\Omega|, k}|x_{3}|}\notag\\ 
    =& \left(\frac{\pi}{2}\right)^{\frac{1}{2}}\sum_{k = 1}^{3}(- 2\delta_{0}(x_{3}) + \lambda_{|\Omega|, k})B_{k}e^{-\lambda_{|\Omega|, k}|x_{3}|}\notag\\
    =& \left(\frac{\pi}{2}\right)^{\frac{1}{2}}\sum_{k = 1}^{3}\lambda_{|\Omega|, k}B_{k}e^{-\lambda_{|\Omega|, k}|x_{3}|} 
    =: K_{2}^{(2)}(\xi', x_{3}). 
  \end{align}
  Here, we note that the equality $\partial_{x_{3}}\text{sgn}(x_{3}) = 2\delta_{0}(x_{3})$ appearing in the above is justified in the sense of distributions, 
  and we have used 
  \begin{equation*}
    \sum_{k = 1}^{3}B_{k} = 0. 
  \end{equation*}
  Also, we get 
  \begin{align}\label{0612-5}
    \mathcal{F}_{\mathbb{R}_{\xi_{3}}}^{-1}\left[\frac{|\xi'|^{2} + \xi_{3}^{2}}{(|\xi'|^{2} + \xi_{3}^{2})^{3} + |\Omega|^{2}\xi_{3}^{2}}\right](\xi', x_{3}) 
    &= \left(\frac{\pi}{2}\right)^{\frac{1}{2}}\sum_{k=1}^{3}\frac{(|\xi'|^{2} - \lambda_{|\Omega|, k}^{2})}{\lambda_{|\Omega|, k}}B_{k}e^{-\lambda_{|\Omega|, k}|x_{3}|}\notag\\ 
    &=: K_{3}^{(2)}(\xi', x_{3}),\notag\\ 
    \mathcal{F}_{\mathbb{R}_{\xi_{3}}}^{-1}\left[\frac{(i\xi_{3})^{3}}{(|\xi'|^{2} + \xi_{3}^{2})^{3} + |\Omega|^{2}\xi_{3}^{2}}\right](\xi', x_{3}) 
    &= - \left(\frac{\pi}{2}\right)^{\frac{1}{2}}\sum_{k=1}^{3}\text{sgn}(x_3)\lambda_{|\Omega|, k}^{2}B_{k}e^{-\lambda_{|\Omega|, k}|x_{3}|}\notag\\ 
    &=: K_{4}^{(2)}(\xi', x_{3}),\notag\\ 
    \mathcal{F}_{\mathbb{R}_{\xi_{3}}}^{-1}\left[\frac{i\xi_{3}(|\xi'|^{2} + \xi_{3}^{2})}{(|\xi'|^{2} + \xi_{3}^{2})^{3} + |\Omega|^{2}\xi_{3}^{2}}\right](\xi', x_{3}) 
    &= - \left(\frac{\pi}{2}\right)^{\frac{1}{2}}\sum_{k=1}^{3}\text{sgn}(x_3)(|\xi'|^{2} - \lambda_{|\Omega|, k}^{2})B_{k}e^{-\lambda_{|\Omega|, k}|x_{3}|}\notag\\ 
    &=: K_{5}^{(2)}(\xi', x_{3}),\notag\\ 
    \mathcal{F}_{\mathbb{R}_{\xi_{3}}}^{-1}\left[\frac{(i\xi_{3})^{2}(|\xi'|^{2} + \xi_{3}^{2})}{(|\xi'|^{2} + \xi_{3}^{2})^{3} + |\Omega|^{2}\xi_{3}^{2}}\right](\xi', x_{3}) 
    &= \left(\frac{\pi}{2}\right)^{\frac{1}{2}}\sum_{k=1}^{3}\lambda_{|\Omega|, k}(|\xi'|^{2} - \lambda_{|\Omega|, k}^{2})B_{k}e^{-\lambda_{|\Omega|, k}|x_{3}|}\notag\\ 
    &=: K_{6}^{(2)}(\xi', x_{3}),\notag\\ 
    \mathcal{F}_{\mathbb{R}_{\xi_{3}}}^{-1}\left[\frac{(|\xi'|^{2} + \xi_{3}^{2})^{2}}{(|\xi'|^{2} + \xi_{3}^{2})^{3} + |\Omega|^{2}\xi_{3}^{2}}\right](\xi', x_{3}) 
    &= \left(\frac{\pi}{2}\right)^{\frac{1}{2}}\sum_{k=1}^{3}\frac{(|\xi'|^{2} - \lambda_{|\Omega|, k}^{2})^{2}}{\lambda_{|\Omega|, k}}B_{k}e^{-\lambda_{|\Omega|, k}|x_{3}|}\notag\\ 
    &=: K_{7}^{(2)}(\xi', x_{3}),\notag\\ 
    \mathcal{F}_{\mathbb{R}_{\xi_{3}}}^{-1}\left[\frac{i\xi_{3}(|\xi'|^{2} + \xi_{3}^{2})^{2}}{(|\xi'|^{2} + \xi_{3}^{2})^{3} + |\Omega|^{2}\xi_{3}^{2}}\right](\xi', x_{3}) 
    &= - \left(\frac{\pi}{2}\right)^{\frac{1}{2}}\sum_{k=1}^{3}\text{sgn}(x_3)(|\xi'|^{2} - \lambda_{|\Omega|, k}^{2})^{2}B_{k}e^{-\lambda_{|\Omega|, k}|x_{3}|}\notag\\ 
    &=: K_{8}^{(2)}(\xi', x_{3}). 
  \end{align}
\end{nota}

Therefore, we can write 
\begin{align*}
  \mathcal{F}_{\mathbb{R}^{3}}[{\bf w}_{h}^{(2)}] 
  =& \sum_{m = 1}^{2}i\xi_{m}\mathcal{F}_{\mathbb{R}_{x_{3}}}\big[K_{7}^{(2)}(\xi', x_{3})\big]\mathcal{F}_{\mathbb{R}^{3}}\left[(\bar{G}_{m, 1}, \bar{G}_{m, 2})\right]\\ 
  &+ \mathcal{F}_{\mathbb{R}_{x_{3}}}\big[K_{8}^{(2)}(\xi', x_{3})\big]\mathcal{F}_{\mathbb{R}^{3}}\left[(\bar{G}_{3, 1}, \bar{G}_{3, 2})\right] \\ 
  &+ i\xi'\sum_{m, n = 1}^{2}i\xi_{m}i\xi_{n}\mathcal{F}_{\mathbb{R}_{x_{3}}}\big[K_{3}^{(2)}(\xi', x_{3})\big]\mathcal{F}_{\mathbb{R}^{3}}\left[\bar{G}_{m,n}\right]\\ 
  &+ i\xi'\sum_{m = 1}^{2}i\xi_{m}\mathcal{F}_{\mathbb{R}_{x_{3}}}\big[K_{5}^{(2)}(\xi', x_{3})\big]\mathcal{F}_{\mathbb{R}^{3}}\left[\bar{G}_{m, 3} + \bar{G}_{3,m}\right]\\ 
  &+ i\xi'\mathcal{F}_{\mathbb{R}_{x_{3}}}\big[K_{6}^{(2)}(\xi', x_{3})\big]\mathcal{F}_{\mathbb{R}^{3}}\left[\bar{G}_{3,3}\right] \\ 
  &+ \Omega\sum_{m = 1}^{2}i\xi_{m}\mathcal{F}_{\mathbb{R}_{x_{3}}}\big[K_{2}^{(2)}(\xi', x_{3})\big]\mathcal{F}_{\mathbb{R}^{3}}[(\bar{G}_{m, 1}, \bar{G}_{m, 2})^{\perp}]\\ 
  &+ \Omega\mathcal{F}_{\mathbb{R}_{x_{3}}}\big[K_{4}^{(2)}(\xi', x_{3})\big]\mathcal{F}_{\mathbb{R}^{3}}[(\bar{G}_{3, 1}, \bar{G}_{3, 2})^{\perp}]\\ 
  &- \Omega (i\xi')^{\perp}\sum_{m = 1}^{2}i\xi_{m}\mathcal{F}_{\mathbb{R}_{x_{3}}}\big[K_{1}^{(2)}(\xi', x_{3})\big]\mathcal{F}_{\mathbb{R}^{3}}[\bar{G}_{m, 3}]\\ 
  &- \Omega (i\xi')^{\perp}\mathcal{F}_{\mathbb{R}_{x_{3}}}\big[K_{2}^{(2)}(\xi', x_{3})\big]\mathcal{F}_{\mathbb{R}^{3}}[\bar{G}_{3, 3}], \\ 
  \mathcal{F}_{\mathbb{R}^{3}}[w_{3}^{(2)}] 
  =& |\xi'|^{2}\sum_{m = 1}^{2}i\xi_{m}\mathcal{F}_{\mathbb{R}_{x_{3}}}\big[K_{3}^{(2)}(\xi', x_{3})\big]\mathcal{F}_{\mathbb{R}^{3}}\left[\bar{G}_{m, 3}\right] 
  + |\xi'|^{2}\mathcal{F}_{\mathbb{R}_{x_{3}}}\big[K_{5}^{(2)}(\xi', x_{3})\big]\mathcal{F}_{\mathbb{R}^{3}}\left[\bar{G}_{3,3}\right]\\ 
  &+ \sum_{m,n = 1}^{2}i\xi_{m}i\xi_{n}\mathcal{F}_{\mathbb{R}_{x_{3}}}\big[K_{5}^{(2)}(\xi', x_{3})\big]\mathcal{F}_{\mathbb{R}^{3}}\left[\bar{G}_{m,n}\right]\\ 
  &+ \sum_{m = 1}^{2}i\xi_{m}\mathcal{F}_{\mathbb{R}_{x_{3}}}\big[K_{6}^{(2)}(\xi', x_{3})\big]\mathcal{F}_{\mathbb{R}^{3}}\left[\bar{G}_{3,m}\right]\\ 
  &+ \Omega (i\xi')^{\perp}\cdot\Big(\sum_{m = 1}^{2}i\xi_{m}\mathcal{F}_{\mathbb{R}_{x_{3}}}\big[K_{1}^{(2)}(\xi', x_{3})\big]\mathcal{F}_{\mathbb{R}^{3}}\left[(\bar{G}_{m, 1}, \bar{G}_{m, 2})\right]\Big)\\ 
  &+ \Omega (i\xi')^{\perp}\cdot\big(\mathcal{F}_{\mathbb{R}_{x_{3}}}\big[K_{2}^{(2)}(\xi', x_{3})\big]\mathcal{F}_{\mathbb{R}^{3}}\left[(\bar{G}_{3, 1}, \bar{G}_{3, 2})\right]\big). 
\end{align*}
The remainder of the derivation is the same as that for ${\bf w}^{(1)}$. 
Similarly to $T_{K_{l}^{(1)}}^{(\pm)}$, we define 
\begin{align}\label{0612-6}
  &T_{K_{l}^{(2)}}^{(\pm)}[f](x', x_{3})\notag\\ 
  :=& \frac{1}{\sqrt{2\pi}}\int_{0}^{+\infty}\mathcal{F}_{\mathbb{R}^{2}}^{-1}\left[\left(K_{l}^{(2)}(\xi', x_{3} - y_{3}) \pm K_{l}^{(2)}(\xi', x_{3} + y_{3})\right)\widehat{f}(\xi', y_{3})\right](x')~{\rm d}y_{3}, 
\end{align}
for $l = 1, 2, \cdots 8$. 
Recall that if $\bar{G}_{m,n}$ is the even extension of $G_{m, n}$ with respect to $x_{3}$, then we have 
\begin{equation*}
  \mathcal{F}_{\mathbb{R}^{3}}^{-1}\left[\mathcal{F}_{\mathbb{R}_{x_{3}}}\big[K_{l}^{(2)}(\xi', x_{3})\big]\mathcal{F}_{\mathbb{R}^{3}}\left[\bar{G}_{m, n}\right]\right](x', x_{3})
  = T_{K_{l}^{(2)}}^{(+)}[G_{m, n}](x', x_{3}), 
\end{equation*}
also, if $\bar{G}_{m,n}$ is the odd extension of $G_{m, n}$ with respect to $x_{3}$, then we have 
\begin{equation*}
  \mathcal{F}_{\mathbb{R}^{3}}^{-1}\left[\mathcal{F}_{\mathbb{R}_{x_{3}}}\big[K_{l}^{(2)}(\xi', x_{3})\big]\mathcal{F}_{\mathbb{R}^{3}}\left[\bar{G}_{m, n}\right]\right](x', x_{3})
  = T_{K_{l}^{(2)}}^{(-)}[G_{m, n}](x', x_{3}). 
\end{equation*}
By recalling the definition of $\bar{\bf G} = \bar{\bf F} - \bar{\bf u}\otimes\bar{\bf u}$ (\eqref{0625-6}, \eqref{0625-7} and \eqref{0625-8}), 
we obtain the following proposition. 
\begin{prop}\label{0626-4}
  Let $\Omega \in \mathbb{R}\setminus\{0\}$. 
  Then the solution ${\bf w}^{(2)} = (w_{1}^{(2)}, w_{2}^{(2)}, w_{3}^{(2)}) := ({\bf w}_{h}^{(2)}, w_{3}^{(2)})$ of \eqref{0611-14} is given by 
\begin{align}\label{0612-18}
  {\bf w}_{h}^{(2)}(x', x_{3}) 
  =& \sum_{m = 1}^{2}\partial_{m}T_{K_{7}^{(2)}}^{(+)}\left[(F_{m, 1}, F_{m, 2}) - u_{m}{\bf u}_{h}\right](x', x_{3}) \notag\\ 
  &+ T_{K_{8}^{(2)}}^{(-)}\left[(F_{3, 1}, F_{3, 2}) - {\bf u}_{h}u_{3}\right](x', x_{3})\notag\\ 
  &+ \nabla'\Big(\sum_{m,n = 1}^{2}\partial_{m}\partial_{n}T_{K_{3}^{(2)}}^{(+)}\left[F_{m,n} - u_{m}u_{n}\right]\Big)(x', x_{3})\notag\\ 
  &+ \nabla'\Big(\sum_{m = 1}^{2}\partial_{m}T_{K_{5}^{(2)}}^{(-)}\left[F_{m,3} + F_{3,m} - 2u_{m}u_{3}\right]\Big)(x', x_{3})\notag\\ 
  &+ \nabla'\big(T_{K_{6}^{(2)}}^{(+)}\left[F_{3,3} - u_{3}u_{3}\right]\big)(x', x_{3})\notag\\ 
  &+ \Omega \sum_{m = 1}^{2}\partial_{m}T_{K_{2}^{(2)}}^{(+)}\left[(F_{m, 1}, F_{m, 2})^{\perp} - u_{m}{\bf u}_{h}^{\perp}\right](x', x_{3})\notag\\ 
  &+ \Omega T_{K_{4}^{(2)}}^{(-)}\left[(F_{3, 1}, F_{3, 2})^{\perp} - u_{3}{\bf u}_{h}^{\perp}\right](x', x_{3}),\notag\\ 
  &- \Omega (\nabla')^{\perp}\Big(\sum_{m = 1}^{2}\partial_{m}T_{K_{1}^{(2)}}^{(-)}\left[F_{m,3} - u_{m}u_{3}\right]\Big)(x', x_{3})\notag\\ 
  &- \Omega (\nabla')^{\perp}\big(T_{K_{2}^{(2)}}^{(+)}\left[F_{3,3} - u_{3}u_{3}\right]\big)(x', x_{3})\notag\\ 
  w_{3}^{(2)}(x', x_{3})
  =& -\Delta_{\mathbb{R}^{2}}\Big(\sum_{m = 1}^{2}\partial_{m}T_{K_{3}^{(2)}}^{(-)}\left[F_{m,3} - u_{m}u_{3}\right]\Big)(x', x_{3})\notag\\ 
  &-\Delta_{\mathbb{R}^{2}}\big(T_{K_{5}^{(2)}}^{(+)}\left[F_{3,3} - u_{3}u_{3}\right]\big)(x', x_{3})\notag\\ 
  &+ \sum_{m,n = 1}^{2}\partial_{m}\partial_{n}T_{K_{5}^{(2)}}^{(+)}\left[F_{m,n} - u_{m}u_{n}\right](x', x_{3})\notag\\ 
  &+ \sum_{m = 1}^{2}\partial_{m}T_{K_{6}^{(2)}}^{(-)}\left[F_{3,m} - u_{3}u_{m}\right](x', x_{3})\notag\\ 
  &+ \Omega (\nabla')^{\perp}\cdot\Big(\sum_{m = 1}^{2}\partial_{m} T_{K_{1}^{(2)}}^{(+)}\left[(F_{m, 1}, F_{m, 2}) - {\bf u}_{h}u_{m}\right]\Big)(x', x_{3})\notag\\ 
  &+ \Omega(\nabla')^{\perp}\cdot \big(T_{K_{2}^{(2)}}^{(-)}\left[(F_{3, 1}, F_{3, 2}) - {\bf u}_{h}u_{3}\right]\big)(x', x_{3}). 
\end{align}
where $K_{l}^{(2)}$ and $T_{K_{l}^{(2)}}^{(\pm)}$ ($l = 1, 2, \cdots , 8$) are given by \eqref{0612-4}, \eqref{0612-5} and \eqref{0612-6}. 
\end{prop}

\begin{rem}
  Similarly to ${\bf w}^{(1)}$, it follows from the definition of $K_{l}^{(2)}$ (\eqref{0612-4}, \eqref{0612-5}) and $T_{K_{l}^{2}}^{(\pm)}$ \eqref{0612-6} $(l = 1, 2, \cdots , 8)$ that 
  ${\bf w}^{(2)}$ preserves the symmetry of ${\bf u}$ and ${\bf F}$ with respect to $x_{3}$. 
  Thus, we have 
  \begin{equation*}
    w_{3}^{(2)}(\cdot, 0) 
    = 0. 
  \end{equation*}
\end{rem}

We now consider the explicit formula for the solution ${\bf v}^{(2)}$ of \eqref{0611-15}. 
In the case $\Omega = 1$, the explicit formula for ${\bf v}^{(2)}$ was derived by Dalibard and Prange~\cite{Da_Pr_2014}, 
and we follow the same derivation in the present paper. 
However, since the dependence of the explicit formula on the parameter $\Omega$ plays a crucial role in our study, 
we do not omit the derivation and provide it for completeness. 
We note that 
\begin{equation*}
  a_{3}(\cdot) - w_{3}^{(2)}(\cdot, 0) 
  = 0 
\end{equation*}
holds. 

To begin with, we note that the pressure $q^{(2)}$ is given by 
\begin{equation}\label{0612-8}
  -\Delta q^{(2)} 
  = -\Omega (\nabla')^{\perp}\cdot{\bf v}_{h}^{(2)}. 
\end{equation}
Taking the Fourier transform on $\mathbb{R}^{2}$, 
we find that $\widehat{\bf v}^{(2)} = (\widehat{\bf v}_{h}^{(2)}, \widehat{v}_{3}^{(2)})$ satisfies 
\begin{equation}\label{0612-7}
  \begin{cases}
    (|\xi'|^{2} - \partial_{3}^{2})\widehat{\bf v}_{h}^{(2)} + \Omega(\widehat{\bf v}_{h}^{(2)})^{\perp} + i\xi'\widehat{q}^{(2)} = 0, \\ 
    (|\xi'|^{2} - \partial_{3}^{2})\widehat{v}_{3}^{(2)} + \partial_{3}\widehat{q}^{(2)} = 0. 
  \end{cases}
\end{equation}
Thanks to $(i\xi')^{\perp}\cdot i\xi' = 0$, we have 
\begin{equation*}
  (|\xi'|^{2} - \partial_{3}^{2})((i\xi')^{\perp}\cdot\widehat{\bf v}_{h}^{(2)}) + \Omega (i\xi'\cdot\widehat{\bf v}_{h}^{(2)}) = 0. 
\end{equation*}
Using the incompressibility condition, we get 
\begin{equation}\label{0612-10}
  (|\xi'|^{2} - \partial_{3}^{2})((i\xi')^{\perp}\cdot\widehat{\bf v}_{h}^{(2)}) - \Omega\partial_{3}\widehat{v}_{3}^{(2)} = 0. 
\end{equation}
Applying $(|\xi'|^{2} - \partial_{3}^{2})$ $(= \mathcal{F}_{\mathbb{R}^{2}}[(-\Delta)])$ to the second equation in \eqref{0612-7}, thanks to \eqref{0612-8}, we obtain 
\begin{equation}\label{0612-9}
  (|\xi'|^{2} - \partial_{3}^{2})^{2}\widehat{v}_{3}^{(2)} - \Omega\partial_{3}((i\xi')^{\perp}\cdot\widehat{\bf v}_{h}^{(2)}) = 0.
\end{equation}
Thus, applying $(|\xi'|^{2} - \partial_{3}^{2})$ to \eqref{0612-9} once again, using \eqref{0612-10}, 
we have an equation involving only $\widehat{v}_{3}^{(2)}$. 
\begin{equation}\label{0612-11}
  (|\xi'|^2 - \partial_{3}^{2})^{3}\widehat{v}_{3}^{(2)} -{|\Omega|}^2 \partial_{3}^{2} \widehat{v}_{3}^{(2)} =0. 
\end{equation}
To obtain an explicit formula for the sixth-order ordinary differential equation \eqref{0612-11}, 
we consider the following characteristic equation. 
\begin{equation}\label{0612-12}
  (|\xi'|^{2} - \lambda^{2})^{3} - |\Omega|^{2}\lambda^{2} = 0. 
\end{equation}
The algebraic equation \eqref{0612-12} has already appeared in the derivation of the explicit formula for ${\bf w}^{(2)}$, 
where $\lambda_{|\Omega|, k}$ $(k = 1,2,3)$ denoted a root of \eqref{0612-12}. 
Without loss of generality, we may assume that $\Re(\lambda_{|\Omega|, k})>0$ $(k = 1,2,3)$ 
(note that \eqref{0612-12} has only real coefficients). 
Since \eqref{0612-12} contains only even powers of the unknown, $-\lambda_{|\Omega|, k}$ $(k = 1,2,3)$ are also roots of \eqref{0612-12}. 
However, when deriving the explicit formula for \eqref{0612-11}, 
we restrict our attention to solutions ${\bf v}^{(2)}$ that decay as $x_{3} \to +\infty$. 
Under this condition, the root $-\lambda_{|\Omega|, k}$ can be discarded. 
In \cite{Da_Pr_2014}, this restriction on the characteristic roots is justified under an appropriate integrability condition. 
Also in the present framework, thanks to the estimate for ${\bf v}^{(2)}$ (see Proposition~\ref{0618-4} below), 
we obtain $L^{r}$ $(r < \infty)$ integrability with respect to $x_{3}$. 
Furthermore, all coefficients in \eqref{0612-12} are real. 
It is easy to see that a real root exists, and 
the complex conjugate $\bar{\lambda}_{|\Omega|, k}$ of $\lambda_{|\Omega|, k}$ $(k = 1,2,3)$ is also a root of \eqref{0612-12}. 
Therefore, we can assume that 
\begin{equation}\label{0916-2}
  \lambda_{|\Omega|, 1} \in \mathbb{R} 
  \text{ with }
  \lambda_{|\Omega|, 1}>0, 
\end{equation}
\begin{equation}\label{0916-3}
  \lambda_{|\Omega|, 2}, \lambda_{|\Omega|, 3} \in \mathbb{C} 
  \text{ with }
  \Re(\lambda_{|\Omega|, 2}), 
  \Re(\lambda_{|\Omega|, 3}) > 0 
  \text{ and } 
  \lambda_{|\Omega|, 3} 
  = \bar{\lambda}_{|\Omega|, 2}. 
\end{equation}
Consequently, there exist constants $C_{k}$ $(k = 1,2,3)$ depending on $|\xi'|$, $\Omega$ and the boundary data $\widehat{\bf a}(\xi') - \widehat{\bf w}^{(2)}(\xi', 0)$ 
such that we can write 
\begin{equation}\label{0612-15}
  \widehat{v}_{3}^{(2)}(\xi', x_{3}) 
  = \sum_{k = 1}^{3}C_{k}e^{-x_{3}\lambda_{|\Omega|, k}}. 
\end{equation}
By the incompressibility condition, we have 
\begin{equation}\label{0612-13}
  i\xi'\cdot\widehat{\bf v}_{h}^{(2)}(\xi', x_{3}) 
  = -\partial_{3}\widehat{v}_{3}^{(2)}(\xi', x_{3}) 
  = \sum_{k = 1}^{3}\lambda_{|\Omega|, k}C_{k}e^{-x_{3}\lambda_{|\Omega|, k}}. 
\end{equation}
Integrating \eqref{0612-9} with respect to $x_{3}$ over $[x_{3}, \infty)$, it follows that 
\begin{equation}\label{0612-14}
  (i\xi')^{\perp}\cdot\widehat{\bf v}_{h}^{(2)}(\xi', x_{3}) 
  = - \sum_{k = 1}^{3}\frac{(|\xi'|^{2} - \lambda_{|\Omega|, k}^{2})^{2}}{\Omega \lambda_{|\Omega|, k}}C_{k}e^{-x_{3}\lambda_{|\Omega|, k}}. 
\end{equation}
Using \eqref{0612-13} and \eqref{0612-14}, we get 
\begin{equation*}
  \widehat{\bf v}_{h}^{(2)}(\xi', x_{3}) 
  = -\frac{1}{|\xi'|^{2}}\sum_{k = 1}^{3}\left(i\xi'\lambda_{|\Omega|, k} - (i\xi')^{\perp}\frac{(|\xi'|^{2} - \lambda_{|\Omega|, k}^{2})^{2}}{\Omega \lambda_{|\Omega|, k}}\right)C_{k}e^{-x_{3}\lambda_{|\Omega|, k}}. 
\end{equation*}
We derive the constants $C_{k}$. 
For notational convenience, we denote 
\begin{equation*}
  {\bf b}(x') := {\bf a}(x') - {\bf w}^{(2)}(x', 0). 
\end{equation*}
Note that $b_{3}(\cdot) = a_{3}(\cdot) - w^{(2)}(\cdot, 0) = 0$. 
From \eqref{0612-15}, \eqref{0612-13} and \eqref{0612-14}, and their boundary condition, $C_{k}$ satisfy 
\begin{equation*}
  \begin{pmatrix}
    \lambda_{|\Omega|, 1} & \lambda_{|\Omega|, 2} & \lambda_{|\Omega|, 3} \\ 
    \displaystyle\frac{({|\xi'|}^{2} - \lambda_{|\Omega|, 1}^2)^{2}}{\Omega\lambda_{|\Omega|, 1}} & \displaystyle\frac{({|\xi'|}^{2} - \lambda_{|\Omega|, 2}^2)^{2}}{\Omega\lambda_{|\Omega|, 2}} & \displaystyle\frac{({|\xi'|}^{2} - \lambda_{|\Omega|, 3}^2)^{2}}{\Omega\lambda_{|\Omega|, 3}} \\ 
    1 & 1 & 1 \\ 
  \end{pmatrix}
  \begin{pmatrix}
    C_{1} \\ 
    C_{2} \\ 
    C_{3}
  \end{pmatrix}
  = 
  \begin{pmatrix}
    i\xi'\cdot\widehat{\bf b}_{h}(\xi')\\ 
    -(i\xi')^{\perp}\cdot\widehat{\bf b}_{h}(\xi')\\ 
    0
  \end{pmatrix}
  . 
\end{equation*}
We denote by ${\bf A}$ the matrix on the left-hand side of the above equation. 
For the inverse matrix ${\bf A}^{-1}$, the following holds. 
\begin{lem}\label{0612-16}
  Let $\Omega\in\mathbb{R}\setminus\{0\}$ and $\xi'\in \mathbb{R}^{2}\setminus\{0\}$. 
  Then the inverse matrix ${\bf A}^{-1}$ of ${\bf A}$ is given by 
  \begin{align}\label{0612-17}
    {\bf A}^{-1} 
    =& \frac{1}{\det{\bf A}}\notag\\  
    &\times 
    \begin{pmatrix}
      A_{1}(|\xi'| + \lambda_{|\Omega|, 1})(\lambda_{|\Omega|, 2} - \lambda_{|\Omega|, 3}) & -(\lambda_{|\Omega|, 2} - \lambda_{|\Omega|, 3}) & -|\xi'|A_{1}(\lambda_{|\Omega|, 2}^{2} - \lambda_{|\Omega|, 3}^{2}) \\ 
      A_{2}(|\xi'| + \lambda_{|\Omega|, 2})(\lambda_{|\Omega|, 3} - \lambda_{|\Omega|, 1}) & -(\lambda_{|\Omega|, 3} - \lambda_{|\Omega|, 1}) & -|\xi'|A_{2}(\lambda_{|\Omega|, 3}^{2} - \lambda_{|\Omega|, 1}^{2}) \\ 
      A_{3}(|\xi'| + \lambda_{|\Omega|, 3})(\lambda_{|\Omega|, 1} - \lambda_{|\Omega|, 2}) & -(\lambda_{|\Omega|, 1} - \lambda_{|\Omega|, 2}) & -|\xi'|A_{3}(\lambda_{|\Omega|, 1}^{2} - \lambda_{|\Omega|, 2}^{2}) \\ 
    \end{pmatrix}
    ,
  \end{align}
  where 
  \begin{equation*}
    A_{k} := \frac{(|\xi'|^{2} - \lambda_{|\Omega|, k}^{2})}{\Omega\lambda_{|\Omega|, k}}, 
  \end{equation*}
  and 
  \begin{equation*}
    \det{\bf A} 
    = \frac{1}{\Omega}(\lambda_{|\Omega|, 1} - \lambda_{|\Omega|, 2})(\lambda_{|\Omega|, 2} - \lambda_{|\Omega|, 3})(\lambda_{|\Omega|, 3} - \lambda_{|\Omega|, 1})(|\xi'| + \lambda_{|\Omega|, 1} + \lambda_{|\Omega|, 2} + \lambda_{|\Omega|, 3}). 
  \end{equation*}
\end{lem}
Lemma~\ref{0612-16} follows from straightforward computations and the fact that $\lambda_{|\Omega|, k}$ is a root of the algebraic equation \eqref{0612-12}. 
For the derivation of $\det A$, we refer the reader to \cite{Da_Pr_2014}, Appendix A.2. 

From the preceding discussion, we obtain an explicit formula for ${\bf v}^{(2)}$. 
\begin{prop}\label{0626-12}
  Let $\Omega\in\mathbb{R}\setminus\{0\}$. 
  The solution ${\bf v}^{(2)} = (v_{1}^{(2)}, v_{2}^{(2)}, v_{3}^{(2)}) =: ({\bf v}_{h}^{(2)}, v_{3}^{(2)})$ of \eqref{0611-15} is given by 
  \begin{align}\label{0612-19}
    {\bf v}_{h}^{(2)}(x', x_{3}) 
    =& -\nabla'(-\Delta_{\mathbb{R}^{2}})^{-1}\sum_{k = 1}^{3}\mathcal{F}_{\mathbb{R}^{2}}^{-1}\left[\lambda_{|\Omega|, k}C_{k}e^{-x_{3}\lambda_{|\Omega|, k}}\right](x')\notag\\ 
    &+ (\nabla')^{\perp}(-\Delta_{\mathbb{R}^{2}})^{-1}\sum_{k = 1}^{3}\mathcal{F}_{\mathbb{R}^{2}}^{-1}\left[\frac{(|\xi'|^{2} - \lambda_{|\Omega|, k}^{2})^{2}}{\Omega \lambda_{|\Omega|, k}}C_{k}e^{-x_{3}\lambda_{|\Omega|, k}}\right](x'),\notag\\ 
    v_{3}^{(2)}(x', x_{3}) 
    =& \sum_{k = 1}^{3}\mathcal{F}_{\mathbb{R}^{2}}^{-1}\left[C_{k}e^{-x_{3}\lambda_{|\Omega|, k}}\right](x'), 
  \end{align}
  where $C_{k}$ satisfies 
  \begin{equation}\label{0613-11}
    \begin{pmatrix}
      C_{1} \\ 
      C_{2} \\ 
      C_{3}
    \end{pmatrix}
    = 
    {\bf A}^{-1}
    \begin{pmatrix}
      i\xi'\cdot(\widehat{\bf a}_{h}(\xi') - \widehat{\bf w}_{h}^{(2)}(\xi',0))\\ 
      -(i\xi')^{\perp}\cdot(\widehat{\bf a}_{h}(\xi') - \widehat{\bf w}_{h}^{(2)}(\xi',0))\\ 
      0
    \end{pmatrix}
    , 
  \end{equation}
  and ${\bf A}^{-1}$ is given by \eqref{0612-17}. 
\end{prop}

\subsection{Estimates for ${\bf w}^{(2)}$ and ${\bf v}^{(2)}$}\label{0623-3} 
In order to estimate ${\bf w}^{(2)}$ and ${\bf v}^{(2)}$, given in \eqref{0612-18} and \eqref{0612-19}, respectively, 
we need to understand the detailed structure and asymptotic behavior of $\lambda_{|\Omega|, k}$ $(k = 1,2,3)$ with respect to $\xi'$ for fixed $\Omega$. 
Recall that $\lambda_{|\Omega|, k}$ is a root of the following algebraic equation, 
\begin{equation*}
  (|\xi'|^{2} - \lambda^{2})^{3} - |\Omega|^{2}\lambda^{2} = 0. 
\end{equation*}
By the classical formula for cubic equations, $\lambda_{|\Omega|, k}$ can be written as follows. 
\begin{align}\label{0625-9}
  \lambda_{|\Omega|, k}^{2} 
  =& {|\xi'|}^{2} 
   + \beta_{k}\left(\frac{|\Omega|^2}{2}\left(-{|\xi'|}^{2} + \left({|\xi'|}^{4} + \frac{4}{27}{|\Omega|}^{2}\right)^{\frac{1}{2}}\right)\right)^{\frac{1}{3}}\notag\\ 
   &- \bar{\beta_{k}}\left(\frac{|\Omega|^2}{2}\left({|\xi'|}^{2} + \left({|\xi'|}^{4} + \frac{4}{27}{|\Omega|}^{2}\right)^{\frac{1}{2}}\right)\right)^{\frac{1}{3}}, 
\end{align}
where $\beta_{1}:= 1$, $\beta_{2} := e^{\frac{2}{3}i\pi}$ and $\beta_{3} := e^{\frac{4}{3}i\pi}$. 
Moreover, we can write 
\begin{align}\label{0613-2}
  \lambda_{|\Omega|, k}^{2} 
  =& |\Omega|\left(\frac{|\xi'|}{|\Omega|^{\frac{1}{2}}}\right)^{2} 
   + |\Omega|2^{-\frac{1}{3}}\beta_{k}\left(-\left(\frac{|\xi'|}{|\Omega|^{\frac{1}{2}}}\right)^{2} + \left(\left(\frac{|\xi'|}{|\Omega|^{\frac{1}{2}}}\right)^{4} + \frac{4}{27}\right)^{\frac{1}{2}}\right)^{\frac{1}{3}}\notag\\ 
   &- |\Omega|2^{-\frac{1}{3}}\bar{\beta}_{k}\left(\left(\frac{|\xi'|}{|\Omega|^{\frac{1}{2}}}\right)^{2} + \left(\left(\frac{|\xi'|}{|\Omega|^{\frac{1}{2}}}\right)^{4} + \frac{4}{27}\right)^{\frac{1}{2}}\right)^{\frac{1}{3}}, 
\end{align}
which shows that the leading-order term of $\lambda_{|\Omega|, k}$ varies with the magnitude of $|\xi'|/|\Omega|^{\frac{1}{2}}$. 
Motivated by the above observations, we decompose the frequency domain as follows: 
\begin{equation*}
  \{\xi'\in\mathbb{R}^{2}|\ |\xi'|> C_{0}|\Omega|^{\frac{1}{2}}\},\ 
  \{\xi'\in\mathbb{R}^{2}|\ C_{0}^{-1}|\Omega|^{\frac{1}{2}}\leq |\xi'|\leq C_{0}|\Omega|^{\frac{1}{2}}\}
\end{equation*}
and 
\begin{equation*}
  \{\xi'\in\mathbb{R}^{2}|\ |\xi'| < C_{0}^{-1}|\Omega|^{\frac{1}{2}}\}, 
\end{equation*}
where $C_{0} \geq 1$ is a sufficiently large constant to be fixed later. 
In this setting, when $|\xi'| > C_{0}|\Omega|^{\frac{1}{2}}$, we instead employ different representations ${\bf w}^{(1)}$ and ${\bf v}^{(1)}$, 
in which the Coriolis term is treated as a perturbation, 
since the required estimates can be derived more directly. 
Therefore, we only consider the cases of $C_{0}^{-1}|\Omega|^{\frac{1}{2}}\leq |\xi'|\leq C_{0}|\Omega|^{\frac{1}{2}}$ and $|\xi'| < C_{0}^{-1}|\Omega|^{\frac{1}{2}}$ in what follows. 

We can obtain the following lemma. 
\begin{lem}\label{0613-7}
  Let $\Omega\in\mathbb{R}\setminus\{0\}$. 
  \begin{enumerate}
    \item\label{0613-1} Let $C_{0} \geq 1$. 
    If $\xi'\in \mathbb{R}^{2}$ satisfies $C_{0}^{-1}|\Omega|^{\frac{1}{2}}\leq |\xi'|\leq C_{0}|\Omega|^{\frac{1}{2}}$, 
    then there exists a constant $C>0$ such that for $k = 1,2,3$, we have 
    \begin{equation}\label{0613-3}
      C^{-1}|\xi'|\leq |\lambda_{|\Omega|, k}| \leq C|\xi'|. 
    \end{equation}
    \item\label{0613-4} There exists a sufficiently large constant $C_{0} \geq 1$ such that 
    if $\xi'\in \mathbb{R}^{2}$ satisfies $|\xi'|< C_{0}^{-1}|\Omega|^{\frac{1}{2}}$, then we have the asymptotic expansion 
    \begin{align*}
      \lambda_{|\Omega|, 1} 
      &= \frac{|\xi'|^{3}}{|\Omega|}\left(1 - \frac{3}{2}\frac{|\xi'|^{4}}{|\Omega|^{2}} + \mathcal{O}\left(\frac{|\xi'|^{8}}{|\Omega|^{4}}\right)\right),\\ 
      \lambda_{|\Omega|, 2} 
      &= e^{\frac{1}{4}i\pi}|\Omega|^{\frac{1}{2}}\left(1 - \frac{3}{4}i\frac{|\xi'|^{2}}{|\Omega|} + \frac{15}{32}\frac{|\xi'|^{4}}{|\Omega|^{2}} + \mathcal{O}\left(\frac{|\xi'|^{6}}{|\Omega|^{3}}\right)\right),\\ 
      \lambda_{|\Omega|, 3} 
      &= e^{-\frac{1}{4}i\pi}|\Omega|^{\frac{1}{2}}\left(1 + \frac{3}{4}i\frac{|\xi'|^{2}}{|\Omega|} + \frac{15}{32}\frac{|\xi'|^{4}}{|\Omega|^{2}} + \mathcal{O}\left(\frac{|\xi'|^{6}}{|\Omega|^{3}}\right)\right). 
    \end{align*}
  \end{enumerate}
\end{lem}
\begin{proof}
  The upper bound in \eqref{0613-1} follows immediately from \eqref{0613-2}. 
  For the lower bound, we can write 
  \begin{equation*}
    \lambda_{|\Omega|, k}^{2} 
    = \frac{(|\xi'|^{2} - \lambda_{|\Omega|, k}^{2})^{3}}{|\Omega|^{2}}, 
  \end{equation*}
  and 
  \begin{align*}
    |\xi'|^{2} - \lambda_{|\Omega|, k}^{2}
    =& \bar{\beta_{k}}\left(\frac{|\Omega|^2}{2}\left({|\xi'|}^{2} + \left({|\xi'|}^{4} + \frac{4}{27}{|\Omega|}^{2}\right)^{\frac{1}{2}}\right)\right)^{\frac{1}{3}}\\ 
    &- \beta_{k}\left(\frac{|\Omega|^2}{2}\left(-{|\xi'|}^{2} + \left({|\xi'|}^{4} + \frac{4}{27}{|\Omega|}^{2}\right)^{\frac{1}{2}}\right)\right)^{\frac{1}{3}}. 
  \end{align*}
  It is clear that if $|\xi'|^{2} - \lambda_{|\Omega|, k}^{2} = 0$, then $\xi' = 0$. 
  Thus, $|\lambda_{|\Omega|, k}| > 0$ for any $|\xi'| > 0$. 
  Moreover, from \eqref{0613-2}, we have 
  \begin{equation*}
    \frac{1}{|\Omega|^{\frac{1}{2}}}|\lambda_{k, |\Omega|}| > 0 
    \quad \text{for any } \frac{|\xi'|}{|\Omega|^{\frac{1}{2}}} > 0. 
  \end{equation*}
  By the assumption $C_{0}^{-1}\leq |\xi'|/|\Omega|^{\frac{1}{2}}\leq C_{0}$, we have 
  \begin{equation*}
    |\lambda_{k,|\Omega|}|\geq C^{-1}|\Omega|^{\frac{1}{2}} \geq C^{-1}|\xi'|. 
  \end{equation*}

  The expansions of $\lambda_{|\Omega|, 2}$ and $\lambda_{|\Omega|, 3}$ in \eqref{0613-4} follow from a Taylor expansion around $|\xi'|/|\Omega|^{\frac{1}{2}} = 0$. 
  In the case of $\lambda_{|\Omega|, 1}$, the leading-order terms are canceled, which makes the same approach cumbersome. 
  Hence we use the following fact. 
  $\lambda_{|\Omega|, k}$ $(k = 1,2,3)$ satisfies 
  \begin{equation*}
    (\lambda_{|\Omega|, 1}^{2} - \lambda^{2})(\lambda_{|\Omega|, 2}^{2} - \lambda^{2})(\lambda_{|\Omega|, 3}^{2} - \lambda^{2}) 
    = (|\xi'|^{2} - \lambda^{2})^{3} - |\Omega|^{2}\lambda^{2}, 
  \end{equation*}
  which implies 
  \begin{equation*}
    \lambda_{|\Omega|, 1}^{2}\lambda_{|\Omega|, 2}^{2}\lambda_{|\Omega|, 3}^{2} = |\xi'|^{6}. 
  \end{equation*}
  Moreover, since $\lambda_{|\Omega|, 1}, \Re(\lambda_{|\Omega|, 2}), \Re(\lambda_{|\Omega|, 3}) > 0$ (see \eqref{0916-2} and \eqref{0916-3}), we have 
  \begin{equation*}
    \lambda_{|\Omega|, 1}\lambda_{|\Omega|, 2}\lambda_{|\Omega|, 3} = |\xi'|^{3}. 
  \end{equation*}
  Using this fact, we obtain the expansion of $\lambda_{|\Omega|, 1}$. 
\end{proof}

By the preceding lemma, we have the following for 
\begin{equation*}
  B_{k} = \left(\prod_{\alpha = 1, \alpha \ne k}^{3}(\lambda_{|\Omega|, \alpha}^{2}-\lambda_{|\Omega|, k}^2)\right)^{-1}, 
\end{equation*}
which appears in the explicit formulas for $({\bf w}^{2}, {\bf v}^{2})$, \eqref{0612-18} and \eqref{0612-19}. 

\begin{cor}\label{0916-1}
  Let $\Omega\in\mathbb{R}\setminus\{0\}$. 
  \begin{enumerate}
    \item Let $C_{0} \geq 1$. 
    If $\xi'\in \mathbb{R}^{2}$ satisfies $C_{0}^{-1}|\Omega|^{\frac{1}{2}}\leq |\xi'|\leq C_{0}|\Omega|^{\frac{1}{2}}$, 
    then there exists a constant $C>0$ such that for $k = 1,2,3$, we have 
    \begin{equation*}
      C^{-1}\frac{1}{|\xi'|^{4}}\leq |B_{k}| \leq C\frac{1}{|\xi'|^{4}}. 
    \end{equation*}
    \item There exists a sufficiently large constant $C_{0} \geq 1$ such that 
    if $\xi'\in \mathbb{R}^{2}$ satisfies $|\xi'|< C_{0}^{-1}|\Omega|^{\frac{1}{2}}$, then we have the asymptotic expansion 
    \begin{align*}
      B_{1} 
      &= \frac{1}{|\Omega|^{2}}\left(1 - 3\frac{|\xi'|^{4}}{|\Omega|^{2}} + \mathcal{O}\left(\frac{|\xi'|^{8}}{|\Omega|^{4}}\right)\right),\\ 
      B_{2} 
      &= -\frac{1}{2|\Omega|^{2}}\left(1 + \frac{3}{2}i\frac{|\xi'|^{2}}{|\Omega|} + \mathcal{O}\left(\frac{|\xi'|^{4}}{|\Omega|^{2}}\right)\right),\\ 
      B_{3} 
      &= -\frac{1}{2|\Omega|^{2}}\left(1 - \frac{3}{2}i\frac{|\xi'|^{2}}{|\Omega|} + \mathcal{O}\left(\frac{|\xi'|^{4}}{|\Omega|^{2}}\right)\right). 
    \end{align*}
  \end{enumerate}
\end{cor}

As a consequence of Lemma~\ref{0613-7}, we obtain estimates for the Fourier multiplier operator with symbol $e^{-x_{3}\lambda_{|\Omega|, k}}$. 
\begin{lem}\label{0613-8}
  Let $x_{3} > 0$, $j \in \mathbb{Z}$, $\Omega\in\mathbb{R}\setminus\{0\}$ and $1\leq p \leq \infty$. 
  \begin{enumerate}
    \item\label{0613-5} Let $C_{0} \geq 1$. 
    If $C_{0}^{-1}|\Omega|^{\frac{1}{2}}\leq 2^{j} \leq C_{0}|\Omega|^{\frac{1}{2}}$, 
    then for each $k = 1, 2, 3$, there exist constants $C, c > 0$ such that 
    for any $f$ with $\phi_{j}*f \in L^{p}(\mathbb{R}^{2})$, we have 
    \begin{equation*}
      {\left\|\mathcal{F}_{\mathbb{R}^{2}}^{-1}\left[\widehat{\phi}_{j}(\xi')e^{-x_{3}\lambda_{|\Omega|, k}}\widehat{f}(\xi')\right]\right\|}_{L^{p}} 
      \leq C e^{-cx_{3}2^{j}}{\|\phi_{j}*f\|}_{L^{p}}, 
    \end{equation*}
    \item\label{0613-6} There exists a sufficiently large constant $C_{0} \geq 1$ such that 
    if $2^{j} < C_{0}^{-1}|\Omega|^{\frac{1}{2}}$, 
    then there exist constants $C, c > 0$ such that 
    for any $f$ with $\phi_{j}*f \in L^{p}(\mathbb{R}^{2})$, we have 
    \begin{align*}
      {\left\|\mathcal{F}_{\mathbb{R}^{2}}^{-1}\left[\widehat{\phi}_{j}(\xi')e^{-x_{3}\lambda_{|\Omega|, 1}}\widehat{f}(\xi')\right]\right\|}_{L^{p}} 
      &\leq C e^{-cx_{3}\frac{2^{3j}}{|\Omega|}}{\|\phi_{j}*f\|}_{L^{p}}, \\ 
      {\left\|\mathcal{F}_{\mathbb{R}^{2}}^{-1}\left[\widehat{\phi}_{j}(\xi')e^{-x_{3}\lambda_{|\Omega|, 2}}\widehat{f}(\xi')\right]\right\|}_{L^{p}}
      &\leq C e^{-cx_{3}|\Omega|^{\frac{1}{2}}}{\|\phi_{j}*f\|}_{L^{p}}, \\
      {\left\|\mathcal{F}_{\mathbb{R}^{2}}^{-1}\left[\widehat{\phi}_{j}(\xi')e^{-x_{3}\lambda_{|\Omega|, 3}}\widehat{f}(\xi')\right]\right\|}_{L^{p}} 
      &\leq C e^{-cx_{3}|\Omega|^{\frac{1}{2}}}{\|\phi_{j}*f\|}_{L^{p}}. 
    \end{align*}
  \end{enumerate}
\end{lem}
\begin{proof}
  We only prove \eqref{0613-5}, since the proof of \eqref{0613-6} is similar. 
  It follows from \eqref{0613-2} that there exists a function $\Psi_{k} : \mathbb{R} \to \mathbb{C}$ $(k = 1, 2, 3)$ satisfying 
  \begin{equation*}
    \lambda_{|\Omega|, k} = |\Omega|^{\frac{1}{2}}\Psi_{k}\left(\frac{|\xi'|}{|\Omega|^{\frac{1}{2}}}\right). 
  \end{equation*}
  For each $j\in \mathbb{Z}$, we set $\widetilde{\phi}_{j}:= \phi_{j-1} + \phi_{j} + \phi_{j+1}$. 
  Then we have $\widehat{\phi}_{j}(\xi') = \widehat{\widetilde{\phi}}_{j}(\xi')\widehat{\phi}_{j}(\xi')$. 
  Thus, using Young's inequality, we get 
  \begin{equation*}
    {\left\|\mathcal{F}_{\mathbb{R}^{2}}^{-1}\left[\widehat{\phi}_{j}(\xi')e^{-x_{3}\lambda_{|\Omega|, k}}\widehat{f}(\xi')\right]\right\|}_{L^{p}} 
    \leq C {\left\|\mathcal{F}_{\mathbb{R}^{2}}^{-1}\left[\widehat{\widetilde{\phi}}_{j}(\xi')e^{-x_{3}|\Omega|^{\frac{1}{2}}\Psi_{k}(|\xi'|/|\Omega|^{\frac{1}{2}})}\right]\right\|}_{L^{1}}
    {\|\phi_{j}*f\|}_{L^{p}}. 
  \end{equation*}
  Without loss of generality, we may choose the dyadic partition of unity defining the Besov spaces so that 
  \begin{equation*}
    \phi_{j}(\xi') = \phi_{0}(2^{-j}\xi'). 
  \end{equation*}
  Since $\widetilde{\phi}_{j}(\xi') = \widetilde{\phi}_{0}(2^{-j}\xi')$, 
  it follows by a change of variables that 
  \begin{equation*}
    {\left\|\mathcal{F}_{\mathbb{R}^{2}}^{-1}\left[\widehat{\widetilde{\phi}}_{j}(\xi')e^{-x_{3}|\Omega|^{\frac{1}{2}}\Psi_{k}(|\xi'|/|\Omega|^{\frac{1}{2}})}\right]\right\|}_{L^{1}} 
    = {\left\|\mathcal{F}_{\mathbb{R}^{2}}^{-1}\left[\widehat{\widetilde{\phi}}_{0}(\xi')e^{-x_{3}|\Omega|^{\frac{1}{2}}\Psi_{k}(2^{j}|\xi'|/|\Omega|^{\frac{1}{2}})}\right]\right\|}_{L^{1}}
  \end{equation*}
  Also, we have 
  \begin{align*}
    &{\left\|\mathcal{F}_{\mathbb{R}^{2}}^{-1}\left[\widehat{\widetilde{\phi}}_{0}(\xi')e^{-x_{3}|\Omega|^{\frac{1}{2}}\Psi_{k}(2^{j}|\xi'|/|\Omega|^{\frac{1}{2}})}\right]\right\|}_{L^{1}}\\ 
    =& {\left\|(1 + |x'|^{2})^{-2}(1 + |x'|^{2})^{2}\mathcal{F}_{\mathbb{R}^{2}}^{-1}\left[\widehat{\widetilde{\phi}}_{0}(\xi')e^{-x_{3}|\Omega|^{\frac{1}{2}}\Psi_{k}(2^{j}|\xi'|/|\Omega|^{\frac{1}{2}})}\right]\right\|}_{L^{1}}\\ 
    =& {\left\|(1 + |x'|^{2})^{-2}\mathcal{F}_{\mathbb{R}^{2}}^{-1}\left[(1 - \Delta_{\xi'})^{2}\widehat{\widetilde{\phi}}_{0}(\xi')e^{-x_{3}|\Omega|^{\frac{1}{2}}\Psi_{k}(2^{j}|\xi'|/|\Omega|^{\frac{1}{2}})}\right]\right\|}_{L^{1}}\\ 
    \leq&C{\left\|\mathcal{F}_{\mathbb{R}^{2}}^{-1}\left[(1 - \Delta_{\xi'})^{2}\widehat{\widetilde{\phi}}_{0}(\xi')e^{-x_{3}|\Omega|^{\frac{1}{2}}\Psi_{k}(2^{j}|\xi'|/|\Omega|^{\frac{1}{2}})}\right]\right\|}_{L^{\infty}}\\ 
    \leq&C\int_{\mathbb{R}^{2}}\left|(1 - \Delta_{\xi'})^{2}\widehat{\widetilde{\phi}}_{0}(\xi')e^{-x_{3}|\Omega|^{\frac{1}{2}}\Psi_{k}(2^{j}|\xi'|/|\Omega|^{\frac{1}{2}})}\right|~{\rm d}\xi', 
  \end{align*}
  where $\Delta_{\xi'} := \partial_{\xi_{1}}^{2} + \partial_{\xi_{2}}^{2}$. 
  By $\text{supp }\widehat{\phi}_{0} = \{\xi'\in \mathbb{R}^{2}|\ 2^{-1}\leq |\xi'|\leq 2\}$ and Lemma~\ref{0613-7} \eqref{0613-1}, 
  the chain rule implies that for any multi-index $\beta\in (\mathbb{N}\cup\{0\})^{2}$ with $|\beta|\leq 4$, we have 
  \begin{align*}
    \left|\partial_{\xi'}^{\beta}e^{-x_{3}|\Omega|^{\frac{1}{2}}\Psi_{k}(2^{j}|\xi'|/|\Omega|^{\frac{1}{2}})}\right| 
    \leq& C \sum_{0 \leq n\leq |\beta|}(x_{3}2^{j})^{n}|\Omega|^{\frac{n}{2}}\left|(\Psi_{k})^{(n)}(2^{j}|\xi'|/|\Omega|^{\frac{1}{2}})\right|e^{-cx_{3}2^{j}}\\ 
    \leq& C \sum_{0 \leq n\leq |\beta|}(x_{3}2^{j})^{n}2^{nj}\left|(\Psi_{k})^{(n)}(2^{j}|\xi'|/|\Omega|^{\frac{1}{2}})\right|e^{-cx_{3}2^{j}}\\ 
    \leq& C\sum_{0 \leq n\leq |\beta|}(x_{3}2^{j})^{n}e^{-cx_{3}2^{j}}. 
  \end{align*}
  Therefore, we get 
  \begin{align*}
    &\int_{\mathbb{R}^{2}}\left|(1 - \Delta_{\xi'})^{2}\widehat{\widetilde{\phi}}_{0}(\xi')e^{-x_{3}|\Omega|^{\frac{1}{2}}\Psi_{k}(2^{j}|\xi'|/|\Omega|^{\frac{1}{2}})}\right|~{\rm d}\xi'\\ 
    \leq& C \int_{\mathbb{R}^{2}}\sum_{\beta\in(\mathbb{N}\cup\{0\})^{2}, |\beta|\leq4}\left|\partial_{\xi'}^{\beta}\widehat{\widetilde{\phi}}_{0}(\xi')\right|(1 + x_{3}2^{j})^{4}e^{-cx_{3}2^{j}}~{\rm d}\xi'\\ 
    \leq& C e^{-cx_{3}2^{j}}. 
  \end{align*}
  This yields \eqref{0613-5}. 
  For \eqref{0613-6}, we use Lemma~\ref{0613-7} \eqref{0613-4}.  
\end{proof}

By Lemma~\ref{0613-7}, Corollary~\ref{0916-1} and Lemma~\ref{0613-8}, we obtain estimates for the multiplier $K_{l}^{(2)}$ $(l = 1, 2, \cdots , 8)$ given by \eqref{0612-4} and \eqref{0612-5}. 

\begin{lem}\label{0613-10}
  Let $x_{3} > 0$, $\Omega\in\mathbb{R}\setminus\{0\}$ and $1\leq p \leq \infty$. 
  \begin{enumerate}
    \item\label{0616-3} Let $C_{0} \geq 1$. 
    If $C_{0}^{-1}|\Omega|^{\frac{1}{2}}\leq 2^{j} \leq C_{0}|\Omega|^{\frac{1}{2}}$, 
    then there exist constants $C, c > 0$ such that 
    for any $f$ with $\phi_{j}*f \in L^{p}(\mathbb{R}^{2})$, we have 
    \begin{align*}
      {\left\|\mathcal{F}_{\mathbb{R}^{2}}^{-1}\left[\widehat{\phi}_{j}(\xi')K_{1}^{(2)}(\xi', x_{3})\widehat{f}(\xi')\right]\right\|}_{L^{p}}
      &\leq C2^{-4j}e^{-cx_{3}2^{j}}{\|\phi_{j}*f\|}_{L^{p}}, \notag\\ 
      {\left\|\mathcal{F}_{\mathbb{R}^{2}}^{-1}\left[\widehat{\phi}_{j}(\xi')K_{2}^{(2)}(\xi', x_{3})\widehat{f}(\xi')\right]\right\|}_{L^{p}} 
      &\leq C2^{-3j}e^{-cx_{3}2^{j}}{\|\phi_{j}*f\|}_{L^{p}}, \notag\\ 
      {\left\|\mathcal{F}_{\mathbb{R}^{2}}^{-1}\left[\widehat{\phi}_{j}(\xi')K_{3}^{(2)}(\xi', x_{3})\widehat{f}(\xi')\right]\right\|}_{L^{p}} 
      &\leq C2^{-3j}e^{-cx_{3}2^{j}}{\|\phi_{j}*f\|}_{L^{p}}, \notag\\ 
      {\left\|\mathcal{F}_{\mathbb{R}^{2}}^{-1}\left[\widehat{\phi}_{j}(\xi')K_{4}^{(2)}(\xi', x_{3})\widehat{f}(\xi')\right]\right\|}_{L^{p}} 
      &\leq C2^{-2j}e^{-cx_{3}2^{j}}{\|\phi_{j}*f\|}_{L^{p}}, \notag\\ 
      {\left\|\mathcal{F}_{\mathbb{R}^{2}}^{-1}\left[\widehat{\phi}_{j}(\xi')K_{5}^{(2)}(\xi', x_{3})\widehat{f}(\xi')\right]\right\|}_{L^{p}} 
      &\leq C2^{-2j}e^{-cx_{3}2^{j}}{\|\phi_{j}*f\|}_{L^{p}}, \notag\\ 
      {\left\|\mathcal{F}_{\mathbb{R}^{2}}^{-1}\left[\widehat{\phi}_{j}(\xi')K_{6}^{(2)}(\xi', x_{3})\widehat{f}(\xi')\right]\right\|}_{L^{p}} 
      &\leq C2^{-j}e^{-cx_{3}2^{j}}{\|\phi_{j}*f\|}_{L^{p}}, \notag\\ 
      {\left\|\mathcal{F}_{\mathbb{R}^{2}}^{-1}\left[\widehat{\phi}_{j}(\xi')K_{7}^{(2)}(\xi', x_{3})\widehat{f}(\xi')\right]\right\|}_{L^{p}} 
      &\leq C2^{-j}e^{-cx_{3}2^{j}}{\|\phi_{j}*f\|}_{L^{p}}, \notag\\ 
      {\left\|\mathcal{F}_{\mathbb{R}^{2}}^{-1}\left[\widehat{\phi}_{j}(\xi')K_{8}^{(2)}(\xi', x_{3})\widehat{f}(\xi')\right]\right\|}_{L^{p}} 
      &\leq Ce^{-cx_{3}2^{j}}{\|\phi_{j}*f\|}_{L^{p}}. 
    \end{align*}
    \item\label{0613-9} There exists a sufficiently large constant $C_{0} \geq 1$ such that 
    if $2^{j} < C_{0}^{-1}|\Omega|^{\frac{1}{2}}$, 
    then there exist constants $C, c > 0$ such that 
    for any $f$ with $\phi_{j}*f \in L^{p}(\mathbb{R}^{2})$, we have 
    \begin{align*}
      {\left\|\mathcal{F}_{\mathbb{R}^{2}}^{-1}\left[\widehat{\phi}_{j}(\xi')K_{1}^{(2)}(\xi', x_{3})\widehat{f}(\xi')\right]\right\|}_{L^{p}}
      &\leq C\frac{1}{|\Omega|^{2}}(e^{-cx_{3}\frac{2^{3j}}{|\Omega|}} + e^{-cx_{3}|\Omega|^{\frac{1}{2}}}){\|\phi_{j}*f\|}_{L^{p}}, \notag\\ 
      {\left\|\mathcal{F}_{\mathbb{R}^{2}}^{-1}\left[\widehat{\phi}_{j}(\xi')K_{2}^{(2)}(\xi', x_{3})\widehat{f}(\xi')\right]\right\|}_{L^{p}} 
      &\leq C\left(\frac{2^{3j}}{|\Omega|^{3}}e^{-cx_{3}\frac{2^{3j}}{|\Omega|}} + \frac{1}{|\Omega|^{\frac{3}{2}}}e^{-cx_{3}|\Omega|^{\frac{1}{2}}}\right){\|\phi_{j}*f\|}_{L^{p}}, \notag\\ 
      {\left\|\mathcal{F}_{\mathbb{R}^{2}}^{-1}\left[\widehat{\phi}_{j}(\xi')K_{3}^{(2)}(\xi', x_{3})\widehat{f}(\xi')\right]\right\|}_{L^{p}} 
      &\leq C\left(\frac{1}{|\Omega|2^{j}}e^{-cx_{3}\frac{2^{3j}}{|\Omega|}} + \frac{1}{|\Omega|^{\frac{3}{2}}}e^{-cx_{3}|\Omega|^{\frac{1}{2}}}\right){\|\phi_{j}*f\|}_{L^{p}}, \notag\\ 
      {\left\|\mathcal{F}_{\mathbb{R}^{2}}^{-1}\left[\widehat{\phi}_{j}(\xi')K_{4}^{(2)}(\xi', x_{3})\widehat{f}(\xi')\right]\right\|}_{L^{p}} 
      &\leq C\left(\frac{2^{6j}}{|\Omega|^{4}}e^{-cx_{3}\frac{2^{3j}}{|\Omega|}} + \frac{1}{|\Omega|}e^{-cx_{3}|\Omega|^{\frac{1}{2}}}\right){\|\phi_{j}*f\|}_{L^{p}}, \notag\\ 
      {\left\|\mathcal{F}_{\mathbb{R}^{2}}^{-1}\left[\widehat{\phi}_{j}(\xi')K_{5}^{(2)}(\xi', x_{3})\widehat{f}(\xi')\right]\right\|}_{L^{p}} 
      &\leq C\left(\frac{2^{2j}}{|\Omega|^{2}}e^{-cx_{3}\frac{2^{3j}}{|\Omega|}} + \frac{1}{|\Omega|}e^{-cx_{3}|\Omega|^{\frac{1}{2}}}\right){\|\phi_{j}*f\|}_{L^{p}}, \notag\\ 
      {\left\|\mathcal{F}_{\mathbb{R}^{2}}^{-1}\left[\widehat{\phi}_{j}(\xi')K_{6}^{(2)}(\xi', x_{3})\widehat{f}(\xi')\right]\right\|}_{L^{p}} 
      &\leq C\left(\frac{2^{5j}}{|\Omega|^{3}}e^{-cx_{3}\frac{2^{3j}}{|\Omega|}} + \frac{1}{|\Omega|^{\frac{1}{2}}}e^{-cx_{3}|\Omega|^{\frac{1}{2}}}\right){\|\phi_{j}*f\|}_{L^{p}}, \notag\\ 
      {\left\|\mathcal{F}_{\mathbb{R}^{2}}^{-1}\left[\widehat{\phi}_{j}(\xi')K_{7}^{(2)}(\xi', x_{3})\widehat{f}(\xi')\right]\right\|}_{L^{p}} 
      &\leq C\left(\frac{2^{j}}{|\Omega|}e^{-cx_{3}\frac{2^{3j}}{|\Omega|}} + \frac{1}{|\Omega|^{\frac{1}{2}}}e^{-cx_{3}|\Omega|^{\frac{1}{2}}}\right){\|\phi_{j}*f\|}_{L^{p}}, \notag\\ 
      {\left\|\mathcal{F}_{\mathbb{R}^{2}}^{-1}\left[\widehat{\phi}_{j}(\xi')K_{8}^{(2)}(\xi', x_{3})\widehat{f}(\xi')\right]\right\|}_{L^{p}} 
      &\leq C\left(\frac{2^{4j}}{|\Omega|^{2}}e^{-cx_{3}\frac{2^{3j}}{|\Omega|}} + e^{-cx_{3}|\Omega|^{\frac{1}{2}}}\right){\|\phi_{j}*f\|}_{L^{p}}. 
    \end{align*}
  \end{enumerate}
\end{lem}
\begin{proof}
  The proposition follows immediately from Lemma~\ref{0613-7} and Lemma~\ref{0613-8}. 
  We omit the details of the proof. 
  Note that, in \eqref{0613-9}, the expansions of $\lambda_{|\Omega|, 1}$ and $\lambda_{|\Omega|, 2}$, $\lambda_{|\Omega|, 3}$ have different leading-order behaviors. 
  For example, if $|\xi'| < C_{0}^{-1}|\Omega|^{\frac{1}{2}}$, then we observe the following difference: 
  \begin{align*}
    |\xi'|^{2} - \lambda_{|\Omega|, 1}^{2} 
    &= |\xi'|^{2}\left(1 - \frac{|\xi'|^{4}}{|\Omega|^{2}} + \mathcal{O}\left(\frac{|\xi'|^{8}}{|\Omega|^{4}}\right)\right),\\ 
    |\xi'|^{2} - \lambda_{|\Omega|, 2}^{2} 
    &= -i|\Omega|\left(1 - \frac{1}{2}i\frac{|\xi'|^{2}}{|\Omega|} + \mathcal{O}\left(\frac{|\xi'|^{4}}{|\Omega|^{2}}\right)\right),\\ 
    |\xi'|^{2} - \lambda_{|\Omega|, 3}^{2} 
    &= i|\Omega|\left(1 + \frac{1}{2}i\frac{|\xi'|^{2}}{|\Omega|} + \mathcal{O}\left(\frac{|\xi'|^{4}}{|\Omega|^{2}}\right)\right). 
  \end{align*}
  Taking the above observation into account, we can obtain Lemma~\ref{0613-10}. 
\end{proof}

Similarly, we obtain estimates for $C_{k}$ $(k = 1, 2, 3)$ defined by \eqref{0613-11}. 
The proof is omitted. 
\begin{lem}\label{0618-2}
  Let $x_{3} > 0$, $\Omega\in\mathbb{R}\setminus\{0\}$, $j\in\mathbb{Z}$ and $1\leq p \leq \infty$. 
  Let $C_{k}$ $(k = 1, 2, 3)$ satisfy \eqref{0613-11} for given ${\bf a}:\mathbb{R}^{2}\to \mathbb{R}$ and ${\bf w}^{(2)}:\mathbb{R}_{+}^{3}\to \mathbb{R}$ 
  with $\phi_{j}*{\bf a}, \phi_{j}*{\bf w}^{(2)}(\cdot, 0)\in L^{p}(\mathbb{R}^{2})$ and $a_{3}(\cdot) = w_{3}^{(2)}(\cdot, 0) = 0$. 
  \begin{enumerate}
    \item\label{0618-3} Let $C_{0} \geq 1$. 
    If $C_{0}^{-1}|\Omega|^{\frac{1}{2}}\leq 2^{j} \leq C_{0}|\Omega|^{\frac{1}{2}}$, 
    then for each $k = 1, 2, 3$, 
    there exists a constant $C > 0$ independent of ${\bf a}$ and ${\bf w}^{(2)}$ such that we have 
    \begin{equation*}
      {\left\|\mathcal{F}_{\mathbb{R}^{2}}^{-1}\left[\widehat{\phi}_{j}(\xi')C_{k}\right]\right\|}_{L^{p}} 
      \leq C{\|\phi_{j}*({\bf a} - {\bf w}^{(2)}(\cdot, 0))\|}_{L^{p}}. 
    \end{equation*}
    \item\label{0618-8} There exists a sufficiently large constant $C_{0} \geq 1$ such that 
    if $\xi'\in \mathbb{R}^{2}$ satisfies $|\xi'| < C_{0}^{-1}|\Omega|^{\frac{1}{2}}$, we have 
    \begin{align*}
      &
      \begin{pmatrix}
        C_{1} \\ 
        C_{2} \\ 
        C_{3}
      \end{pmatrix}
      \\
      =& 
      \begin{pmatrix}
        \displaystyle\frac{-\sqrt{2}}{2|\Omega|^{\frac{1}{2}}}\Bigg(1 + \mathcal{O}\Bigg(\frac{|\xi'|}{|\Omega|^{\frac{1}{2}}}\Bigg)\Bigg) & \displaystyle\frac{\sqrt{2}\Omega}{2|\Omega|^{\frac{3}{2}}}\Bigg( 1 + \mathcal{O}\Bigg(\frac{|\xi'|}{|\Omega|^{\frac{1}{2}}}\Bigg)\Bigg) & \displaystyle1 + \mathcal{O}\Bigg(\frac{|\xi'|}{|\Omega|^{\frac{1}{2}}}\Bigg)\\ 
        \displaystyle\frac{e^{-\frac{1}{4}i\pi}}{2|\Omega|^{\frac{1}{2}}}\Bigg(1 + \mathcal{O}\Bigg(\frac{|\xi'|}{|\Omega|^{\frac{1}{2}}}\Bigg)\Bigg) & \displaystyle\frac{e^{\frac{5}{4}i\pi}\Omega }{2|\Omega|^{\frac{3}{2}}}\Bigg( 1 + \mathcal{O}\Bigg(\frac{|\xi'|}{|\Omega|^{\frac{1}{2}}}\Bigg)\Bigg) & \displaystyle\frac{e^{\frac{1}{4}i\pi}|\xi'|}{2|\Omega|^{\frac{1}{2}}}\Bigg(1 + \mathcal{O}\Bigg(\frac{|\xi'|}{|\Omega|^{\frac{1}{2}}}\Bigg)\Bigg)\\ 
        \displaystyle\frac{e^{\frac{1}{4}i\pi}}{2|\Omega|^{\frac{1}{2}}}\Bigg(1 + \mathcal{O}\Bigg(\frac{|\xi'|}{|\Omega|^{\frac{1}{2}}}\Bigg)\Bigg) & \displaystyle\frac{e^{\frac{3}{4}i\pi}\Omega }{2|\Omega|^{\frac{3}{2}}}\Bigg( 1 + \mathcal{O}\Bigg(\frac{|\xi'|}{|\Omega|^{\frac{1}{2}}}\Bigg)\Bigg) & \displaystyle\frac{e^{-\frac{1}{4}i\pi}|\xi'|}{2|\Omega|^{\frac{1}{2}}}\Bigg(1 + \mathcal{O}\Bigg(\frac{|\xi'|}{|\Omega|^{\frac{1}{2}}}\Bigg)\Bigg) 
      \end{pmatrix}\\
      &
      \begin{pmatrix}
        i\xi'\cdot(\widehat{\bf a}_{h}(\xi') - \widehat{\bf w}_{h}^{(2)}(\xi',0))\\ 
        -(i\xi')^{\perp}\cdot(\widehat{\bf a}_{h}(\xi') - \widehat{\bf w}_{h}^{(2)}(\xi',0))\\ 
        0
      \end{pmatrix}
      . 
    \end{align*}
    Also, if $2^{j} < C_{0}^{-1}|\Omega|^{\frac{1}{2}}$, 
    then there exists a constant $C > 0$ independent of ${\bf a}$ and ${\bf w}^{(2)}$ such that for $k = 1,2,3$, we have 
    \begin{equation*}
      {\left\|\mathcal{F}_{\mathbb{R}^{2}}^{-1}\left[\widehat{\phi}_{j}(\xi')C_{3}\right]\right\|}_{L^{p}} 
      \leq C\frac{2^{j}}{|\Omega|^{\frac{1}{2}}}{\|\phi_{j}*({\bf a} - {\bf w}^{(2)}(\cdot, 0))\|}_{L^{p}}. 
    \end{equation*}
  \end{enumerate}
\end{lem}

Using the above lemmas, we obtain estimates for ${\bf w}^{(2)}$ and ${\bf v}^{(2)}$ under the frequency restriction. 
Here, we recall that the operators $P_{C_{0}^{-1}|\Omega|^{\frac{1}{2}}\leq 2^{j}\leq C_{0}|\Omega|^{\frac{1}{2}}}$ and $P_{2^{j}< C_{0}^{-1}|\Omega|^{\frac{1}{2}}}$ are defined as follows: 
\begin{equation*}
  P_{C_{0}^{-1}|\Omega|^{\frac{1}{2}}\leq 2^{j}\leq C_{0}|\Omega|^{\frac{1}{2}}}f 
  = \sum_{C_{0}^{-1}|\Omega|^{\frac{1}{2}}\leq 2^{j}\leq C_{0}|\Omega|^{\frac{1}{2}}}\phi_{j}*f 
  \quad 
  \text{and}
  \quad 
  P_{2^{j}\leq C_{0}^{-1}|\Omega|^{\frac{1}{2}}}f 
  = \sum_{2^{j}\leq C_{0}^{-1}|\Omega|^{\frac{1}{2}}}\phi_{j}*f. 
\end{equation*}
We now state a basic property of these operators. 
\begin{lem}\label{0617-1}
  Let $\Omega\in\mathbb{R}\setminus\{0\}$, $C_{0} \geq 1$, $s\in\mathbb{R}$, $s_{1}>0$ and $1\leq p \leq \infty$. 
  Then there exists a constant $C > 0$ such that for any $j'\in \mathbb{Z}$ and $f$ with $\phi_{j'}*f \in L^{p}(\mathbb{R}^{2})$, we have 
  \begin{equation}\label{0617-2}
    |\Omega|^{\frac{s}{2}}{\left\|\phi_{j'}*P_{C_{0}^{-1}|\Omega|^{\frac{1}{2}}\leq 2^{j}\leq C_{0}|\Omega|^{\frac{1}{2}}}f\right\|}_{L^{p}}
    \leq C2^{sj'}{\left\|\phi_{j'}*P_{C_{0}^{-1}|\Omega|^{\frac{1}{2}}\leq 2^{j}\leq C_{0}|\Omega|^{\frac{1}{2}}}f\right\|}_{L^{p}}, 
  \end{equation}
  and 
  \begin{equation}\label{0617-3}
    |\Omega|^{-\frac{s_{1}}{2}}{\left\|\phi_{j'}*P_{2^{j}\leq C_{0}^{-1}|\Omega|^{\frac{1}{2}}}f\right\|}_{L^{p}}
    \leq C2^{-s_{1}j'}{\left\|\phi_{j'}*P_{2^{j}\leq C_{0}^{-1}|\Omega|^{\frac{1}{2}}}f\right\|}_{L^{p}}. 
  \end{equation}
\end{lem}
\begin{proof}
  By the definition of $\phi_{j}$, we can write 
  \begin{equation*}
    \phi_{j'}*P_{C_{0}^{-1}|\Omega|^{\frac{1}{2}}\leq 2^{j}\leq C_{0}|\Omega|^{\frac{1}{2}}}f 
    = \phi_{j'}*(\phi_{j-1} + \phi_{j} + \phi_{j+1})*f. 
  \end{equation*}
  Thus, if $2^{j'} < (2C_{0})^{-1}|\Omega|^{\frac{1}{2}}$ or $2^{j'} > 2C_{0}|\Omega|^{\frac{1}{2}}$, then we have 
  \begin{equation*}
    \phi_{j'}*P_{C_{0}^{-1}|\Omega|^{\frac{1}{2}}\leq 2^{j}\leq C_{0}|\Omega|^{\frac{1}{2}}}f 
    = 0. 
  \end{equation*}
  Therefore, it suffices to consider only the case of $(2C_{0})^{-1}|\Omega|^{\frac{1}{2}}\leq 2^{j'}\leq 2C_{0}|\Omega|^{\frac{1}{2}}$. 
  Consequently, we obtain 
  \begin{align*}
    |\Omega|^{\frac{s}{2}}{\left\|\phi_{j'}*P_{C_{0}^{-1}|\Omega|^{\frac{1}{2}}\leq 2^{j}\leq C_{0}|\Omega|^{\frac{1}{2}}}f\right\|}_{L^{p}} 
    &= C_{0}^{|s|}(C_{0}^{-|s|}|\Omega|^{\frac{s}{2}}){\left\|\phi_{j'}*P_{C_{0}^{-1}|\Omega|^{\frac{1}{2}}\leq 2^{j}\leq C_{0}|\Omega|^{\frac{1}{2}}}f\right\|}_{L^{p}}\\ 
    &\leq C2^{sj'}{\left\|\phi_{j'}*P_{C_{0}^{-1}|\Omega|^{\frac{1}{2}}\leq 2^{j}\leq C_{0}|\Omega|^{\frac{1}{2}}}f\right\|}_{L^{p}}. 
  \end{align*}
  Similarly, it follows that 
  \begin{align*}
    |\Omega|^{-\frac{s_{1}}{2}}{\left\|\phi_{j'}*P_{2^{j}\leq C_{0}^{-1}|\Omega|^{\frac{1}{2}}}f\right\|}_{L^{p}}
    &= C_{0}^{-s_{1}}(C_{0}^{-1}|\Omega|^{\frac{1}{2}})^{-s_{1}}{\left\|\phi_{j'}*P_{2^{j}\leq C_{0}^{-1}|\Omega|^{\frac{1}{2}}}f\right\|}_{L^{p}}\\ 
    &\leq C2^{-s_{1}j'}{\left\|\phi_{j'}*P_{2^{j}\leq C_{0}^{-1}|\Omega|^{\frac{1}{2}}}f\right\|}_{L^{p}}. 
  \end{align*}
\end{proof}

For the formulas \eqref{0612-18}, we have the following proposition. 
\begin{prop}\label{0619-4}
  Let $\Omega \in \mathbb{R}\setminus\{0\}$, $s\in \mathbb{R}$ and $1\leq p,q,r,r_{0},r_{1} \leq \infty$ with $2 \leq r_{1}$ and $\max\{r_{0}, r_{1}\}\leq r$. 
  Let ${\bf u}$ and ${\bf F}$ be smooth functions on $\mathbb{R}_{+}^{3}$. 
  \begin{enumerate}
    \item\label{0616-1} Let $C_{0} \geq 1$. Then there exists a constant $C>0$ such that for ${\bf w}^{(2)}$ given by \eqref{0612-18}, we have 
    \begin{align*}
      {\left\|P_{C_{0}^{-1}|\Omega|^{\frac{1}{2}}\leq 2^{j}\leq C_{0}|\Omega|^{\frac{1}{2}}}{\bf w}^{(2)}\right\|}_{\widetilde{L}_{x_{3}}^{r}(0, \infty ; \dot{B}_{p,q}^{s + \frac{1}{r}})} 
      \leq& C {\left\|P_{C_{0}^{-1}|\Omega|^{\frac{1}{2}}\leq 2^{j}\leq C_{0}|\Omega|^{\frac{1}{2}}}{\bf F}\right\|}_{\widetilde{L}_{x_{3}}^{r_{0}}(0, \infty ; \dot{B}_{p,q}^{s + \frac{1}{r_{0}} - 1})}\\ 
      &+ C {\left\|P_{C_{0}^{-1}|\Omega|^{\frac{1}{2}}\leq 2^{j}\leq C_{0}|\Omega|^{\frac{1}{2}}}({\bf u}\otimes{\bf u})\right\|}_{\widetilde{L}_{x_{3}}^{\frac{r_{1}}{2}}(0, \infty ; \dot{B}_{p,q}^{s + \frac{2}{r_{1}} - 1})}. 
    \end{align*}
    \item\label{0617-5} There exists a sufficiently large constant $C_{0} \geq 1$ such that 
    there exists a constant $C>0$ such that for ${\bf w}^{(2)}$ given by \eqref{0612-18}, we have 
    \begin{align*}
      {\left\|P_{2^{j}< C_{0}^{-1}|\Omega|^{\frac{1}{2}}}{\bf w}^{(2)}\right\|}_{\widetilde{L}_{x_{3}}^{r}(0,\infty;\dot{B}_{p,q}^{s + \frac{1}{r}})} 
      \leq&C |\Omega|^{- \frac{1}{r_{0}} + \frac{1}{r}}{\left\|P_{2^{j}< C_{0}^{-1}|\Omega|^{\frac{1}{2}}}{\bf F}\right\|}_{\widetilde{L}_{x_{3}}^{r_{0}}(0,\infty;\dot{B}_{p,q}^{s + \frac{3}{r_{0}} - \frac{2}{r} - 1})}\notag\\ 
      &+ C |\Omega|^{-\frac{1}{2}\left(1 + \frac{1}{r} - \frac{1}{r_{0}}\right)}{\left\|P_{2^{j}< C_{0}^{-1}|\Omega|^{\frac{1}{2}}}{\bf F}\right\|}_{\widetilde{L}_{x_{3}}^{r_{0}}(0,\infty;\dot{B}_{p,q}^{s + \frac{1}{r}})}\notag\\ 
      &+ C |\Omega|^{-\frac{2}{r_{1}} + \frac{1}{r}}{\left\|P_{2^{j}< C_{0}^{-1}|\Omega|^{\frac{1}{2}}}({\bf u}\otimes{\bf u})\right\|}_{\widetilde{L}_{x_{3}}^{\frac{r_{1}}{2}}(0,\infty;\dot{B}_{p,q}^{s + \frac{6}{r_{1}} - \frac{2}{r} - 1})}\notag\\ 
      &+ C |\Omega|^{-\frac{1}{2}\left(1 + \frac{1}{r} - \frac{2}{r_{1}}\right)}{\left\|P_{2^{j}< C_{0}^{-1}|\Omega|^{\frac{1}{2}}}({\bf u}\otimes{\bf u})\right\|}_{\widetilde{L}_{x_{3}}^{\frac{r_{1}}{2}}(0,\infty;\dot{B}_{p,q}^{s + \frac{1}{r}})}. 
    \end{align*}
  \end{enumerate}
\end{prop}
\begin{proof}
  We first prove \eqref{0616-1}. 
  Since an estimate similar to \eqref{0616-2} also holds for $T_{K_{l}^{(2)}}^{(\pm)}$, 
  using Lemma~\ref{0611-6}, Lemma~\ref{0613-10} \eqref{0616-3} and Lemma~\ref{0617-1} \eqref{0617-2}, for each $j'\in \mathbb{Z}$, we have 
  \begin{align*}
    &{\left\|\phi_{j'}*P_{C_{0}^{-1}|\Omega|^{\frac{1}{2}}\leq 2^{j}\leq C_{0}|\Omega|^{\frac{1}{2}}}{\bf w}^{(2)}\right\|}_{L^{p}}\\ 
    \leq& C(1 + |\Omega|2^{-2j'})\int_{0}^{+\infty}e^{-c|x_{3} - y_{3}|2^{j'}}{\left\|\phi_{j'}*P_{C_{0}^{-1}|\Omega|^{\frac{1}{2}}\leq 2^{j}\leq C_{0}|\Omega|^{\frac{1}{2}}}{\bf F}\right\|}_{L^{p}}~{\rm d}y_{3}\\ 
    &+ C(1 + |\Omega|2^{-2j'})\int_{0}^{+\infty}e^{-c|x_{3} - y_{3}|2^{j'}}{\left\|\phi_{j'}*P_{C_{0}^{-1}|\Omega|^{\frac{1}{2}}\leq 2^{j}\leq C_{0}|\Omega|^{\frac{1}{2}}}({\bf u}\otimes{\bf u})\right\|}_{L^{p}}~{\rm d}y_{3}\\ 
    \leq& C\int_{0}^{+\infty}e^{-c|x_{3} - y_{3}|2^{j'}}{\left\|\phi_{j'}*P_{C_{0}^{-1}|\Omega|^{\frac{1}{2}}\leq 2^{j}\leq C_{0}|\Omega|^{\frac{1}{2}}}{\bf F}\right\|}_{L^{p}}~{\rm d}y_{3}\\ 
    &+ C\int_{0}^{+\infty}e^{-c|x_{3} - y_{3}|2^{j'}}{\left\|\phi_{j'}*P_{C_{0}^{-1}|\Omega|^{\frac{1}{2}}\leq 2^{j}\leq C_{0}|\Omega|^{\frac{1}{2}}}({\bf u}\otimes{\bf u})\right\|}_{L^{p}}~{\rm d}y_{3}. 
  \end{align*}
  By Young's inequality, we obtain 
  \begin{align}\label{0616-4}
    &{\left\|\phi_{j'}*P_{C_{0}^{-1}|\Omega|^{\frac{1}{2}}\leq 2^{j}\leq C_{0}|\Omega|^{\frac{1}{2}}}{\bf w}^{(2)}\right\|}_{L^{r}(0,\infty;L^{p})}\notag\\ 
    \leq& C2^{\left(-1 - \frac{1}{r} + \frac{1}{r_{0}}\right)j'}{\left\|\phi_{j'}*P_{C_{0}^{-1}|\Omega|^{\frac{1}{2}}\leq 2^{j}\leq C_{0}|\Omega|^{\frac{1}{2}}}{\bf F}\right\|}_{L^{r_{0}}(0,\infty; L^{p})}\notag\\ 
    &+ C2^{\left(-1 - \frac{1}{r} + \frac{2}{r_{1}}\right)j'}{\left\|\phi_{j'}*P_{C_{0}^{-1}|\Omega|^{\frac{1}{2}}\leq 2^{j}\leq C_{0}|\Omega|^{\frac{1}{2}}}({\bf u}\otimes{\bf u})\right\|}_{L^{\frac{r_{1}}{2}}(0, \infty; L^{p})}. 
  \end{align}
  Multiplying $2^{\left(s + \frac{1}{r}\right)j'}$ to \eqref{0616-4} and taking the $l^{q}(\mathbb{Z})$ norm with respect to $j'$, 
  we get \eqref{0616-1}. 

  Similarly, by Lemma~\ref{0611-6}, Lemma~\ref{0613-10} \eqref{0613-9} and Lemma~\ref{0617-1} \eqref{0617-3}, it follows that 
  \begin{align}\label{1005-1}
    &{\left\|\phi_{j'}*P_{2^{j} < C_{0}^{-1}|\Omega|^{\frac{1}{2}}}{\bf w}^{(2)}\right\|}_{L^{p}}\notag\\ 
    \leq& C\max\left\{\frac{2^{2j'}}{|\Omega|}, \frac{2^{4j'}}{|\Omega|^{2}}, \frac{2^{6j'}}{|\Omega|^{3}}\right\}\int_{0}^{+\infty}e^{-c|x_{3} - y_{3}|\frac{2^{3j'}}{|\Omega|}}{\left\|\phi_{j'}*P_{2^{j} < C_{0}^{-1}|\Omega|^{\frac{1}{2}}}{\bf F}\right\|}_{L^{p}}~{\rm d}y_{3}\notag\\ 
    &+ C\max\left\{1,\frac{2^{j'}}{|\Omega|^{\frac{1}{2}}}, \frac{2^{2j'}}{|\Omega|}, \frac{2^{3j'}}{|\Omega|^{\frac{3}{2}}}\right\}\int_{0}^{+\infty}e^{-c|x_{3} - y_{3}||\Omega|^{\frac{1}{2}}}{\left\|\phi_{j'}*P_{2^{j} < C_{0}^{-1}|\Omega|^{\frac{1}{2}}}{\bf F}\right\|}_{L^{p}}~{\rm d}y_{3}\notag\\ 
    &+ C\max\left\{\frac{2^{2j'}}{|\Omega|}, \frac{2^{4j'}}{|\Omega|^{2}}, \frac{2^{6j'}}{|\Omega|^{3}}\right\}\int_{0}^{+\infty}e^{-c|x_{3} - y_{3}|\frac{2^{3j'}}{|\Omega|}}{\left\|\phi_{j'}*P_{2^{j} < C_{0}^{-1}|\Omega|^{\frac{1}{2}}}({\bf u}\otimes{\bf u})\right\|}_{L^{p}}~{\rm d}y_{3}\notag\\ 
    &+ C\max\left\{1,\frac{2^{j'}}{|\Omega|^{\frac{1}{2}}}, \frac{2^{2j'}}{|\Omega|}, \frac{2^{3j'}}{|\Omega|^{\frac{3}{2}}}\right\}\int_{0}^{+\infty}e^{-c|x_{3} - y_{3}||\Omega|^{\frac{1}{2}}}{\left\|\phi_{j'}*P_{2^{j} < C_{0}^{-1}|\Omega|^{\frac{1}{2}}}({\bf u}\otimes{\bf u})\right\|}_{L^{p}}~{\rm d}y_{3}\notag\\ 
    \leq& C\frac{2^{2j'}}{|\Omega|}\int_{0}^{+\infty}e^{-c|x_{3} - y_{3}|\frac{2^{3j'}}{|\Omega|}}{\left\|\phi_{j'}*P_{2^{j} < C_{0}^{-1}|\Omega|^{\frac{1}{2}}}{\bf F}\right\|}_{L^{p}}~{\rm d}y_{3}\notag\\ 
    &+ C\int_{0}^{+\infty}e^{-c|x_{3} - y_{3}||\Omega|^{\frac{1}{2}}}{\left\|\phi_{j'}*P_{2^{j} < C_{0}^{-1}|\Omega|^{\frac{1}{2}}}{\bf F}\right\|}_{L^{p}}~{\rm d}y_{3}\notag\\ 
    &+ C\frac{2^{2j'}}{|\Omega|}\int_{0}^{+\infty}e^{-c|x_{3} - y_{3}|\frac{2^{3j'}}{|\Omega|}}{\left\|\phi_{j'}*P_{2^{j} < C_{0}^{-1}|\Omega|^{\frac{1}{2}}}({\bf u}\otimes{\bf u})\right\|}_{L^{p}}~{\rm d}y_{3}\notag\\ 
    &+ C\int_{0}^{+\infty}e^{-c|x_{3} - y_{3}||\Omega|^{\frac{1}{2}}}{\left\|\phi_{j'}*P_{2^{j} < C_{0}^{-1}|\Omega|^{\frac{1}{2}}}({\bf u}\otimes{\bf u})\right\|}_{L^{p}}~{\rm d}y_{3}. 
  \end{align}
  Using Young's inequality, we have 
  \begin{align}\label{0617-4}
    &{\left\|\phi_{j'}*P_{2^{j} < C_{0}^{-1}|\Omega|^{\frac{1}{2}}}{\bf w}^{(2)}\right\|}_{L^{r}(0,\infty;L^{p})}\notag\\ 
    \leq& C\frac{2^{2j'}}{|\Omega|}\left(\frac{2^{3j'}}{|\Omega|}\right)^{-1 - \frac{1}{r} + \frac{1}{r_{0}}}{\left\|\phi_{j'}*P_{2^{j} < C_{0}^{-1}|\Omega|^{\frac{1}{2}}}{\bf F}\right\|}_{L^{r_{0}}(0,\infty; L^{p})}\notag\\ 
    &+ C |\Omega|^{-\frac{1}{2}\left(1 + \frac{1}{r} - \frac{1}{r_{0}}\right)}{\left\|\phi_{j'}*P_{2^{j} < C_{0}^{-1}|\Omega|^{\frac{1}{2}}}{\bf F}\right\|}_{L^{r_{0}}(0,\infty; L^{p})}\notag\\ 
    &+ C\frac{2^{2j'}}{|\Omega|}\left(\frac{2^{3j'}}{|\Omega|}\right)^{-1 - \frac{1}{r} + \frac{2}{r_{1}}}{\left\|\phi_{j'}*P_{2^{j} < C_{0}^{-1}|\Omega|^{\frac{1}{2}}}({\bf u}\otimes{\bf u})\right\|}_{L^{\frac{r_{1}}{2}}(0, \infty; L^{p})}\notag\\ 
    &+ C |\Omega|^{-\frac{1}{2}\left(1 + \frac{1}{r} - \frac{2}{r_{1}}\right)}{\left\|\phi_{j'}*P_{2^{j} < C_{0}^{-1}|\Omega|^{\frac{1}{2}}}({\bf u}\otimes{\bf u})\right\|}_{L^{\frac{r_{1}}{2}}(0, \infty; L^{p})}. 
  \end{align}
  Multiplying \eqref{0617-4} by $2^{\left(s + \frac{1}{r}\right)j'}$ and taking the $l^{q}(\mathbb{Z})$ norm with respect to $j'$, 
  we obtain \eqref{0617-5}. 
\end{proof}

By a similar argument, we obtain the following proposition for the formulas \eqref{0612-19}. 
\begin{prop}\label{0618-4}
  Let $\Omega \in \mathbb{R}\setminus\{0\}$, $s\in \mathbb{R}$ and $1\leq p,q,r,r_{0},r_{1} \leq \infty$ with $2\leq r_{1}$. 
  Let ${\bf u}$ and ${\bf F}$ be smooth functions on $\mathbb{R}_{+}^{3}$ and ${\bf a}$ be a smooth function on $\mathbb{R}^{2}$ with $a_{3} = 0$. 
  \begin{enumerate}
    \item\label{0618-1} Let $C_{0} \geq 1$. Then there exists a constant $C>0$ such that for ${\bf v}^{(2)}$ given by \eqref{0612-19}, we have 
    \begin{align}\label{0618-7}
      {\left\|P_{C_{0}^{-1}|\Omega|^{\frac{1}{2}}\leq 2^{j}\leq C_{0}|\Omega|^{\frac{1}{2}}}{\bf v}^{(2)}\right\|}_{\widetilde{L}_{x_{3}}^{r}(0, \infty ; \dot{B}_{p,q}^{s + \frac{1}{r}})} 
      \leq& C {\left\|P_{C_{0}^{-1}|\Omega|^{\frac{1}{2}}\leq 2^{j}\leq C_{0}|\Omega|^{\frac{1}{2}}}{\bf a}\right\|}_{\dot{B}_{p,q}^{s}}\notag\\ 
      &+ C {\left\|P_{C_{0}^{-1}|\Omega|^{\frac{1}{2}}\leq 2^{j}\leq C_{0}|\Omega|^{\frac{1}{2}}}{\bf F}\right\|}_{\widetilde{L}_{x_{3}}^{r_{0}}(0, \infty ; \dot{B}_{p,q}^{s + \frac{1}{r_{0}} - 1})}\notag\\ 
      &+ C {\left\|P_{C_{0}^{-1}|\Omega|^{\frac{1}{2}}\leq 2^{j}\leq C_{0}|\Omega|^{\frac{1}{2}}}({\bf u}\otimes{\bf u})\right\|}_{\widetilde{L}_{x_{3}}^{\frac{r_{1}}{2}}(0, \infty ; \dot{B}_{p,q}^{s + \frac{2}{r_{1}} - 1})}. 
    \end{align}
    \item\label{0618-5} There exists a sufficiently large constant $C_{0} \geq 1$ such that there exists a constant $C>0$ such that for ${\bf v}^{(2)}$ given by \eqref{0612-19}, we have 
    \begin{align}\label{0618-10}
      &{\left\|P_{2^{j}< C_{0}^{-1}|\Omega|^{\frac{1}{2}}}{\bf v}^{(2)}\right\|}_{\widetilde{L}_{x_{3}}^{r}(0,\infty;\dot{B}_{p,q}^{s + \frac{1}{r}})}\notag\\ 
      \leq& C |\Omega|^{-\frac{1}{2} + \frac{1}{r}}{\left\|P_{2^{j}< C_{0}^{-1}|\Omega|^{\frac{1}{2}}}{\bf a}\right\|}_{\dot{B}_{p,q}^{s - \frac{2}{r} + 1}} 
      + C |\Omega|^{- \frac{1}{2r}}{\left\|P_{2^{j}< C_{0}^{-1}|\Omega|^{\frac{1}{2}}}{\bf a}\right\|}_{\dot{B}_{p,q}^{s + \frac{1}{r}}}\notag\\ 
      &+ C |\Omega|^{-\frac{1}{2} - \frac{1}{r_{0}}+ \frac{1}{r}}{\left\|P_{2^{j}< C_{0}^{-1}|\Omega|^{\frac{1}{2}}}{\bf F}\right\|}_{\widetilde{L}_{x_{3}}^{r_{0}}(0,\infty;\dot{B}_{p,q}^{s + \frac{3}{r_{0}} - \frac{2}{r}})}\notag\\ 
      & + C |\Omega|^{-1 + \frac{1}{2r_{0}} + \frac{1}{r}}{\left\|P_{2^{j}< C_{0}^{-1}|\Omega|^{\frac{1}{2}}}{\bf F}\right\|}_{\widetilde{L}_{x_{3}}^{r_{0}}(0,\infty;\dot{B}_{p,q}^{s - \frac{2}{r} + 1})}\notag\\ 
      &+ C |\Omega|^{ - \frac{1}{r_{0}} - \frac{1}{2r}}{\left\|P_{2^{j}< C_{0}^{-1}|\Omega|^{\frac{1}{2}}}{\bf F}\right\|}_{\widetilde{L}_{x_{3}}^{r_{0}}(0,\infty;\dot{B}_{p,q}^{s + \frac{3}{r_{0}} + \frac{1}{r} - 1})}\notag\\ 
      & + C |\Omega|^{-\frac{1}{2} + \frac{1}{2r_{0}} - \frac{1}{2r}}{\left\|P_{2^{j}< C_{0}^{-1}|\Omega|^{\frac{1}{2}}}{\bf F}\right\|}_{\widetilde{L}_{x_{3}}^{r_{0}}(0,\infty;\dot{B}_{p,q}^{s + \frac{1}{r}})}\notag\\ 
      &+ C |\Omega|^{- \frac{1}{2} - \frac{2}{r_{1}} + \frac{1}{r}}{\left\|P_{2^{j}< C_{0}^{-1}|\Omega|^{\frac{1}{2}}}({\bf u}\otimes{\bf u})\right\|}_{\widetilde{L}_{x_{3}}^{\frac{r_{1}}{2}}(0,\infty;\dot{B}_{p,q}^{s + \frac{6}{r_{1}} - \frac{2}{r}})}\notag\\ 
      &+ C |\Omega|^{-1 + \frac{1}{r_{1}} + \frac{1}{r}}{\left\|P_{2^{j}< C_{0}^{-1}|\Omega|^{\frac{1}{2}}}({\bf u}\otimes{\bf u})\right\|}_{\widetilde{L}_{x_{3}}^{\frac{r_{1}}{2}}(0,\infty;\dot{B}_{p,q}^{s - \frac{2}{r} + 1})}\notag\\ 
      &+ C |\Omega|^{ - \frac{2}{r_{1}} - \frac{1}{2r}}{\left\|P_{2^{j}< C_{0}^{-1}|\Omega|^{\frac{1}{2}}}({\bf u}\otimes{\bf u})\right\|}_{\widetilde{L}_{x_{3}}^{\frac{r_{1}}{2}}(0,\infty;\dot{B}_{p,q}^{s + \frac{6}{r_{1}} + \frac{1}{r} - 1})}\notag\\ 
      &+ C |\Omega|^{- \frac{1}{2} + \frac{1}{r_{1}} - \frac{1}{2r}}{\left\|P_{2^{j}< C_{0}^{-1}|\Omega|^{\frac{1}{2}}}({\bf u}\otimes{\bf u})\right\|}_{\widetilde{L}_{x_{3}}^{\frac{r_{1}}{2}}(0,\infty;\dot{B}_{p,q}^{s + \frac{1}{r}})}. 
    \end{align}
  \end{enumerate}
\end{prop}
\begin{proof}
  We first prove \eqref{0618-1}. 
  For each $j'\in \mathbb{Z}$, by Lemma~\ref{0611-6}, Lemma~\ref{0613-7} \eqref{0613-3} and Lemma~\ref{0617-1} \eqref{0617-2}, observe that the following holds. 
  \begin{align*}
    &{\left\|\phi_{j'}*P_{C_{0}^{-1}|\Omega|^{\frac{1}{2}}\leq 2^{j}\leq C_{0}|\Omega|^{\frac{1}{2}}}{\bf v}_{h}^{(2)}\right\|}_{L^{p}}\\ 
    \leq& \sum_{k = 1}^{3}2^{-j'}{\left\|\phi_{j'}*P_{C_{0}^{-1}|\Omega|^{\frac{1}{2}}\leq 2^{j}\leq C_{0}|\Omega|^{\frac{1}{2}}}\mathcal{F}_{\mathbb{R}^{2}}^{-1}\left[\lambda_{|\Omega|, k}C_{k}e^{-x_{3}\lambda_{|\Omega|, k}}\right]\right\|}_{L^{p}}\\ 
    +& \sum_{k = 1}^{3}2^{-j'}{\left\|\phi_{j'}*P_{C_{0}^{-1}|\Omega|^{\frac{1}{2}}\leq 2^{j}\leq C_{0}|\Omega|^{\frac{1}{2}}}\mathcal{F}_{\mathbb{R}^{2}}^{-1}\left[\frac{(|\xi'|^{2} - \lambda_{|\Omega|, k}^{2})^{2}}{\Omega \lambda_{|\Omega|, k}}C_{k}e^{-x_{3}\lambda_{|\Omega|, k}}\right]\right\|}_{L^{p}}\\ 
    \leq& C\sum_{k = 1}^{3}{\left\|\phi_{j'}*P_{C_{0}^{-1}|\Omega|^{\frac{1}{2}}\leq 2^{j}\leq C_{0}|\Omega|^{\frac{1}{2}}}\mathcal{F}_{\mathbb{R}^{2}}^{-1}\left[C_{k}e^{-x_{3}\lambda_{|\Omega|, k}}\right]\right\|}_{L^{p}}. 
  \end{align*}
  The same bound for $v_{3}^{(2)}$ as above follows. 
  Thanks to Lemma~\ref{0613-8} \eqref{0613-5} and Lemma~\ref{0618-2} \eqref{0618-3}, we have 
  \begin{align*}
    &{\left\|\phi_{j'}*P_{C_{0}^{-1}|\Omega|^{\frac{1}{2}}\leq 2^{j}\leq C_{0}|\Omega|^{\frac{1}{2}}}{\bf v}^{(2)}\right\|}_{L^{p}}\\ 
    \leq& C e^{-cx_{3}2^{j'}}{\left\|\phi_{j'}*P_{C_{0}^{-1}|\Omega|^{\frac{1}{2}}\leq 2^{j}\leq C_{0}|\Omega|^{\frac{1}{2}}}({\bf a} - {\bf w}^{(2)}(\cdot, 0))\right\|}_{L^{p}}. 
  \end{align*}
  Taking the $L^{r}$ norm with respect to $x_{3}$, we obtain 
  \begin{align*}
    &\sum_{k = 1}^{3}{\left\|\phi_{j'}*P_{C_{0}^{-1}|\Omega|^{\frac{1}{2}}\leq 2^{j}\leq C_{0}|\Omega|^{\frac{1}{2}}}\mathcal{F}_{\mathbb{R}^{2}}^{-1}\left[C_{k}e^{-x_{3}\lambda_{|\Omega|, k}}\right]\right\|}_{L_{x_{3}}^{r}(0, \infty; L^{p})}\\ 
    \leq& C 2^{-\frac{1}{r}j'}{\left\|\phi_{j'}*P_{C_{0}^{-1}|\Omega|^{\frac{1}{2}}\leq 2^{j}\leq C_{0}|\Omega|^{\frac{1}{2}}}{\bf a}\right\|}_{L^{p}} 
    + C 2^{-\frac{1}{r}j'}{\left\|\phi_{j'}*P_{C_{0}^{-1}|\Omega|^{\frac{1}{2}}\leq 2^{j}\leq C_{0}|\Omega|^{\frac{1}{2}}}{\bf w}^{(2)}(\cdot, 0)\right\|}_{L^{p}}. 
  \end{align*}
  We estimate ${\bf w}^{(2)}(\cdot, 0)$. 
  In the same way as for ${\bf w}^{(1)}(\cdot, 0)$, by Lemma~\ref{0611-6}, Lemma~\ref{0613-10} \eqref{0616-3} and H\"older's inequality, it follows that 
  \begin{align*}
    &{\left\|\phi_{j'}*P_{C_{0}^{-1}|\Omega|^{\frac{1}{2}}\leq 2^{j}\leq C_{0}|\Omega|^{\frac{1}{2}}}{\bf w}^{(2)}(\cdot, 0)\right\|}_{L^{p}}\\ 
    \leq& C \int_{0}^{+\infty}e^{-cy_{3}2^{j'}}{\left\|\phi_{j'}*P_{C_{0}^{-1}|\Omega|^{\frac{1}{2}}\leq 2^{j}\leq C_{0}|\Omega|^{\frac{1}{2}}}{\bf F}\right\|}_{L^{p}}~{\rm d}y_{3}\\ 
    & + C \int_{0}^{+\infty}e^{-cy_{3}2^{j'}}{\left\|\phi_{j'}*P_{C_{0}^{-1}|\Omega|^{\frac{1}{2}}\leq 2^{j}\leq C_{0}|\Omega|^{\frac{1}{2}}}({\bf u}\otimes{\bf u})\right\|}_{L^{p}}~{\rm d}y_{3}\\ 
    \leq& C2^{\left(-1 + \frac{1}{r_{0}}\right)j'}{\left\|\phi_{j'}*P_{C_{0}^{-1}|\Omega|^{\frac{1}{2}}\leq 2^{j}\leq C_{0}|\Omega|^{\frac{1}{2}}}{\bf F}\right\|}_{L_{x_{3}}^{r_{0}}(0,\infty ; L^{p})}\\ 
    &+ C 2^{\left(-1 + \frac{2}{r_{1}}\right)j'}{\left\|\phi_{j'}*P_{C_{0}^{-1}|\Omega|^{\frac{1}{2}}\leq 2^{j}\leq C_{0}|\Omega|^{\frac{1}{2}}}({\bf u}\otimes{\bf u})\right\|}_{L_{x_{3}}^{\frac{r_{1}}{2}}(0,\infty ; L^{p})}. 
  \end{align*}
  Therefore, we have 
  \begin{align}\label{0618-6}
    &\sum_{k = 1}^{3}{\left\|\phi_{j'}*P_{C_{0}^{-1}|\Omega|^{\frac{1}{2}}\leq 2^{j}\leq C_{0}|\Omega|^{\frac{1}{2}}}\mathcal{F}_{\mathbb{R}^{2}}^{-1}\left[C_{k}e^{-x_{3}\lambda_{|\Omega|, k}}\right]\right\|}_{L_{x_{3}}^{r}(0, \infty; L^{p})}\notag\\ 
    \leq& C 2^{-\frac{1}{r}j'}{\left\|\phi_{j'}*P_{C_{0}^{-1}|\Omega|^{\frac{1}{2}}\leq 2^{j}\leq C_{0}|\Omega|^{\frac{1}{2}}}{\bf a}\right\|}_{L^{p}}\notag\\ 
    & + C 2^{\left(-1 + \frac{1}{r_{0}} - \frac{1}{r}\right)j'}{\left\|\phi_{j'}*P_{C_{0}^{-1}|\Omega|^{\frac{1}{2}}\leq 2^{j}\leq C_{0}|\Omega|^{\frac{1}{2}}}{\bf F}\right\|}_{L_{x_{3}}^{r_{0}}(0,\infty ; L^{p})}\notag\\ 
    & + C2^{\left(-1 + \frac{2}{r_{1}} - \frac{1}{r}\right)j'}{\left\|\phi_{j'}*P_{C_{0}^{-1}|\Omega|^{\frac{1}{2}}\leq 2^{j}\leq C_{0}|\Omega|^{\frac{1}{2}}}({\bf u}\otimes{\bf u})\right\|}_{L_{x_{3}}^{\frac{r_{1}}{2}}(0,\infty ; L^{p})}. 
  \end{align}
  Multiplying \eqref{0618-6} by $2^{\left(s + \frac{1}{r}\right)j'}$ and taking the $l^{q}(\mathbb{Z})$ norm with respect to $j'$, 
  we get \eqref{0618-7}. 

  The proof of \eqref{0618-5} is similar. 
  If $C_{0}> 0$ is sufficiently large, then for each $j' \in \mathbb{Z}$, using Lemma~\ref{0617-1} \eqref{0617-3}, we have 
  \begin{align*}
    &{\left\|\phi_{j'}*P_{2^{j}< C_{0}^{-1}|\Omega|^{\frac{1}{2}}}{\bf v}_{h}^{(2)}\right\|}_{L^{p}}\\ 
    \leq& \sum_{k = 1}^{3}{\left\|\phi_{j'}*P_{2^{j}< C_{0}^{-1}|\Omega|^{\frac{1}{2}}}\mathcal{F}_{\mathbb{R}^{2}}^{-1}\left[\left(\frac{i\xi'\lambda_{|\Omega|, k}}{|\xi'|^{2}} - (i\xi')^{\perp}\frac{(|\xi'|^{2} - \lambda_{|\Omega|, k}^{2})^{2}}{\Omega |\xi'|^{2}\lambda_{|\Omega|, k}}\right)C_{k}e^{-x_{3}\lambda_{|\Omega|, k}}\right]\right\|}_{L^{p}}\\ 
    \leq& C{\left\|\phi_{j'}*P_{2^{j}< C_{0}^{-1}|\Omega|^{\frac{1}{2}}}\mathcal{F}_{\mathbb{R}^{2}}^{-1}\left[C_{1}e^{-x_{3}\lambda_{|\Omega|, 1}}\right]\right\|}_{L^{p}}\\ 
    &+ C\frac{|\Omega|^{\frac{1}{2}}}{2^{j'}}\sum_{k = 2}^{3}{\left\|\phi_{j'}*P_{2^{j}< C_{0}^{-1}|\Omega|^{\frac{1}{2}}}\mathcal{F}_{\mathbb{R}^{2}}^{-1}\left[C_{k}e^{-x_{3}\lambda_{|\Omega|, k}}\right]\right\|}_{L^{p}}. 
  \end{align*}
  The same bound for $v_{3}^{(2)}$ as above follows. 
  By Lemma~\ref{0613-8} \eqref{0613-6} and Lemma~\ref{0618-2} \eqref{0618-8}, we obtain 
  \begin{align*}
    {\left\|\phi_{j'}*P_{2^{j}< C_{0}^{-1}|\Omega|^{\frac{1}{2}}}{\bf v}^{(2)}\right\|}_{L^{p}} 
    \leq& C\frac{2^{j'}}{|\Omega|^{\frac{1}{2}}}e^{-cx_{3}\frac{2^{3j'}}{|\Omega|}}{\left\|\phi_{j'}*P_{2^{j}< C_{0}^{-1}|\Omega|^{\frac{1}{2}}}({\bf a} - {\bf w}^{(2)}(\cdot, 0))\right\|}_{L^{p}}\\ 
    &+ Ce^{-cx_{3}|\Omega|^{\frac{1}{2}}}{\left\|\phi_{j'}*P_{2^{j}< C_{0}^{-1}|\Omega|^{\frac{1}{2}}}({\bf a} - {\bf w}^{(2)}(\cdot, 0))\right\|}_{L^{p}}. 
  \end{align*}
  Taking the $L^{r}$ norm with respect to $x_{3}$, we get 
  \begin{align*}
    &{\left\|\phi_{j'}*P_{2^{j}< C_{0}^{-1}|\Omega|^{\frac{1}{2}}}{\bf v}^{(2)}\right\|}_{L_{x_{3}}^{r}(0, \infty ; L^{p})}\\ 
    \leq& C|\Omega|^{-\frac{1}{2} + \frac{1}{r}}2^{\left(1 - \frac{3}{r}\right)j'}{\left\|\phi_{j'}*P_{2^{j}< C_{0}^{-1}|\Omega|^{\frac{1}{2}}}{\bf a}\right\|}_{L^{p}}\\ 
    &+ C|\Omega|^{-\frac{1}{2} + \frac{1}{r}}2^{\left(1 - \frac{3}{r}\right)j'}{\left\|\phi_{j'}*P_{2^{j}< C_{0}^{-1}|\Omega|^{\frac{1}{2}}}{\bf w}^{(2)}(\cdot, 0)\right\|}_{L^{p}}\\ 
    &+ C|\Omega|^{-\frac{1}{2r}}{\left\|\phi_{j'}*P_{2^{j}< C_{0}^{-1}|\Omega|^{\frac{1}{2}}}{\bf a}\right\|}_{L^{p}} 
    + C|\Omega|^{-\frac{1}{2r}}{\left\|\phi_{j'}*P_{2^{j}< C_{0}^{-1}|\Omega|^{\frac{1}{2}}}{\bf w}^{(2)}(\cdot, 0)\right\|}_{L^{p}}. 
  \end{align*}
  Similarly to \eqref{1005-1}, thanks to Lemma~\ref{0611-6}, Lemma~\ref{0613-10} \eqref{0613-9} and H\"older's inequality, we obtain 
  \begin{align*}
    &{\left\|\phi_{j'}*P_{2^{j}< C_{0}^{-1}|\Omega|^{\frac{1}{2}}}{\bf w}^{(2)}(\cdot, 0)\right\|}_{L^{p}}\\ 
    \leq& C \frac{2^{2j'}}{|\Omega|}\int_{0}^{+\infty}e^{-cy_{3}\frac{2^{3j'}}{|\Omega|}}{\left\|\phi_{j'}*P_{2^{j}< C_{0}^{-1}|\Omega|^{\frac{1}{2}}}({\bf F} - {\bf u}\otimes{\bf u})\right\|}_{L^{p}}~{\rm d}y_{3}\\ 
    &+ C \int_{0}^{+\infty}e^{-cy_{3}|\Omega|^{\frac{1}{2}}}{\left\|\phi_{j'}*P_{2^{j}< C_{0}^{-1}|\Omega|^{\frac{1}{2}}}({\bf F} - {\bf u}\otimes{\bf u})\right\|}_{L^{p}}~{\rm d}y_{3}\\ 
    \leq& C |\Omega|^{-\frac{1}{r_{0}}}2^{\left(-1 + \frac{3}{r_{0}}\right)j'}{\left\|\phi_{j'}*P_{2^{j}< C_{0}^{-1}|\Omega|^{\frac{1}{2}}}{\bf F}\right\|}_{L_{x_{3}}^{r_{0}}(0, \infty; L^{p})}\\ 
    &+ C |\Omega|^{-\frac{1}{2}\left(1 - \frac{1}{r_{0}}\right)}{\left\|\phi_{j'}*P_{2^{j}< C_{0}^{-1}|\Omega|^{\frac{1}{2}}}{\bf F}\right\|}_{L_{x_{3}}^{r_{0}}(0, \infty; L^{p})}\\ 
    &+ C |\Omega|^{-\frac{2}{r_{1}}}2^{\left(-1 + \frac{6}{r_{1}}\right)j'}{\left\|\phi_{j'}*P_{2^{j}< C_{0}^{-1}|\Omega|^{\frac{1}{2}}}({\bf u}\otimes{\bf u})\right\|}_{L_{x_{3}}^{\frac{r_{1}}{2}}(0, \infty; L^{p})}\\ 
    &+ C |\Omega|^{-\frac{1}{2}\left(1 - \frac{2}{r_{1}}\right)}{\left\|\phi_{j'}*P_{2^{j}< C_{0}^{-1}|\Omega|^{\frac{1}{2}}}({\bf u}\otimes{\bf u})\right\|}_{L_{x_{3}}^{\frac{r_{1}}{2}}(0, \infty; L^{p})}. 
  \end{align*}
  Therefore, it follows that 
  \begin{align}\label{0618-9}
    &{\left\|\phi_{j'}*P_{2^{j}< C_{0}^{-1}|\Omega|^{\frac{1}{2}}}{\bf v}^{(2)}\right\|}_{L_{x_{3}}^{r}(0, \infty ; L^{p})}\notag\\ 
    \leq& C|\Omega|^{-\frac{1}{2} + \frac{1}{r}}2^{\left(1 - \frac{3}{r}\right)j'}{\left\|\phi_{j'}*P_{2^{j}< C_{0}^{-1}|\Omega|^{\frac{1}{2}}}{\bf a}\right\|}_{L^{p}} 
    + C|\Omega|^{-\frac{1}{2r}}{\left\|\phi_{j'}*P_{2^{j}< C_{0}^{-1}|\Omega|^{\frac{1}{2}}}{\bf a}\right\|}_{L^{p}}\notag\\ 
    &+ C|\Omega|^{-\frac{1}{2} -\frac{1}{r_{0}} + \frac{1}{r}}2^{\left(\frac{3}{r_{0}} - \frac{3}{r}\right)j'}{\left\|\phi_{j'}*P_{2^{j}< C_{0}^{-1}|\Omega|^{\frac{1}{2}}}{\bf F}\right\|}_{L_{x_{3}}^{r_{0}}(0, \infty; L^{p})}\notag\\ 
    &+ C|\Omega|^{-1 + \frac{1}{2r_{0}} + \frac{1}{r}}2^{\left(1 - \frac{3}{r}\right)j'}{\left\|\phi_{j'}*P_{2^{j}< C_{0}^{-1}|\Omega|^{\frac{1}{2}}}{\bf F}\right\|}_{L_{x_{3}}^{r_{0}}(0, \infty; L^{p})}\notag\\ 
    &+ C|\Omega|^{-\frac{1}{2} - \frac{2}{r_{1}} + \frac{1}{r}}2^{\left(\frac{6}{r_{1}} - \frac{3}{r}\right)j'}{\left\|\phi_{j'}*P_{2^{j}< C_{0}^{-1}|\Omega|^{\frac{1}{2}}}({\bf u}\otimes{\bf u})\right\|}_{L_{x_{3}}^{\frac{r_{1}}{2}}(0, \infty; L^{p})}\notag\\ 
    &+ C|\Omega|^{-1 + \frac{1}{r_{1}} + \frac{1}{r}}2^{\left(1 - \frac{3}{r}\right)j'}{\left\|\phi_{j'}*P_{2^{j}< C_{0}^{-1}|\Omega|^{\frac{1}{2}}}({\bf u}\otimes{\bf u})\right\|}_{L_{x_{3}}^{\frac{r_{1}}{2}}(0, \infty; L^{p})}\notag\\ 
    &+ C|\Omega|^{-\frac{1}{r_{0}} -\frac{1}{2r}}2^{\left(-1 + \frac{3}{r_{0}}\right)j'}{\left\|\phi_{j'}*P_{2^{j}< C_{0}^{-1}|\Omega|^{\frac{1}{2}}}{\bf F}\right\|}_{L_{x_{3}}^{r_{0}}(0, \infty; L^{p})}\notag\\ 
    &+ C|\Omega|^{-\frac{1}{2}\left(1 - \frac{1}{r_{0}} + \frac{1}{r}\right)}{\left\|\phi_{j'}*P_{2^{j}< C_{0}^{-1}|\Omega|^{\frac{1}{2}}}{\bf F}\right\|}_{L_{x_{3}}^{r_{0}}(0, \infty; L^{p})}\notag\\ 
    &+ C|\Omega|^{-\frac{2}{r_{1}} -\frac{1}{2r}}2^{\left(-1 + \frac{6}{r_{1}}\right)j'}{\left\|\phi_{j'}*P_{2^{j}< C_{0}^{-1}|\Omega|^{\frac{1}{2}}}({\bf u}\otimes{\bf u})\right\|}_{L_{x_{3}}^{\frac{r_{1}}{2}}(0, \infty; L^{p})}\notag\\ 
    &+ C|\Omega|^{-\frac{1}{2}\left(1 - \frac{2}{r_{1}} + \frac{1}{r}\right)}{\left\|\phi_{j'}*P_{2^{j}< C_{0}^{-1}|\Omega|^{\frac{1}{2}}}({\bf u}\otimes{\bf u})\right\|}_{L_{x_{3}}^{\frac{r_{1}}{2}}(0, \infty; L^{p})}. 
  \end{align}
  Multiplying $2^{\left(s + \frac{1}{r}\right)j'}$ to \eqref{0618-9} and taking the $l^{q}(\mathbb{Z})$ norm with respect to $j'$, 
  we get \eqref{0618-10}. 
\end{proof}

\subsection{Auxiliary lemmas}\label{0624-3} 
In this subsection, we present some properties of frequency-restricted functions, 
and a bilinear estimate in Besov spaces, 
both of which are needed for the proofs of our main theorems. 

The following lemma establishes the boundedness of the operators 
\begin{equation*}
  P_{2^{j} > C_{0}|\Omega|^{\frac{1}{2}}}, \quad 
  P_{C_{0}^{-1}|\Omega|^{\frac{1}{2}} \leq 2^{j} \leq C_{0}|\Omega|^{\frac{1}{2}}} 
  \quad 
  \text{and} 
  \quad 
  P_{2^{j} < C_{0}^{-1}|\Omega|^{\frac{1}{2}}}, 
\end{equation*}
uniformly in $\Omega$ and $C_{0}$. 
\begin{lem}\label{0619-2}
  Let $\Omega\in\mathbb{R}\setminus\{0\}$, $C_{0} \geq 1$ and $1\leq p \leq \infty$. 
  Then there exists a constant $C > 0$ independent of $\Omega$ and $C_{0}$ such that for any $f \in L^{p}(\mathbb{R}^{2})$, we have 
  \begin{equation}\label{0619-1}
    {\left\|P_{2^{j} > C_{0}|\Omega|^{\frac{1}{2}}}f\right\|}_{L^{p}} 
    \leq C {\|f\|}_{L^{p}}, 
  \end{equation}
  \begin{equation}\label{0619-7}
    {\left\|P_{C_{0}^{-1}|\Omega|^{\frac{1}{2}} \leq 2^{j} \leq C_{0}|\Omega|^{\frac{1}{2}}}f\right\|}_{L^{p}} 
    \leq C {\|f\|}_{L^{p}}, 
  \end{equation}
  and 
  \begin{equation}\label{0619-8}
    {\left\|P_{2^{j} < C_{0}^{-1}|\Omega|^{\frac{1}{2}}}f\right\|}_{L^{p}} 
    \leq C {\|f\|}_{L^{p}}. 
  \end{equation}
\end{lem}
\begin{proof}
  We only prove \eqref{0619-1}, since the proofs of the others are similar. 
  By the definition of $P_{2^{j} > C_{0}|\Omega|^{\frac{1}{2}}}$ (Definition~\ref{0625-14}) , we can write 
  \begin{equation*}
    P_{2^{j} > C_{0}|\Omega|^{\frac{1}{2}}}f 
    = \mathcal{F}_{\mathbb{R}^{2}}^{-1}\Bigg[\sum_{2^{j} > C_{0}|\Omega|^{\frac{1}{2}}}\widehat{\phi}_{j}(\xi')\widehat{f}(\xi')\Bigg]
    = \mathcal{F}_{\mathbb{R}^{2}}^{-1}\Bigg[\Bigg(Id - \sum_{2^{j} \leq C_{0}|\Omega|^{\frac{1}{2}}}\widehat{\phi}_{j}(\xi')\Bigg)\widehat{f}(\xi')\Bigg]. 
  \end{equation*}
  The following function 
  \begin{equation*}
    \sum_{2^{j} \leq C_{0}|\Omega|^{\frac{1}{2}}}\widehat{\phi}_{j} 
  \end{equation*}
  is supported in $\{\xi'\in \mathbb{R}^{2} |\  |\xi'| \leq 2C_{0}|\Omega|^{\frac{1}{2}}\}$, 
  and is identically equal to $1$ in a neighborhood of the origin. 
  Thus, there exists a smooth radial function $\chi$ such that 
  \begin{equation*}
    \text{supp }\chi \subset \{\xi'\in \mathbb{R}^{2} |\ |\xi'| \leq 1\}, 
  \end{equation*}
  $\chi (\xi') = 1$ if $\xi'$ is near $0$, and 
  \begin{equation*}
    \sum_{2^{j} \leq C_{0}|\Omega|^{\frac{1}{2}}}\widehat{\phi}_{j}(\xi') 
    = \chi ((2C_{0})^{-1}|\Omega|^{-\frac{1}{2}}\xi'). 
  \end{equation*}
  Using Young's inequality, we have 
  \begin{equation*}
    {\left\|P_{2^{j} > C_{0}|\Omega|^{\frac{1}{2}}}f\right\|}_{L^{p}} 
    \leq {\|f\|}_{L^{p}} 
    + {\left\|\mathcal{F}_{\mathbb{R}^{2}}^{-1}\left[\chi ((2C_{0})^{-1}|\Omega|^{-\frac{1}{2}}\xi')\right]\right\|}_{L^{1}}{\|f\|}_{L^{p}}. 
  \end{equation*}
  By a change of variables, we obtain 
  \begin{equation*}
    {\left\|\mathcal{F}_{\mathbb{R}^{2}}^{-1}\left[\chi ((2C_{0})^{-1}|\Omega|^{-\frac{1}{2}}\xi')\right]\right\|}_{L^{1}} 
    = {\left\|\mathcal{F}_{\mathbb{R}^{2}}^{-1}\left[\chi (\xi')\right]\right\|}_{L^{1}} 
    \leq C, 
  \end{equation*}
  where the constant $C > 0$ is independent of $\Omega$ and $C_{0}$. 
\end{proof}

As in Lemma~\ref{0617-1}, the following property holds for the operator $P_{2^{j} > C_{0}|\Omega|^{\frac{1}{2}}}$. 
Since the dependence on the constant $C_{0}$ is important, we state it explicitly. 
The proof is the same as that of Lemma~\ref{0617-1} and is omitted. 
\begin{lem}\label{0619-3}
  Let $\Omega\in\mathbb{R}\setminus\{0\}$, $C_{0} \geq 1$, $s > 0$ and $1\leq p \leq \infty$. 
  Then there exists a constant $C > 0$ independent of $\Omega$ and $C_{0}$ such that for any $j'\in \mathbb{Z}$ and $f$ with $\phi_{j'}*f \in L^{p}(\mathbb{R}^{2})$, we have 
  \begin{equation*}
    |\Omega|^{\frac{s}{2}}{\left\|\phi_{j'}*P_{2^{j} > C_{0}|\Omega|^{\frac{1}{2}}}f\right\|}_{L^{p}}
    \leq \frac{C}{C_{0}^{s}}2^{sj'}{\left\|\phi_{j'}*P_{2^{j} > C_{0}|\Omega|^{\frac{1}{2}}}f\right\|}_{L^{p}}. 
  \end{equation*}
\end{lem}

In Theorem~\ref{0609-4}, it is important to estimate the solution for sufficiently large $|\Omega|$. 
The following lemma follows immediately from the definition of the Besov norm \eqref{0625-15}. 
\begin{lem}\label{0619-16}
  Let $\Omega\in\mathbb{R}\setminus\{0\}$, $C_{0} \geq 1$, $s \in \mathbb{R}$, $1\leq p \leq \infty$ and $1 \leq q < \infty$. 
  Then for any fixed $C_{0}$ and for any $f \in \dot{B}_{p,q}^{s}(\mathbb{R}^{2})$, we have 
  \begin{equation*}
    \lim_{|\Omega|\to \infty}{\left\|P_{2^{j} > C_{0}|\Omega|^{\frac{1}{2}}}f\right\|}_{\dot{B}_{p,q}^{s}} 
    = \lim_{|\Omega|\to \infty}{\left\|P_{C_{0}^{-1}|\Omega|^{\frac{1}{2}}\leq 2^{j}\leq C_{0}|\Omega|^{\frac{1}{2}}}f\right\|}_{\dot{B}_{p,q}^{s}} 
    = 0. 
  \end{equation*}
\end{lem}

By Lemma~\ref{0617-1} \eqref{0617-3}, the following lemma holds in the case $2^{j} < C_{0}^{-1}|\Omega|^{\frac{1}{2}}$. 
\begin{lem}\label{0619-15}
  Let $\Omega\in\mathbb{R}\setminus\{0\}$, $C_{0} \geq 1$, $s \in \mathbb{R}$, $s_{1} > 0$, $1\leq p \leq \infty$ and $1 \leq q < \infty$. 
  Then for any fixed $C_{0}$ and for any $f \in \dot{B}_{p,q}^{s}(\mathbb{R}^{2})$, we have 
  \begin{equation*}
    \lim_{|\Omega|\to \infty}|\Omega|^{-\frac{s_{1}}{2}}{\left\|P_{2^{j} < C_{0}^{-1}|\Omega|^{\frac{1}{2}}}f\right\|}_{\dot{B}_{p,q}^{s + s_{1}}} 
    = 0. 
  \end{equation*}
  Also, let $1\leq r \leq \infty$. 
  Then for any fixed $C_{0}$ and for any $f \in \widetilde{L}_{x_{3}}^{r}(0, \infty ; \dot{B}_{p,q}^{s}(\mathbb{R}^{2}))$, we have 
  \begin{equation*}
    \lim_{|\Omega|\to \infty}|\Omega|^{-\frac{s_{1}}{2}}{\left\|P_{2^{j} < C_{0}^{-1}|\Omega|^{\frac{1}{2}}}f\right\|}_{\widetilde{L}_{x_{3}}^{r}(0, \infty ; \dot{B}_{p,q}^{s + s_{1}})} 
    = 0. 
  \end{equation*}
\end{lem}
\begin{proof}
  We only consider the case $f \in \dot{B}_{p,q}^{s}(\mathbb{R}^{2})$. 
  The case $f \in \widetilde{L}_{x_{3}}^{r}(0, \infty ; \dot{B}_{p,q}^{s}(\mathbb{R}^{2}))$ can be proved similarly. 
  Since $q < \infty$, it follows that 
  \begin{equation*}
    \lim_{M \to \infty}{\Bigg\|f - \sum_{|j'|\leq M}\phi_{j'}*f\Bigg\|}_{\dot{B}_{p,q}^{s}} = 0. 
  \end{equation*}
  We denote 
  \begin{equation*}
    f_{M} := \sum_{|j'|\leq M}\phi_{j'}*f. 
  \end{equation*}
  For any $M \in \mathbb{N}$, we have 
  \begin{equation*}
    {\left\|P_{2^{j} < C_{0}^{-1}|\Omega|^{\frac{1}{2}}}f\right\|}_{\dot{B}_{p,q}^{s + s_{1}}} 
    \leq {\left\|P_{2^{j} < C_{0}^{-1}|\Omega|^{\frac{1}{2}}}(f - f_{M})\right\|}_{\dot{B}_{p,q}^{s + s_{1}}} 
    + {\left\|P_{2^{j} < C_{0}^{-1}|\Omega|^{\frac{1}{2}}}f_{M}\right\|}_{\dot{B}_{p,q}^{s + s_{1}}}. 
  \end{equation*}
  Moreover, using Lemma~\ref{0617-1} \eqref{0617-3} and Lemma~\ref{0619-2} \eqref{0619-8}, we obtain 
  \begin{equation*}
    |\Omega|^{-\frac{s_{1}}{2}}{\left\|P_{2^{j} < C_{0}^{-1}|\Omega|^{\frac{1}{2}}}(f - f_{M})\right\|}_{\dot{B}_{p,q}^{s + s_{1}}} 
    \leq C{\left\|P_{2^{j} < C_{0}^{-1}|\Omega|^{\frac{1}{2}}}(f - f_{M})\right\|}_{\dot{B}_{p,q}^{s}} 
    \leq C{\|f - f_{M}\|}_{\dot{B}_{p,q}^{s}}. 
  \end{equation*}
  Therefore, for any $\varepsilon > 0$, there exists $N \in \mathbb{N}$ such that for any $M\in \mathbb{N}$ with $M \geq N$, 
  and for any $\Omega$, it follows that 
  \begin{equation}\label{0619-13}
    |\Omega|^{-\frac{s_{1}}{2}}{\left\|P_{2^{j} < C_{0}^{-1}|\Omega|^{\frac{1}{2}}}(f - f_{M})\right\|}_{\dot{B}_{p,q}^{s + s_{1}}} 
    \leq \varepsilon. 
  \end{equation}
  For fixed $M\in \mathbb{N}$, $f_{M}$ is a smooth function, that is, by Lemma~\ref{0619-2} \eqref{0619-8}, we have 
  \begin{equation*}
    {\left\|P_{2^{j} < C_{0}^{-1}|\Omega|^{\frac{1}{2}}}f_{M}\right\|}_{\dot{B}_{p,q}^{s + s_{1}}} 
    \leq C {\|f_{M}\|}_{\dot{B}_{p,q}^{s + s_{1}}} 
    < \infty, 
  \end{equation*}
  which implies 
  \begin{equation}\label{0619-14}
    \lim_{|\Omega| \to \infty}|\Omega|^{-\frac{s_{1}}{2}}{\left\|P_{2^{j} < C_{0}^{-1}|\Omega|^{\frac{1}{2}}}f_{M}\right\|}_{\dot{B}_{p,q}^{s + s_{1}}} 
    = 0. 
  \end{equation}
  It follows from \eqref{0619-13} and \eqref{0619-14} that Lemma~\ref{0619-15} holds. 
\end{proof}

The following lemma provides a bilinear estimate in the Chemin--Lerner spaces. 
We state it here in the form needed for the proof of our main theorems. 
For the proof, we refer the reader to \cite{Fu_2025} (see also \cite{Ba_Ch_Da_2011}). 
\begin{lem}\label{0619-9}
  Let $1 \leq p < \infty$, $1 \leq q \leq \infty$ and $2\leq r < \infty$ with $p < 2r'$, 
  where $r' = r/(r - 1)$ is the H\"older conjugate of $r$. 
  Then there exists a constant $C > 0$ such that for any $f, g \in \widetilde{L}_{x_{3}}^{r}(0,\infty ; \dot{B}_{p,q}^{\frac{2}{p} + \frac{1}{r} - 1}(\mathbb{R}^{2}))$, we have 
  \begin{equation*}
    {\|fg\|}_{\widetilde{L}_{x_{3}}^{\frac{r}{2}}(0,\infty ; \dot{B}_{p,q}^{\frac{2}{p} + \frac{2}{r} - 2})} 
    \leq C {\|f\|}_{\widetilde{L}_{x_{3}}^{r}(0,\infty ; \dot{B}_{p,q}^{\frac{2}{p} + \frac{1}{r} - 1})} 
    {\|g\|}_{\widetilde{L}_{x_{3}}^{r}(0,\infty ; \dot{B}_{p,q}^{\frac{2}{p} + \frac{1}{r} - 1})}. 
  \end{equation*}
\end{lem}

\section{Proof of Theorems}\label{0625-5}
\subsection{Proof of Theorem~\ref{0609-3}}
Let ${\bf a} \in \dot{B}_{p,q}^{\frac{2}{p} - 1}(\mathbb{R}^{2})$ 
and ${\bf F} \in \widetilde{L}_{x_{3}}^{r_{0}}(0,\infty; \dot{B}_{p,q}^{\frac{2}{p} + \frac{1}{r_{0}} - 2}(\mathbb{R}^{2}))$ 
satisfy the smallness condition \eqref{0609-5} for some constant $\delta_{0} > 0$ to be fixed later. 
We set 
\begin{equation}\label{0619-12}
  X := \left\{{\bf u} \in \widetilde{L}_{x_{3}}^{r}(0,\infty;\dot{B}_{p,q}^{\frac{2}{p} + \frac{1}{r} - 1}(\mathbb{R}^{2})) |\ {\|{\bf u}\|}_{\widetilde{L}_{x_{3}}^{r}(0,\infty;\dot{B}_{p,q}^{\frac{2}{p} + \frac{1}{r} - 1})} \leq \varepsilon_{0}\right\}, 
\end{equation}
where $\varepsilon_{0} > 0$ is some constant to be fixed later. 
For given ${\bf a}$, ${\bf F}$ and ${\bf u}\in X$, 
we obtain ${\bf w}^{(1)}$, ${\bf v}^{(1)}$, ${\bf w}^{(2)}$ and ${\bf v}^{(2)}$, 
which are represented as \eqref{0611-4}, \eqref{0611-5}, \eqref{0612-18} and \eqref{0612-19}, respectively. 
Thus, for any ${\bf u}\in X$ and any constant $C_{0} \geq 1$, we can define a map 
\begin{align}\label{1015-1}
  \Phi_{C_{0}}({\bf u}) 
  :=& P_{2^{j}> C_{0}|\Omega|^{\frac{1}{2}}}({\bf w}^{(1)}|_{\mathbb{R}_{+}^{3}} + {\bf v}^{(1)})\notag\\ 
  &+ \big(P_{C_{0}^{-1}|\Omega|^{\frac{1}{2}} \leq 2^{j} \leq C_{0}|\Omega|^{\frac{1}{2}}} + P_{2^{j} < C_{0}^{-1}|\Omega|^{\frac{1}{2}}}\big)({\bf w}^{(2)}|_{\mathbb{R}_{+}^{3}} + {\bf v}^{(2)}). 
\end{align}
If $\Phi_{C_{0}}$ has a fixed point in $X$, 
then it is a solution to \eqref{SNSC} satisfying \eqref{0609-1}. 

First, we show that there exists a constant $C_{0} \geq 1$ such that 
for any ${\bf u}\in X$, $\Phi_{C_{0}}({\bf u})\in X$. 
By Proposition~\ref{0611-11} for $s = 2/p -1$ and $r_{1} = r$, we have 
\begin{align*}
  &{\left\|P_{2^{j}> C_{0}|\Omega|^{\frac{1}{2}}}{\bf w}^{(1)}\right\|}_{\widetilde{L}_{x_{3}}^{r}(0,\infty;\dot{B}_{p,q}^{\frac{2}{p} + \frac{1}{r} - 1})}\\ 
  \leq& C {\left\|P_{2^{j}> C_{0}|\Omega|^{\frac{1}{2}}}{\bf F}\right\|}_{\widetilde{L}_{x_{3}}^{r_{0}}(0,\infty;\dot{B}_{p,q}^{\frac{2}{p} + \frac{1}{r_{0}} - 2})} 
  + C {\left\|P_{2^{j}> C_{0}|\Omega|^{\frac{1}{2}}}({\bf u}\otimes{\bf u})\right\|}_{\widetilde{L}_{x_{3}}^{\frac{r}{2}}(0,\infty;\dot{B}_{p,q}^{\frac{2}{p} + \frac{2}{r} - 2})}\\ 
  &+ C |\Omega|{\left\|P_{2^{j}> C_{0}|\Omega|^{\frac{1}{2}}}{\bf u}\right\|}_{\widetilde{L}_{x_{3}}^{r}(0,\infty;\dot{B}_{p,q}^{\frac{2}{p} + \frac{1}{r} - 3})}. 
\end{align*}
Note that thanks to Lemma~\ref{0619-2} \eqref{0619-1}, 
we get the uniform bounds without frequency restriction $P_{2^{j}> C_{0}|\Omega|^{\frac{1}{2}}}$ for the first and second terms. 
Moreover, using Lemma~\ref{0619-3}, we obtain 
\begin{equation*}
  |\Omega|{\left\|P_{2^{j}> C_{0}|\Omega|^{\frac{1}{2}}}{\bf u}\right\|}_{\widetilde{L}_{x_{3}}^{r}(0,\infty;\dot{B}_{p,q}^{\frac{2}{p} + \frac{1}{r} - 3})} 
  \leq \frac{C}{C_{0}^{2}}{\left\|P_{2^{j}> C_{0}|\Omega|^{\frac{1}{2}}}{\bf u}\right\|}_{\widetilde{L}_{x_{3}}^{r}(0,\infty;\dot{B}_{p,q}^{\frac{2}{p} + \frac{1}{r} - 1})}. 
\end{equation*}
Hence, we have 
\begin{align}\label{0619-5}
  &{\left\|P_{2^{j}> C_{0}|\Omega|^{\frac{1}{2}}}{\bf w}^{(1)}\right\|}_{\widetilde{L}_{x_{3}}^{r}(0,\infty;\dot{B}_{p,q}^{\frac{2}{p} + \frac{1}{r} - 1})}\notag\\ 
  \leq& C {\left\|{\bf F}\right\|}_{\widetilde{L}_{x_{3}}^{r_{0}}(0,\infty;\dot{B}_{p,q}^{\frac{2}{p} + \frac{1}{r_{0}} - 2})}
  + C {\left\|({\bf u}\otimes{\bf u})\right\|}_{\widetilde{L}_{x_{3}}^{\frac{r}{2}}(0,\infty;\dot{B}_{p,q}^{\frac{2}{p} + \frac{2}{r} - 2})} 
  + \frac{C}{C_{0}^{2}} {\left\|{\bf u}\right\|}_{\widetilde{L}_{x_{3}}^{r}(0,\infty;\dot{B}_{p,q}^{\frac{2}{p} + \frac{1}{r} - 1})}. 
\end{align}
Similarly, by Proposition~\ref{0611-13} for $s = 2/p -1$ and $r_{1} = r$, we get 
\begin{align}\label{0619-6}
  &{\left\|P_{2^{j} > C_{0}|\Omega|^{\frac{1}{2}}}{\bf v}^{(1)}\right\|}_{\widetilde{L}_{x_{3}}^{r}(0,\infty;\dot{B}_{p,q}^{\frac{2}{p} + \frac{1}{r} - 1})}\notag\\ 
  \leq& C {\left\|P_{2^{j} > C_{0}|\Omega|^{\frac{1}{2}}}{\bf a}\right\|}_{\dot{B}_{p,q}^{\frac{2}{p} - 1}} 
  + C {\left\|P_{2^{j} > C_{0}|\Omega|^{\frac{1}{2}}}{\bf F}\right\|}_{\widetilde{L}_{x_{3}}^{r_{0}}(0,\infty;\dot{B}_{p,q}^{\frac{2}{p} + \frac{1}{r_{0}} - 2})}\notag\\ 
  &+ C {\left\|P_{2^{j} > C_{0}|\Omega|^{\frac{1}{2}}}({\bf u}\otimes{\bf u})\right\|}_{\widetilde{L}_{x_{3}}^{\frac{r}{2}}(0,\infty;\dot{B}_{p,q}^{\frac{2}{p} + \frac{2}{r} - 2})} 
  + C|\Omega|{\left\|P_{2^{j} > C_{0}|\Omega|^{\frac{1}{2}}}{\bf u}\right\|}_{\widetilde{L}_{x_{3}}^{r}(0, \infty ; \dot{B}_{p,q}^{\frac{2}{p} + \frac{1}{r} - 3})}\notag\\ 
  \leq& C {\left\|{\bf a}\right\|}_{\dot{B}_{p,q}^{\frac{2}{p} - 1}} 
  + C {\left\|{\bf F}\right\|}_{\widetilde{L}_{x_{3}}^{r_{0}}(0,\infty;\dot{B}_{p,q}^{\frac{2}{p} + \frac{1}{r_{0}} - 2})}\notag\\ 
  &+ C {\left\|{\bf u}\otimes{\bf u}\right\|}_{\widetilde{L}_{x_{3}}^{\frac{r}{2}}(0,\infty;\dot{B}_{p,q}^{\frac{2}{p} + \frac{2}{r} - 2})} 
  + \frac{C}{C_{0}^{2}} {\left\|{\bf u}\right\|}_{\widetilde{L}_{x_{3}}^{r}(0,\infty;\dot{B}_{p,q}^{\frac{2}{p} + \frac{1}{r} - 1})}. 
\end{align}
Here, we note that the constant $C > 0$ in \eqref{0619-5} and \eqref{0619-6} is independent of $C_{0}$. 

We now establish estimates for ${\bf w}^{(2)}$ and ${\bf v}^{(2)}$. 
Using Proposition~\ref{0619-4} \eqref{0616-1} for $s = 2/p -1$ and $r_{1} = r$ and Lemma~\ref{0619-2} \eqref{0619-7}, we obtain 
\begin{align*}
  &{\left\|P_{C_{0}^{-1}|\Omega|^{\frac{1}{2}} \leq 2^{j} \leq C_{0}|\Omega|^{\frac{1}{2}}}{\bf w}^{(2)}\right\|}_{\widetilde{L}_{x_{3}}^{r}(0,\infty;\dot{B}_{p,q}^{\frac{2}{p} + \frac{1}{r} - 1})}\\ 
  \leq& C {\left\|P_{C_{0}^{-1}|\Omega|^{\frac{1}{2}} \leq 2^{j} \leq C_{0}|\Omega|^{\frac{1}{2}}}{\bf F}\right\|}_{\widetilde{L}_{x_{3}}^{r_{0}}(0,\infty;\dot{B}_{p,q}^{\frac{2}{p} + \frac{1}{r_{0}} - 2})}\\ 
  &+ C {\left\|P_{C_{0}^{-1}|\Omega|^{\frac{1}{2}} \leq 2^{j} \leq C_{0}|\Omega|^{\frac{1}{2}}}({\bf u}\otimes{\bf u})\right\|}_{\widetilde{L}_{x_{3}}^{\frac{r}{2}}(0,\infty;\dot{B}_{p,q}^{\frac{2}{p} + \frac{2}{r} - 2})}\\ 
  \leq& C {\left\|{\bf F}\right\|}_{\widetilde{L}_{x_{3}}^{r_{0}}(0,\infty;\dot{B}_{p,q}^{\frac{2}{p} + \frac{1}{r_{0}} - 2})} 
  + C {\left\|{\bf u}\otimes{\bf u}\right\|}_{\widetilde{L}_{x_{3}}^{\frac{r}{2}}(0,\infty;\dot{B}_{p,q}^{\frac{2}{p} + \frac{2}{r} - 2})}. 
\end{align*}
Also, by Proposition~\ref{0618-4} \eqref{0618-1} for $s = 2/p -1$ and $r_{1} = r$, we get 
\begin{align*}
  &{\left\|P_{C_{0}^{-1}|\Omega|^{\frac{1}{2}} \leq 2^{j} \leq C_{0}|\Omega|^{\frac{1}{2}}}{\bf v}^{(2)}\right\|}_{\widetilde{L}_{x_{3}}^{r}(0,\infty;\dot{B}_{p,q}^{\frac{2}{p} + \frac{1}{r} - 1})}\notag\\ 
  \leq& C {\left\|{\bf a}\right\|}_{\dot{B}_{p,q}^{\frac{2}{p} - 1}} 
  + C {\left\|{\bf F}\right\|}_{\widetilde{L}_{x_{3}}^{r_{0}}(0,\infty;\dot{B}_{p,q}^{\frac{2}{p} + \frac{1}{r_{0}} - 2})} 
  + C {\left\|{\bf u}\otimes{\bf u}\right\|}_{\widetilde{L}_{x_{3}}^{\frac{r}{2}}(0,\infty;\dot{B}_{p,q}^{\frac{2}{p} + \frac{2}{r} - 2})}. 
\end{align*}
We take $C_{0}$ sufficiently large. 
Then thanks to Proposition~\ref{0619-4} \eqref{0617-5} for $s = 2/p -1$ and $r_{1} = r$ , we have 
\begin{align}\label{0610-3}
  {\left\|P_{2^{j}< C_{0}^{-1}|\Omega|^{\frac{1}{2}}}{\bf w}^{(2)}\right\|}_{\widetilde{L}_{x_{3}}^{r}(0,\infty;\dot{B}_{p,q}^{\frac{2}{p} + \frac{1}{r} - 1})} 
  \leq&C |\Omega|^{- \frac{1}{r_{0}} + \frac{1}{r}}{\left\|P_{2^{j}< C_{0}^{-1}|\Omega|^{\frac{1}{2}}}{\bf F}\right\|}_{\widetilde{L}_{x_{3}}^{r_{0}}(0,\infty;\dot{B}_{p,q}^{\frac{2}{p} + \frac{3}{r_{0}} - \frac{2}{r} - 2})}\notag\\ 
  &+ C |\Omega|^{- \frac{1}{2} + \frac{1}{2r_{0}} - \frac{1}{2r}}{\left\|P_{2^{j}< C_{0}^{-1}|\Omega|^{\frac{1}{2}}}{\bf F}\right\|}_{\widetilde{L}_{x_{3}}^{r_{0}}(0,\infty;\dot{B}_{p,q}^{\frac{2}{p} + \frac{1}{r} - 1})}\notag\\ 
  &+ C |\Omega|^{-\frac{1}{r}}{\left\|P_{2^{j}< C_{0}^{-1}|\Omega|^{\frac{1}{2}}}({\bf u}\otimes{\bf u})\right\|}_{\widetilde{L}_{x_{3}}^{\frac{r}{2}}(0,\infty;\dot{B}_{p,q}^{\frac{2}{p} + \frac{4}{r} - 2})}\notag\\ 
  &+ C |\Omega|^{- \frac{1}{2} + \frac{1}{2r}}{\left\|P_{2^{j}< C_{0}^{-1}|\Omega|^{\frac{1}{2}}}({\bf u}\otimes{\bf u})\right\|}_{\widetilde{L}_{x_{3}}^{\frac{r}{2}}(0,\infty;\dot{B}_{p,q}^{\frac{2}{p} + \frac{1}{r} - 1})}. 
\end{align}
Since $r_{0} \leq r $, using Lemma~\ref{0617-1} \eqref{0617-3} and Lemma~\ref{0619-2} \eqref{0619-8} , we obtain 
\begin{align*}
  &{\left\|P_{2^{j}< C_{0}^{-1}|\Omega|^{\frac{1}{2}}}{\bf w}^{(2)}\right\|}_{\widetilde{L}_{x_{3}}^{r}(0,\infty;\dot{B}_{p,q}^{\frac{2}{p} + \frac{1}{r} - 1})}\\ 
  \leq& C {\left\|P_{2^{j}< C_{0}^{-1}|\Omega|^{\frac{1}{2}}}{\bf F}\right\|}_{\widetilde{L}_{x_{3}}^{r_{0}}(0,\infty;\dot{B}_{p,q}^{\frac{2}{p} + \frac{1}{r_{0}} - 2})} 
  + C {\left\|P_{2^{j}< C_{0}^{-1}|\Omega|^{\frac{1}{2}}}({\bf u}\otimes{\bf u})\right\|}_{\widetilde{L}_{x_{3}}^{\frac{r}{2}}(0,\infty;\dot{B}_{p,q}^{\frac{2}{p} + \frac{2}{r} - 2})}\\ 
  \leq& C {\left\|{\bf u}\otimes{\bf u}\right\|}_{\widetilde{L}_{x_{3}}^{\frac{r}{2}}(0,\infty;\dot{B}_{p,q}^{\frac{2}{p} + \frac{2}{r} - 2})} 
  + C {\left\|{\bf F}\right\|}_{\widetilde{L}_{x_{3}}^{r_{0}}(0,\infty;\dot{B}_{p,q}^{\frac{2}{p} + \frac{1}{r_{0}} - 2})}. 
\end{align*}
For the estimate of ${\bf v}^{(2)}$, by Proposition~\ref{0618-4} \eqref{0618-5}  and Lemma~\ref{0617-1} \eqref{0617-3}, we have 
\begin{align}\label{0610-4}
  &{\left\|P_{2^{j}< C_{0}^{-1}|\Omega|^{\frac{1}{2}}}{\bf v}^{(2)}\right\|}_{\widetilde{L}_{x_{3}}^{r}(0,\infty;\dot{B}_{p,q}^{\frac{2}{p} + \frac{1}{r} - 1})}\notag\\ 
  \leq& C |\Omega|^{-\frac{1}{2} + \frac{1}{r}}{\left\|P_{2^{j}< C_{0}^{-1}|\Omega|^{\frac{1}{2}}}{\bf a}\right\|}_{\dot{B}_{p,q}^{\frac{2}{p} - \frac{2}{r}}} 
  + C |\Omega|^{-\frac{1}{2r}}{\left\|P_{2^{j}< C_{0}^{-1}|\Omega|^{\frac{1}{2}}}{\bf a}\right\|}_{\dot{B}_{p,q}^{\frac{2}{p} + \frac{1}{r} - 1}}\notag\\ 
  &+ C |\Omega|^{-\frac{1}{2} - \frac{1}{r_{0}}+ \frac{1}{r}}{\left\|P_{2^{j}< C_{0}^{-1}|\Omega|^{\frac{1}{2}}}{\bf F}\right\|}_{\widetilde{L}_{x_{3}}^{r_{0}}(0,\infty;\dot{B}_{p,q}^{\frac{2}{p} + \frac{3}{r_{0}} - \frac{2}{r} - 1})}\notag\\ 
  & + C |\Omega|^{-1 + \frac{1}{2r_{0}} + \frac{1}{r}}{\left\|P_{2^{j}< C_{0}^{-1}|\Omega|^{\frac{1}{2}}}{\bf F}\right\|}_{\widetilde{L}_{x_{3}}^{r_{0}}(0,\infty;\dot{B}_{p,q}^{\frac{2}{p} - \frac{2}{r}})}\notag\\ 
  &+ C |\Omega|^{- \frac{1}{r_{0}} - \frac{1}{2r}}{\left\|P_{2^{j}< C_{0}^{-1}|\Omega|^{\frac{1}{2}}}{\bf F}\right\|}_{\widetilde{L}_{x_{3}}^{r_{0}}(0,\infty;\dot{B}_{p,q}^{\frac{2}{p} + \frac{3}{r_{0}} + \frac{1}{r} - 2})}\notag\\ 
  & + C |\Omega|^{-\frac{1}{2} + \frac{1}{2r_{0}} - \frac{1}{2r}}{\left\|P_{2^{j}< C_{0}^{-1}|\Omega|^{\frac{1}{2}}}{\bf F}\right\|}_{\widetilde{L}_{x_{3}}^{r_{0}}(0,\infty;\dot{B}_{p,q}^{\frac{2}{p} + \frac{1}{r} - 1})}\notag\\ 
  &+ C |\Omega|^{-\frac{1}{2} - \frac{1}{r}}{\left\|P_{2^{j}< C_{0}^{-1}|\Omega|^{\frac{1}{2}}}({\bf u}\otimes{\bf u})\right\|}_{\widetilde{L}_{x_{3}}^{\frac{r}{2}}(0,\infty;\dot{B}_{p,q}^{\frac{2}{p} + \frac{4}{r} - 1})}\notag\\ 
  &+ C |\Omega|^{- 1 + \frac{2}{r}}{\left\|P_{2^{j}< C_{0}^{-1}|\Omega|^{\frac{1}{2}}}({\bf u}\otimes{\bf u})\right\|}_{\widetilde{L}_{x_{3}}^{\frac{r}{2}}(0,\infty;\dot{B}_{p,q}^{\frac{2}{p} - \frac{2}{r}})}\notag\\ 
  &+ C |\Omega|^{- \frac{5}{2r}}{\left\|P_{2^{j}< C_{0}^{-1}|\Omega|^{\frac{1}{2}}}({\bf u}\otimes{\bf u})\right\|}_{\widetilde{L}_{x_{3}}^{\frac{r}{2}}(0,\infty;\dot{B}_{p,q}^{\frac{2}{p} + \frac{7}{r} - 2})}\notag\\ 
  &+ C |\Omega|^{- \frac{1}{2} + \frac{1}{2r}}{\left\|P_{2^{j}< C_{0}^{-1}|\Omega|^{\frac{1}{2}}}({\bf u}\otimes{\bf u})\right\|}_{\widetilde{L}_{x_{3}}^{\frac{r}{2}}(0,\infty;\dot{B}_{p,q}^{\frac{2}{p} + \frac{1}{r} - 1})}\notag\\ 
  \leq& C {\left\|{\bf a}\right\|}_{\dot{B}_{p,q}^{\frac{2}{p} - 1}} 
  + C {\left\|{\bf F}\right\|}_{\widetilde{L}_{x_{3}}^{r_{0}}(0,\infty;\dot{B}_{p,q}^{\frac{2}{p} + \frac{1}{r_{0}} - 2})} 
  + C {\left\|{\bf u}\otimes{\bf u}\right\|}_{\widetilde{L}_{x_{3}}^{\frac{r}{2}}(0,\infty;\dot{B}_{p,q}^{\frac{2}{p} + \frac{2}{r} - 2})}. 
\end{align}
Therefore, we obtain 
\begin{align*}
  {\|\Phi_{C_{0}}({\bf u})\|}_{\widetilde{L}_{x_{3}}^{r}(0,\infty;\dot{B}_{p,q}^{\frac{2}{p} + \frac{1}{r} - 1})} 
  \leq& C {\left\|{\bf a}\right\|}_{\dot{B}_{p,q}^{\frac{2}{p} - 1}} 
  + C {\left\|{\bf F}\right\|}_{\widetilde{L}_{x_{3}}^{r_{0}}(0,\infty;\dot{B}_{p,q}^{\frac{2}{p} + \frac{1}{r_{0}} - 2})}\\ 
  &+ C {\left\|{\bf u}\otimes{\bf u}\right\|}_{\widetilde{L}_{x_{3}}^{\frac{r}{2}}(0,\infty;\dot{B}_{p,q}^{\frac{2}{p} + \frac{2}{r} - 2})} 
  + \frac{\widetilde{C}}{C_{0}^{2}} {\left\|{\bf u}\right\|}_{\widetilde{L}_{x_{3}}^{r}(0,\infty;\dot{B}_{p,q}^{\frac{2}{p} + \frac{1}{r} - 1})}, 
\end{align*}
where the constant $\widetilde{C} > 0$ is independent of $C_{0}$. 
Since $p < 2r'$, using the bilinear estimate (Lemma~\ref{0619-9}), we get 
\begin{equation*}
  {\left\|{\bf u}\otimes{\bf u}\right\|}_{\widetilde{L}_{x_{3}}^{\frac{r}{2}}(0,\infty;\dot{B}_{p,q}^{\frac{2}{p} + \frac{2}{r} - 2})} 
  \leq C {\left\|{\bf u}\right\|}_{\widetilde{L}_{x_{3}}^{r}(0,\infty;\dot{B}_{p,q}^{\frac{2}{p} + \frac{1}{r} - 1})}^{2}. 
\end{equation*}
Thus, for $C_{0} \geq 1$, there exists a constant $C_{1} > 0$ such that for any ${\bf u}\in X$, we have 
\begin{align}\label{0610-1}
  {\|\Phi_{C_{0}}({\bf u})\|}_{\widetilde{L}_{x_{3}}^{r}(0,\infty;\dot{B}_{p,q}^{\frac{2}{p} + \frac{1}{r} - 1})} 
  \leq& C_{1} {\left\|{\bf a}\right\|}_{\dot{B}_{p,q}^{\frac{2}{p} - 1}} 
  + C_{1} {\left\|{\bf F}\right\|}_{\widetilde{L}_{x_{3}}^{r_{0}}(0,\infty;\dot{B}_{p,q}^{\frac{2}{p} + \frac{1}{r_{0}} - 2})}\notag\\ 
  &+ C_{1} {\left\|{\bf u}\right\|}_{\widetilde{L}_{x_{3}}^{r}(0,\infty;\dot{B}_{p,q}^{\frac{2}{p} + \frac{1}{r} - 1})}^{2} 
  + \frac{\widetilde{C}}{C_{0}^{2}} {\left\|{\bf u}\right\|}_{\widetilde{L}_{x_{3}}^{r}(0,\infty;\dot{B}_{p,q}^{\frac{2}{p} + \frac{1}{r} - 1})}. 
\end{align}
By the assumption \eqref{0609-5} on ${\bf a}$ and ${\bf F}$ and since ${\bf u}\in X$, it follows that 
\begin{align*}
  {\|\Phi_{C_{0}}({\bf u})\|}_{\widetilde{L}_{x_{3}}^{r}(0,\infty;\dot{B}_{p,q}^{\frac{2}{p} + \frac{1}{r} - 1})} 
  \leq& 3C_{1} \delta_{0} 
  + C_{1} \varepsilon_{0}^{2} 
  + \frac{\widetilde{C}}{C_{0}^{2}} \varepsilon_{0}. 
\end{align*}
Here, we can choose $C_{0}$ so that 
\begin{equation*}
  \frac{\widetilde{C}}{C_{0}^{2}} \leq \frac{1}{2}
\end{equation*}
holds. 
Also, let 
\begin{equation*}
  \varepsilon_{0} = \frac{1}{4C_{1}} 
  \quad 
  \text{and} 
  \quad 
  \delta_{0} = \frac{1}{48C_{1}^{2}}. 
\end{equation*}
Then we have 
\begin{equation*}
  {\|\Phi_{C_{0}}({\bf u})\|}_{\widetilde{L}_{x_{3}}^{r}(0,\infty;\dot{B}_{p,q}^{\frac{2}{p} + \frac{1}{r} - 1})} 
  \leq \varepsilon_{0}, 
\end{equation*}
which implies $\Phi_{C_{0}}({\bf u}) \in X$. 

Next, we verify that $\Phi_{C_{0}}$ is a contraction mapping on $X$. 
Let ${\bf u}, {\bf U} \in X$ and denote 
\begin{align*}
  \Phi_{C_{0}}({\bf U}) 
  =& P_{2^{j}> C_{0}|\Omega|^{\frac{1}{2}}}({\bf W}^{(1)}|_{\mathbb{R}_{+}^{3}} + {\bf V}^{(1)})\\ 
  &+ \big(P_{C_{0}^{-1}|\Omega|^{\frac{1}{2}} \leq 2^{j} \leq C_{0}|\Omega|^{\frac{1}{2}}} + P_{2^{j} < C_{0}^{-1}|\Omega|^{\frac{1}{2}}}\big)({\bf W}^{(2)}|_{\mathbb{R}_{+}^{3}} + {\bf V}^{(2)}). 
\end{align*}
Then we can write 
\begin{align*}
  \Phi_{C_{0}}({\bf u}) - \Phi_{C_{0}}({\bf U}) 
  =& P_{2^{j}> C_{0}|\Omega|^{\frac{1}{2}}}\left(({\bf w}^{(1)} - {\bf W}^{(1)})|_{\mathbb{R}_{+}^{3}} + ({\bf v}^{(1)} - {\bf V}^{(1)})\right)\\ 
  &+ P_{C_{0}^{-1}|\Omega|^{\frac{1}{2}} \leq 2^{j} \leq C_{0}|\Omega|^{\frac{1}{2}}}\left(({\bf w}^{(2)} - {\bf W}^{(2)})|_{\mathbb{R}_{+}^{3}} + ({\bf v}^{(2)} - {\bf V}^{(2)})\right)\\ 
  &+ P_{2^{j} < C_{0}^{-1}|\Omega|^{\frac{1}{2}}}\left(({\bf w}^{(2)} - {\bf W}^{(2)})|_{\mathbb{R}_{+}^{3}} + ({\bf v}^{(2)} - {\bf V}^{(2)})\right). 
\end{align*}
Note that ${\bf w}^{(1)} - {\bf W}^{(1)}$, ${\bf v}^{(1)} - {\bf V}^{(1)}$, ${\bf w}^{(2)} - {\bf W}^{(2)}$ and ${\bf v}^{(2)} - {\bf V}^{(2)}$ 
formally satisfy the equations below, respectively. 
\begin{equation*}
  \begin{cases}
  -\Delta({\bf w}^{(1)}-{\bf W}^{(1)})
  = -\mathbb{P}_{\mathbb{R}^{3}}
        \left(\text{div }
        (\bar{\bf u}\otimes\bar{\bf u}
        -\bar{\bf U}\otimes\bar{\bf U})\right)  
    +\Omega\partial_{3}(-\Delta)^{-1}
        \left(\nabla\times(\bar{\bf u}-\bar{\bf U})\right), \\ 
  &\hspace{-50pt} x\in\mathbb{R}^{3},\\
  \text{div }({\bf w}^{(1)}-{\bf W}^{(1)})=0,
  &\hspace{-50pt} x\in\mathbb{R}^{3}, 
  \end{cases}
\end{equation*}
\begin{equation*}
  \begin{cases}
    -\Delta ({\bf v}^{(1)} - {\bf V}^{(1)}) + \nabla (q^{(1)} - Q^{(1)}) = 0, 
          &\quad  \   x\in \mathbb{R}_{+}^{3}, \\
    \text{div } ({\bf v}^{(1)} - {\bf V}^{(1)}) = 0,
          &\quad  \   x\in \mathbb{R}_{+}^{3}, \\
    ({\bf v}^{(1)} - {\bf V}^{(1)})(x', 0) = - ({\bf w}^{(1)}-{\bf W}^{(1)})(x', 0),
          &\quad  \ x'\in{\mathbb{R}}^2, 
  \end{cases}
\end{equation*}
\begin{equation*}
  \begin{cases}
  -\Delta({\bf w}^{(2)}-{\bf W}^{(2)}) 
  - \Omega\partial_{3}(-\Delta)^{-1}
  \left(\nabla\times({\bf w}^{(2)} - {\bf W}^{(2)})\right)
  = - \mathbb{P}_{\mathbb{R}^{3}}
        \left(\text{div }
        (\bar{\bf u}\otimes\bar{\bf u}
        -\bar{\bf U}\otimes\bar{\bf U})\right), \\
  &\hspace{-50pt} x\in\mathbb{R}^{3},\\
  \text{div }({\bf w}^{(2)}-{\bf W}^{(2)})=0,
  &\hspace{-50pt} x\in\mathbb{R}^{3}, 
  \end{cases}
\end{equation*}
\begin{equation*}
  \begin{cases}
    -\Delta ({\bf v}^{(2)} - {\bf V}^{(2)}) +\Omega\left(e_{3}\times({\bf v}^{(2)} - {\bf V}^{(2)})\right) + \nabla (q^{(2)} - Q^{(2)}) = 0, 
          &\quad  \   x\in \mathbb{R}_{+}^{3}, \\
    \text{div } ({\bf v}^{(2)} - {\bf V}^{(2)}) = 0,
          &\quad  \   x\in \mathbb{R}_{+}^{3}, \\
    ({\bf v}^{(2)} - {\bf V}^{(2)})(x', 0) = - ({\bf w}^{(2)}-{\bf W}^{(2)})(x', 0),
          &\quad  \ x'\in{\mathbb{R}}^2. 
  \end{cases}
\end{equation*}
Also, we can write 
\begin{equation*}
  \bar{\bf u}\otimes\bar{\bf u} - \bar{\bf U}\otimes\bar{\bf U} 
  = \bar{\bf u}\otimes\left(\bar{\bf u} - \bar{\bf U}\right) 
  + \left(\bar{\bf u} - \bar{\bf U}\right)\otimes\bar{\bf U}. 
\end{equation*}
Therefore, by an argument similar to that for the estimate of $\Phi_{C_{0}}({\bf u})$ \eqref{0610-1}, 
we get 
\begin{align}\label{0610-2}
  &{\|\Phi_{C_{0}}({\bf u}) - \Phi_{C_{0}}({\bf U})\|}_{\widetilde{L}_{x_{3}}^{r}(0,\infty;\dot{B}_{p,q}^{\frac{2}{p} + \frac{1}{r} - 1})}\notag\\ 
  \leq& C {\left\|{\bf u}\otimes\left({\bf u} - {\bf U}\right) + \left({\bf u} - {\bf U}\right)\otimes{\bf U}\right\|}_{\widetilde{L}_{x_{3}}^{\frac{r}{2}}(0,\infty;\dot{B}_{p,q}^{\frac{2}{p} + \frac{2}{r} - 2})} 
  + \frac{\widetilde{C}}{C_{0}^{2}} {\left\|{\bf u} - {\bf U}\right\|}_{\widetilde{L}_{x_{3}}^{r}(0,\infty;\dot{B}_{p,q}^{\frac{2}{p} + \frac{1}{r} - 1})}. 
\end{align}
Here the constant $\widetilde{C} > 0$ is independent of $C_{0}$. 
Using Lemma~\ref{0619-9}, it follows that 
\begin{align*}
  &{\left\|{\bf u}\otimes\left({\bf u} - {\bf U}\right) + \left({\bf u} - {\bf U}\right)\otimes{\bf U}\right\|}_{\widetilde{L}_{x_{3}}^{\frac{r}{2}}(0,\infty;\dot{B}_{p,q}^{\frac{2}{p} + \frac{2}{r} - 2})}\\ 
  \leq& C \left({\left\|{\bf u}\right\|}_{\widetilde{L}_{x_{3}}^{r}(0,\infty;\dot{B}_{p,q}^{\frac{2}{p} + \frac{1}{r} - 1})} 
  + {\left\|{\bf U}\right\|}_{\widetilde{L}_{x_{3}}^{r}(0,\infty;\dot{B}_{p,q}^{\frac{2}{p} + \frac{1}{r} - 1})}\right) 
  {\left\|{\bf u} - {\bf U}\right\|}_{\widetilde{L}_{x_{3}}^{r}(0,\infty;\dot{B}_{p,q}^{\frac{2}{p} + \frac{1}{r} - 1})}
\end{align*}
Thus, there exists a constant $C_{2} > 0$ such that for any ${\bf u}, {\bf U}\in X$, we have 
\begin{align*}
  &{\|\Phi_{C_{0}}({\bf u})- \Phi_{C_{0}}({\bf U})\|}_{\widetilde{L}_{x_{3}}^{r}(0,\infty;\dot{B}_{p,q}^{\frac{2}{p} + \frac{1}{r} - 1})}\\ 
  \leq& C_{2}\left({\left\|{\bf u}\right\|}_{\widetilde{L}_{x_{3}}^{r}(0,\infty;\dot{B}_{p,q}^{\frac{2}{p} + \frac{1}{r} - 1})} 
  + {\left\|{\bf U}\right\|}_{\widetilde{L}_{x_{3}}^{r}(0,\infty;\dot{B}_{p,q}^{\frac{2}{p} + \frac{1}{r} - 1})}\right){\left\|{\bf u} - {\bf U}\right\|}_{\widetilde{L}_{x_{3}}^{r}(0,\infty;\dot{B}_{p,q}^{\frac{2}{p} + \frac{1}{r} - 1})}\\ 
  &+ \frac{\widetilde{C}}{C_{0}^{2}} {\left\|{\bf u} - {\bf U}\right\|}_{\widetilde{L}_{x_{3}}^{r}(0,\infty;\dot{B}_{p,q}^{\frac{2}{p} + \frac{1}{r} - 1})}. 
\end{align*}
We take $C_{0}$ sufficiently large such that 
\begin{equation*}
  \frac{\widetilde{C}}{C_{0}^{2}} 
  \leq \frac{1}{4}
\end{equation*}
holds. 
Since ${\bf u}, {\bf U} \in X$, we obtain 
\begin{equation*}
  {\|\Phi_{C_{0}}({\bf u})- \Phi_{C_{0}}({\bf U})\|}_{\widetilde{L}_{x_{3}}^{r}(0,\infty;\dot{B}_{p,q}^{\frac{2}{p} + \frac{1}{r} - 1})} 
  \leq \left(2C_{2}\varepsilon_{0} + \frac{1}{4}\right){\left\|{\bf u} - {\bf U}\right\|}_{\widetilde{L}_{x_{3}}^{r}(0,\infty;\dot{B}_{p,q}^{\frac{2}{p} + \frac{1}{r} - 1})}. 
\end{equation*}
Hence, taking 
\begin{equation*}
  \varepsilon_{0} = \frac{1}{8C_{2}}, 
\end{equation*}
we have 
\begin{equation*}
  {\|\Phi_{C_{0}}({\bf u})- \Phi_{C_{0}}({\bf U})\|}_{\widetilde{L}_{x_{3}}^{r}(0,\infty;\dot{B}_{p,q}^{\frac{2}{p} + \frac{1}{r} - 1})}\\ 
  \leq \frac{1}{2}{\left\|{\bf u} - {\bf U}\right\|}_{\widetilde{L}_{x_{3}}^{r}(0,\infty;\dot{B}_{p,q}^{\frac{2}{p} + \frac{1}{r} - 1})}. 
\end{equation*}
This shows that $\Phi_{C_{0}}$ is a contraction mapping on $X$. 

From the above, we choose $C_{0}$ sufficiently large, 
and choose $\varepsilon_{0}$, $\delta_{0}$ so that 
\begin{equation*}
  \varepsilon_{0} = \min\left\{\frac{1}{4C_{1}}, \frac{1}{8C_{2}}\right\}
  \quad \text{and}\quad 
  \delta_{0} = \frac{\varepsilon_{0}}{12C_{1}} 
\end{equation*}
is satisfied. 
Then it is possible to apply Banach's fixed point theorem, 
that is, there exists a unique fixed point ${\bf u}\in X$ of $\Phi_{C_{0}}$, 
which satisfies 
\begin{equation}\label{1007-1}
  {\bf u} 
  = P_{2^{j}> C_{0}|\Omega|^{\frac{1}{2}}}({\bf w}^{(1)}|_{\mathbb{R}_{+}^{3}} + {\bf v}^{(1)}) 
  + \Big(P_{C_{0}^{-1}|\Omega|^{\frac{1}{2}} \leq 2^{j} \leq C_{0}|\Omega|^{\frac{1}{2}}} + P_{2^{j} < C_{0}^{-1}|\Omega|^{\frac{1}{2}}}\Big)({\bf w}^{(2)}|_{\mathbb{R}_{+}^{3}} + {\bf v}^{(2)}) 
\end{equation}
in $X$. 

Finally, we prove that ${\bf u} \in C_{x_{3}}([0,\infty) ; \dot{B}_{p,q}^{\frac{2}{p} - 1}(\mathbb{R}^{2}))$. 
Using Proposition~\ref{0611-11}, Proposition~\ref{0611-13}, Proposition~\ref{0619-4} and Proposition~\ref{0618-4} with $r = \infty$ and Lemma~\ref{0619-9} , we obtain 
\begin{align*}
  {\|{\bf u}\|}_{\widetilde{L}_{x_{3}}^{\infty}(0,\infty;\dot{B}_{p,q}^{\frac{2}{p} - 1})} 
  \leq& {\left\|P_{2^{j}> C_{0}|\Omega|^{\frac{1}{2}}}({\bf w}^{(1)} + {\bf v}^{(1)})\right\|}_{\widetilde{L}_{x_{3}}^{\infty}(0,\infty;\dot{B}_{p,q}^{\frac{2}{p} - 1})}\\ 
  &+ {\left\|P_{C_{0}^{-1}|\Omega|^{\frac{1}{2}} \leq 2^{j} \leq C_{0}|\Omega|^{\frac{1}{2}}}({\bf w}^{(2)} + {\bf v}^{(2)})\right\|}_{\widetilde{L}_{x_{3}}^{\infty}(0,\infty;\dot{B}_{p,q}^{\frac{2}{p} - 1})}\\ 
  &+ {\left\|P_{2^{j} < C_{0}^{-1}|\Omega|^{\frac{1}{2}}}({\bf w}^{(2)} + {\bf v}^{(2)})\right\|}_{\widetilde{L}_{x_{3}}^{\infty}(0,\infty;\dot{B}_{p,q}^{\frac{2}{p} - 1})}\\ 
  \leq& C {\left\|{\bf a}\right\|}_{\dot{B}_{p,q}^{\frac{2}{p} - 1}} 
  + C {\left\|{\bf F}\right\|}_{\widetilde{L}_{x_{3}}^{r_{0}}(0,\infty;\dot{B}_{p,q}^{\frac{2}{p} + \frac{1}{r_{0}} - 2})}\\ 
  &+ C {\left\|{\bf u}\otimes{\bf u}\right\|}_{\widetilde{L}_{x_{3}}^{\frac{r}{2}}(0,\infty;\dot{B}_{p,q}^{\frac{2}{p} + \frac{2}{r} - 2})} 
  + C{\left\|{\bf u}\right\|}_{\widetilde{L}_{x_{3}}^{r}(0,\infty;\dot{B}_{p,q}^{\frac{2}{p} + \frac{1}{r} - 1})}\\ 
  \leq& C {\left\|{\bf a}\right\|}_{\dot{B}_{p,q}^{\frac{2}{p} - 1}} 
  + C {\left\|{\bf F}\right\|}_{\widetilde{L}_{x_{3}}^{r_{0}}(0,\infty;\dot{B}_{p,q}^{\frac{2}{p} + \frac{1}{r_{0}} - 2})}\\ 
  &+ C {\left\|{\bf u}\right\|}_{\widetilde{L}_{x_{3}}^{r}(0,\infty;\dot{B}_{p,q}^{\frac{2}{p} + \frac{1}{r} - 1})}^{2} 
  + C{\left\|{\bf u}\right\|}_{\widetilde{L}_{x_{3}}^{r}(0,\infty;\dot{B}_{p,q}^{\frac{2}{p} + \frac{1}{r} - 1})}\\ 
  <& \infty. 
\end{align*}
It follows that ${\bf u} \in \widetilde{L}_{x_{3}}^{\infty}(0,\infty;\dot{B}_{p,q}^{\frac{2}{p} - 1})$. 
By Minkowski's inequality, we get 
\begin{equation*}
  {\bf u}\in 
  \widetilde{L}_{x_{3}}^{\infty}(0,\infty;\dot{B}_{p,q}^{\frac{2}{p} - 1}(\mathbb{R}^{2})) 
  \subset L_{x_{3}}^{\infty}(0,\infty;\dot{B}_{p,q}^{\frac{2}{p} - 1}(\mathbb{R}^{2})). 
\end{equation*}
Thus, for almost everywhere $x_{3} > 0$, ${\bf u}(\cdot, x_{3})\in \dot{B}_{p,q}^{\frac{2}{p} - 1}(\mathbb{R}^{2})$. 
We show that 
\begin{equation*}
  \lim_{x_{3}\to 0}{\|{\bf u}(\cdot, x_{3}) - {\bf a}\|}_{\dot{B}_{p,q}^{\frac{2}{p} - 1}} = 0, 
\end{equation*}
which implies that ${\bf u}$ is continuous at $x_{3} = 0$. 
For any $M > 0$, we can write 
\begin{equation*}
  {\bf u} - {\bf a} 
  = \sum_{|j|\leq M}\phi_{j}*({\bf u} - {\bf a}) 
  + \sum_{|j|> M}\phi_{j}*({\bf u} - {\bf a}). 
\end{equation*}
Hence we have 
\begin{equation*}
  {\|{\bf u}(\cdot, x_{3}) - {\bf a}\|}_{\dot{B}_{p,q}^{\frac{2}{p} - 1}} 
  \leq {\Big\|\sum_{|j|\leq M}\phi_{j}*({\bf u}(\cdot, x_{3}) - {\bf a})\Big\|}_{\dot{B}_{p,q}^{\frac{2}{p} - 1}} 
  + {\Big\|\sum_{|j|> M}\phi_{j}*({\bf u}(\cdot, x_{3}) - {\bf a})\Big\|}_{\dot{B}_{p,q}^{\frac{2}{p} - 1}}. 
\end{equation*}
Since $q < \infty$, for any $\varepsilon > 0$, there exists $M > 0$ such that 
\begin{equation*}
  {\Big\|\sum_{|j|> M}\phi_{j}*({\bf u}(\cdot, x_{3}) - {\bf a})\Big\|}_{\dot{B}_{p,q}^{\frac{2}{p} - 1}} 
  < \frac{\varepsilon}{2}. 
\end{equation*}
From this, it suffices to prove that for any fixed $|j| \leq M$, 
\begin{equation}\label{1010-5}
  \lim_{x_{3}\to 0}{\|\phi_{j}*({\bf u}(\cdot, x_{3}) - {\bf a})\|}_{L^{p}} = 0. 
\end{equation}
Furthermore, by the representation of ${\bf u}$ \eqref{1007-1}, 
it is enough to show that for any fixed $j$ with $2^{j} > C_{0}|\Omega|^{\frac{1}{2}}$, 
\begin{equation}\label{1010-3}
  \lim_{x_{3}\to 0}{\|\phi_{j}*({\bf w}^{(1)}|_{\mathbb{R}_{+}^{3}}(\cdot, x_{3}) + {\bf v}^{(1)}(\cdot, x_{3}) - {\bf a})\|}_{L^{p}} = 0, 
\end{equation}
and for any fixed $j$ with $2^{j} \leq  C_{0}|\Omega|^{\frac{1}{2}}$
\begin{equation}\label{1010-4}
  \lim_{x_{3}\to 0}{\|\phi_{j}*({\bf w}^{(2)}|_{\mathbb{R}_{+}^{3}}(\cdot, x_{3}) + {\bf v}^{(2)}(\cdot, x_{3}) - {\bf a})\|}_{L^{p}} = 0. 
\end{equation}

First, we consider the case $2^{j} > C_{0}|\Omega|^{\frac{1}{2}}$. 
We show that 
\begin{equation}\label{1010-1}
  \lim_{x_{3}\to 0}{\|\phi_{j}*({\bf w}^{(1)}|_{\mathbb{R}_{+}^{3}}(\cdot, x_{3}) - {\bf w}^{(1)}(\cdot, 0))\|}_{L^{p}} = 0. 
\end{equation}
In view of the explicit formulas for ${\bf w}^{(1)}$ in \eqref{0611-4}, 
it suffices to show that $T_{K_{l}^{(1)}}^{(\pm)}[{\bf F} - {\bf u}\otimes{\bf u}]$ ($l = 1,2,\cdots, 5$) defined by \eqref{0611-2} is continuous with respect to $x_{3}$. 
We only check the case $T_{K_{1}^{(1)}}^{(+)}$, since the other cases can be treated similarly. 
We recall that 
\begin{equation*}
  T_{K_{1}^{(1)}}^{(+)}[{\bf F}](x', x_{3}) 
  = \frac{1}{2}\int_{0}^{+\infty}\mathcal{F}_{\mathbb{R}^{2}}^{-1}\left[\frac{1}{|\xi'|}(e^{-|\xi'||x_{3} - y_{3}|} + e^{-|\xi'|(x_{3} + y_{3})})(\widehat{\bf F})(\xi', y_{3})\right](x')~{\rm d}y_{3}. 
\end{equation*}
Hence we have 
\begin{align*}
  &T_{K_{1}^{(1)}}^{(+)}[{\bf F}](x', x_{3}) - T_{K_{1}^{(1)}}^{(+)}[{\bf F}](x', 0)\\ 
  =& \frac{1}{2}\int_{0}^{+\infty}\mathcal{F}_{\mathbb{R}^{2}}^{-1}\left[\frac{1}{|\xi'|}(e^{-|\xi'||x_{3} - y_{3}|} + e^{-|\xi'|(x_{3} + y_{3})} - 2e^{-|\xi'|y_{3}})\widehat{\bf F}(\xi', y_{3})\right](x')~{\rm d}y_{3}\\ 
  =& \frac{1}{2}\int_{0}^{x_{3}}\mathcal{F}_{\mathbb{R}^{2}}^{-1}\left[\frac{1}{|\xi'|}(e^{-|\xi'|(x_{3} - y_{3})} + e^{-|\xi'|(x_{3} + y_{3})} - 2e^{-|\xi'|y_{3}})\widehat{\bf F}(\xi', y_{3})\right](x')~{\rm d}y_{3}\\ 
  &+ \frac{1}{2}\int_{x_{3}}^{+\infty}\mathcal{F}_{\mathbb{R}^{2}}^{-1}\left[\frac{1}{|\xi'|}(e^{-|\xi'|(y_{3} - x_{3})} + e^{-|\xi'|(x_{3} + y_{3})} - 2e^{-|\xi'|y_{3}})\widehat{\bf F}(\xi', y_{3})\right](x')~{\rm d}y_{3}. 
\end{align*}
For the first term on the right-hand side, the boundedness of the operator $\phi_{j}*$ and the Poisson semigroup yields
\begin{align*}
  &{\left\|\phi_{j}*\left(\int_{0}^{x_{3}}\mathcal{F}_{\mathbb{R}^{2}}^{-1}\left[\frac{1}{|\xi'|}(e^{-|\xi'|(x_{3} - y_{3})} + e^{-|\xi'|(x_{3} + y_{3})} - 2e^{-|\xi'|y_{3}})\widehat{\bf F}(\xi', y_{3})\right](x')~{\rm d}y_{3}\right)\right\|}_{L^{p}}\\ 
  \leq& C2^{-j}\int_{0}^{x_{3}}{\|\phi_{j}*{\bf F}(\cdot, y_{3})\|}_{L^{p}}~{\rm d}y_{3}. 
\end{align*}
When $r_{0} = 1$, by the dominated convenience theorem, it follows that 
\begin{equation*}
  \lim_{x_{3} \to 0}\int_{0}^{x_{3}}{\|\phi_{j}*{\bf F}(\cdot, y_{3})\|}_{L^{p}}~{\rm d}y_{3} = 0. 
\end{equation*}
When $r_{0} > 1$, using H\"{o}lder's inequality, we get 
\begin{equation*}
  \int_{0}^{x_{3}}{\|\phi_{j}*{\bf F}(\cdot, y_{3})\|}_{L^{p}}~{\rm d}y_{3} 
  \leq x_{3}^{\frac{1}{r_{0}'}}{\|\phi_{j}*{\bf F}\|}_{L^{r_{0}}(0, \infty; L^{p})} 
  \to 0,\ \text{as}\ x_{3}\to 0. 
\end{equation*}
For the second term, it follows that 
\begin{align*}
  &{\left\|\phi_{j}*\left(\int_{x_{3}}^{+\infty}\mathcal{F}_{\mathbb{R}^{2}}^{-1}\left[\frac{1}{|\xi'|}(e^{-|\xi'|(y_{3} - x_{3})} + e^{-|\xi'|(x_{3} + y_{3})} - 2e^{-|\xi'|y_{3}})\widehat{\bf F}(\xi', y_{3})\right](x')~{\rm d}y_{3}\right)\right\|}_{L^{p}}\\ 
  \leq& C2^{-j}\int_{x_{3}}^{+\infty}{\|\phi_{j}*e^{-|\xi'|y_{3}}(e^{|\xi'|x_{3}} + e^{-|\xi'|x_{3}} - 2Id){\bf F}(\cdot, y_{3})\|}_{L^{p}}~{\rm d}y_{3}. 
\end{align*}
By the continuity of the Poisson semigroup, for almost every $y_{3}$, we have 
\begin{equation*}
  \lim_{x_{3} \to 0}{\|\phi_{j}*e^{-|\xi'|y_{3}}(e^{|\xi'|x_{3}} + e^{-|\xi'|x_{3}} - 2Id){\bf F}(\cdot, y_{3})\|}_{L^{p}}
  =0. 
\end{equation*}
Also, using Lemma~\ref{0611-10} and H\"{o}lder's inequality, we obtain 
\begin{align*}
  &\int_{x_{3}}^{+\infty}{\|\phi_{j}*e^{-|\xi'|y_{3}}(e^{|\xi'|x_{3}} + e^{-|\xi'|x_{3}} - 2Id){\bf F}(\cdot, y_{3})\|}_{L^{p}}~{\rm d}y_{3}\\ 
  \leq& C\int_{0}^{+\infty}e^{-c2^{j}(y_{3} - x_{3})}{\|\phi_{j}*{\bf F}(\cdot, y_{3})\|}_{L^{p}}~{\rm d}y_{3}. 
\end{align*}
We can assume that $x_{3}\leq 1$. Thus we have 
\begin{equation*}
  e^{-c2^{j}(y_{3} - x_{3})} 
  \leq e^{c2^{j}}e^{-c2^{j}y_{3}}. 
\end{equation*}
Using H\"{o}lder's inequality, it follows that 
\begin{equation*}
  \int_{0}^{+\infty}e^{-c2^{j}y_{3}}{\|\phi_{j}*{\bf F}(\cdot, y_{3})\|}_{L^{p}}~{\rm d}y_{3}\\ 
  \leq \left(\int_{0}^{+\infty}e^{-cr_{0}'2^{j}y_{3}}~{\rm d}y_{3}\right)^{\frac{1}{r_{0}'}}{\|\phi_{j}*{\bf F}\|}_{L^{r_{0}}(0, \infty;L^{p})} 
  <\infty. 
\end{equation*}
Therefore, by the dominated convergence theorem, we obtain 
\begin{equation*}
  \lim_{x_{3}\to 0}\int_{x_{3}}^{+\infty}{\|\phi_{j}*e^{-|\xi'|y_{3}}(e^{|\xi'|x_{3}} + e^{-|\xi'|x_{3}} - 2Id){\bf F}(\cdot, y_{3})\|}_{L^{p}}~{\rm d}y_{3} 
  = 0. 
\end{equation*}
Thus, we have 
\begin{equation*}
  \lim_{x_{3}\to 0}{\big\|\phi_{j}*(T_{K_{1}^{(1)}}^{(+)}[{\bf F}](\cdot, x_{3}) - T_{K_{1}^{(1)}}^{(+)}[{\bf F}](\cdot, 0))\big\|}_{L^{p}} = 0. 
\end{equation*}
Similarly, it follows that 
\begin{equation*}
  \lim_{x_{3}\to 0}{\big\|\phi_{j}*(T_{K_{1}^{(1)}}^{(+)}[{\bf u}\otimes{\bf u}](\cdot, x_{3}) - T_{K_{1}^{(1)}}^{(+)}[{\bf u}\otimes{\bf u}](\cdot, 0))\big\|}_{L^{p}} = 0. 
\end{equation*}
By the same argument, for $l = 1, 2, \cdots , 5$, we obtain 
\begin{equation*}
  \lim_{x_{3}\to 0}{\big\|\phi_{j}*(T_{K_{l}^{(1)}}^{(\pm)}[{\bf F} - {\bf u}\otimes{\bf u}](\cdot, x_{3}) - T_{K_{l}^{(1)}}^{(\pm)}[{\bf F} - {\bf u}\otimes{\bf u}](\cdot, 0))\big\|}_{L^{p}} = 0. 
\end{equation*}
From the above, we have \eqref{1010-1}. 
For ${\bf v}^{(1)}$, it follows from the explicit formulas \eqref{0611-5} and the continuity of the Poisson semigroup that 
\begin{equation}\label{1010-2}
  \lim_{x_{3}\to 0}{\|\phi_{j}*({\bf v}^{(1)}(\cdot, x_{3}) - {\bf a} + {\bf w}^{(1)}(\cdot, 0))\|}_{L^{p}} 
  = 0. 
\end{equation}
By \eqref{1010-1} and \eqref{1010-2}, we obtain \eqref{1010-3}. 

Next, we consider the case $2^{j}\leq C_{0}|\Omega|^{\frac{1}{2}}$. 
For the eigenvalues $\lambda_{|\Omega|,k}$ ($k = 1, 2, 3$) given by \eqref{0625-9}, 
which appear in the explicit formulas \eqref{0612-18} and \eqref{0612-19} for ${\bf w}^{(2)}$ and ${\bf v}^{(2)}$, respectively, 
the Fourier multiplier theorem yields 
\begin{equation*}
  \lim_{x_{3}\to 0}{\left\|\mathcal{F}_{\mathbb{R}^{2}}^{-1}\left[\widehat{\phi}_{j}(Id - e^{-x_{3}\lambda_{|\Omega|, k}})\widehat{f}\right]\right\|}_{L^{p}}=0 
\end{equation*}
for any $f$ with $\phi_{j}*f \in L^{p}(\mathbb{R}^{2})$. 
Using this fact, we can prove \eqref{1010-4} by an argument similar to that for the case $2^{j} > C_{0}|\Omega|^{\frac{1}{2}}$. 
Therefore, we obtain \eqref{1010-5} and 
\begin{equation*}
  \lim_{x_{3}\to 0}{\|{\bf u}(\cdot, x_{3}) - {\bf a}\|}_{\dot{B}_{p,q}^{\frac{2}{p} - 1}} = 0. 
\end{equation*}
Similarly, for any $\tilde{x}_{3} > 0$, we have 
\begin{equation*}
  \lim_{x_{3}\to \tilde{x}_{3}}{\|{\bf u}(\cdot, x_{3}) - {\bf u}(\cdot, \tilde{x}_{3})\|}_{\dot{B}_{p,q}^{\frac{2}{p} - 1}} = 0, 
\end{equation*}
which implies ${\bf u} \in C_{x_{3}}([0,\infty) ; \dot{B}_{p,q}^{\frac{2}{p} - 1}(\mathbb{R}^{2}))$. 
This completes the proof of Theorem~\ref{0609-3}. 

\subsection{Proof of Theorem~\ref{0609-4}}\label{1010-6}
The estimate used in the proof of Theorem~\ref{0609-4} are similar to those used in the proof of Theorem~\ref{0609-3}. 
We show that there exists a sufficiently large constant $C_{0} \geq 1$ such that 
there exists a constant $\varepsilon_{0} > 0$ such that 
for any boundary data ${\bf a}^{(1)} \in \dot{B}_{p,q}^{\frac{2}{p} - 1}(\mathbb{R}^{2})$ with $a_{3} = 0$ and any ${\bf F} \in \widetilde{L}_{x_{3}}^{r_{0}}(0,\infty; \dot{B}_{p,q}^{\frac{2}{p} + \frac{1}{r_{0}} - 2}(\mathbb{R}^{2}))$, 
there exists a sufficiently large constant $\Omega_{0} > 0$ such that 
for any $\Omega\in \mathbb{R}\setminus\{0\}$ with $|\Omega| \geq \Omega_{0}$, 
the mapping $\Phi_{C_{0}}$ defined in \eqref{1015-1} is a contraction mapping on $X$. 
Here $X$ is defined by \eqref{0619-12}. 
Note that estimate~\eqref{0610-2} holds independently of ${\bf a}$ and ${\bf F}$. 
Hence it suffices to prove that for any ${\bf u}\in X$, $\Phi_{C_{0}}({\bf u}) \in X$. 

By Proposition~\ref{0611-11}, Proposition~\ref{0611-13}, Proposition~\ref{0619-4} \eqref{0616-1}, Proposition~\ref{0618-4} \eqref{0618-1}, Lemma~\ref{0619-2} \eqref{0619-1} \eqref{0619-7} and Lemma~\ref{0619-7}, we have 
\begin{align*}
  &{\left\|P_{2^{j}> C_{0}|\Omega|^{\frac{1}{2}}}({\bf w}^{(1)} + {\bf v}^{(1)}) + P_{C_{0}^{-1}|\Omega|^{\frac{1}{2}} \leq 2^{j} \leq C_{0}|\Omega|^{\frac{1}{2}}}({\bf w}^{(2)} + {\bf v}^{(2)})\right\|}_{\widetilde{L}_{x_{3}}^{r}(0,\infty;\dot{B}_{p,q}^{\frac{2}{p} + \frac{1}{r} - 1})}\\ 
  \leq & C {\left\|P_{2^{j}> C_{0}|\Omega|^{\frac{1}{2}}}{\bf a}\right\|}_{\dot{B}_{p,q}^{\frac{2}{p} - 1}} 
  + C {\left\|P_{C_{0}^{-1}|\Omega|^{\frac{1}{2}} \leq 2^{j}\leq  C_{0}|\Omega|^{\frac{1}{2}}}{\bf a}\right\|}_{\dot{B}_{p,q}^{\frac{2}{p} - 1}}\\ 
  &+ C {\left\|P_{2^{j}> C_{0}|\Omega|^{\frac{1}{2}}}{\bf F}\right\|}_{\widetilde{L}_{x_{3}}^{r_{0}}(0,\infty;\dot{B}_{p,q}^{\frac{2}{p} + \frac{1}{r_{0}} - 2})} 
  + C {\left\|P_{C_{0}^{-1}|\Omega|^{\frac{1}{2}} \leq 2^{j}\leq  C_{0}|\Omega|^{\frac{1}{2}}}{\bf F}\right\|}_{\widetilde{L}_{x_{3}}^{r_{0}}(0,\infty;\dot{B}_{p,q}^{\frac{2}{p} + \frac{1}{r_{0}} - 2})}\\ 
  &+ C {\left\|{\bf u}\otimes{\bf u}\right\|}_{\widetilde{L}_{x_{3}}^{\frac{r}{2}}(0,\infty;\dot{B}_{p,q}^{\frac{2}{p} + \frac{2}{r} - 2})} 
  + \frac{\widetilde{C}}{C_{0}^{2}} {\left\|{\bf u}\right\|}_{\widetilde{L}_{x_{3}}^{r}(0,\infty;\dot{B}_{p,q}^{\frac{2}{p} + \frac{1}{r} - 1})}, 
\end{align*}
where the constant $\widetilde{C} > 0$ is independent of $C_{0}$. 
Thanks to $q < \infty$, by Lemma~\ref{0619-15}, for any fixed $C_{0}$ and any $\varepsilon_{1} > 0$, there exists a sufficiently large constant $\Omega_{1}>0$ such that 
\begin{equation*}
  {\left\|P_{2^{j}> C_{0}\Omega_{1}^{\frac{1}{2}}}{\bf a}\right\|}_{\dot{B}_{p,q}^{\frac{2}{p} - 1}},\ 
  {\left\|P_{2^{j}> C_{0}\Omega_{1}^{\frac{1}{2}}}{\bf F}\right\|}_{\widetilde{L}_{x_{3}}^{r_{0}}(0,\infty;\dot{B}_{p,q}^{\frac{2}{p} + \frac{1}{r_{0}} - 2})} 
  \leq \varepsilon_{1}, 
\end{equation*}
and 
\begin{equation*}
  {\left\|P_{C_{0}^{-1}\Omega_{1}^{\frac{1}{2}} < 2^{j}\leq  C_{0}\Omega_{1}^{\frac{1}{2}}}{\bf a}\right\|}_{\dot{B}_{p,q}^{\frac{2}{p} - 1}},\ 
  {\left\|P_{C_{0}^{-1}\Omega_{1}^{\frac{1}{2}} \leq 2^{j}\leq  C_{0}\Omega_{1}^{\frac{1}{2}}}{\bf F}\right\|}_{\widetilde{L}_{x_{3}}^{r_{0}}(0,\infty;\dot{B}_{p,q}^{\frac{2}{p} + \frac{1}{r_{0}} - 2})} 
  \leq \varepsilon_{1}. 
\end{equation*}
Thus, for any $\Omega\in \mathbb{R}\setminus\{0\}$ with $|\Omega|\geq \Omega_{1}$, we obtain 
\begin{align*}
  &{\left\|P_{2^{j}> C_{0}|\Omega|^{\frac{1}{2}}}({\bf w}^{(1)} + {\bf v}^{(1)}) + P_{C_{0}^{-1}|\Omega|^{\frac{1}{2}} \leq 2^{j} \leq C_{0}|\Omega|^{\frac{1}{2}}}({\bf w}^{(2)} + {\bf v}^{(2)})\right\|}_{\widetilde{L}_{x_{3}}^{r}(0,\infty;\dot{B}_{p,q}^{\frac{2}{p} + \frac{1}{r} - 1})}\\ 
  \leq & C \varepsilon_{1}
  + C {\left\|{\bf u}\otimes{\bf u}\right\|}_{\widetilde{L}_{x_{3}}^{\frac{r}{2}}(0,\infty;\dot{B}_{p,q}^{\frac{2}{p} + \frac{2}{r} - 2})} 
  + \frac{\widetilde{C}}{C_{0}^{2}} {\left\|{\bf u}\right\|}_{\widetilde{L}_{x_{3}}^{r}(0,\infty;\dot{B}_{p,q}^{\frac{2}{p} + \frac{1}{r} - 1})}. 
\end{align*}
For the estimate of ${\bf w}^{(2)}$, we recall \eqref{0610-3}: 
\begin{align*}
  {\left\|P_{2^{j}< C_{0}^{-1}|\Omega|^{\frac{1}{2}}}{\bf w}^{(2)}\right\|}_{\widetilde{L}_{x_{3}}^{r}(0,\infty;\dot{B}_{p,q}^{\frac{2}{p} + \frac{1}{r} - 1})} 
  \leq&C |\Omega|^{- \frac{1}{r_{0}} + \frac{1}{r}}{\left\|P_{2^{j}< C_{0}^{-1}|\Omega|^{\frac{1}{2}}}{\bf F}\right\|}_{\widetilde{L}_{x_{3}}^{r_{0}}(0,\infty;\dot{B}_{p,q}^{\frac{2}{p} + \frac{3}{r_{0}} - \frac{2}{r} - 2})}\\ 
  &+ C |\Omega|^{- \frac{1}{2} + \frac{1}{2r_{0}} - \frac{1}{2r}}{\left\|P_{2^{j}< C_{0}^{-1}|\Omega|^{\frac{1}{2}}}{\bf F}\right\|}_{\widetilde{L}_{x_{3}}^{r_{0}}(0,\infty;\dot{B}_{p,q}^{\frac{2}{p} + \frac{1}{r} - 1})}\\ 
  &+ C |\Omega|^{-\frac{1}{r}}{\left\|P_{2^{j}< C_{0}^{-1}|\Omega|^{\frac{1}{2}}}({\bf u}\otimes{\bf u})\right\|}_{\widetilde{L}_{x_{3}}^{\frac{r}{2}}(0,\infty;\dot{B}_{p,q}^{\frac{2}{p} + \frac{4}{r} - 2})}\\ 
  &+ C |\Omega|^{- \frac{1}{2} + \frac{1}{2r}}{\left\|P_{2^{j}< C_{0}^{-1}|\Omega|^{\frac{1}{2}}}({\bf u}\otimes{\bf u})\right\|}_{\widetilde{L}_{x_{3}}^{\frac{r}{2}}(0,\infty;\dot{B}_{p,q}^{\frac{2}{p} + \frac{1}{r} - 1})}. 
\end{align*}
Since $r_{0} < r$, we have 
\begin{equation*}
  -\frac{1}{r_{0}} + \frac{1}{r} < 0. 
\end{equation*}
At this point, we can apply Lemma~\ref{0619-15}, that is, 
for fixed $C_{0}$ and any $\varepsilon_{2} > 0$, there exists a constant $\Omega_{2} > 0$ such that for any $\Omega\in \mathbb{R}\setminus\{0\}$ with $|\Omega|\geq \Omega_{2}$, 
it follows that 
\begin{equation*}
  |\Omega|^{- \frac{1}{r_{0}} + \frac{1}{r}}{\left\|P_{2^{j}< C_{0}^{-1}|\Omega|^{\frac{1}{2}}}{\bf F}\right\|}_{\widetilde{L}_{x_{3}}^{r_{0}}(0,\infty;\dot{B}_{p,q}^{\frac{2}{p} + \frac{3}{r_{0}} - \frac{2}{r} - 2})} 
  \leq \varepsilon_{2}
\end{equation*}
and 
\begin{equation*}
  |\Omega|^{- \frac{1}{2} + \frac{1}{2r_{0}} - \frac{1}{2r}}{\left\|P_{2^{j}< C_{0}^{-1}|\Omega|^{\frac{1}{2}}}{\bf F}\right\|}_{\widetilde{L}_{x_{3}}^{r_{0}}(0,\infty;\dot{B}_{p,q}^{\frac{2}{p} + \frac{1}{r} - 1})} 
  \leq \varepsilon_{2}. 
\end{equation*}
Hence we obtain 
\begin{align*}
  {\left\|P_{2^{j}< C_{0}^{-1}|\Omega|^{\frac{1}{2}}}{\bf w}^{(2)}\right\|}_{\widetilde{L}_{x_{3}}^{r}(0,\infty;\dot{B}_{p,q}^{\frac{2}{p} + \frac{1}{r} - 1})} 
  \leq&C \varepsilon_{2} 
  + C {\left\|{\bf u}\otimes{\bf u}\right\|}_{\widetilde{L}_{x_{3}}^{\frac{r}{2}}(0,\infty;\dot{B}_{p,q}^{\frac{2}{p} + \frac{2}{r} - 2})}. 
\end{align*}
Similarly, thanks to $2 < r$, applying Lemma~\ref{0619-15} to estimate \eqref{0610-4}, we get 
\begin{equation*}
  {\left\|P_{2^{j}< C_{0}^{-1}|\Omega|^{\frac{1}{2}}}{\bf v}^{(2)}\right\|}_{\widetilde{L}_{x_{3}}^{r}(0,\infty;\dot{B}_{p,q}^{\frac{2}{p} + \frac{1}{r} - 1})} 
  \leq C \varepsilon_{2} 
  + C {\left\|{\bf u}\otimes{\bf u}\right\|}_{\widetilde{L}_{x_{3}}^{\frac{r}{2}}(0,\infty;\dot{B}_{p,q}^{\frac{2}{p} + \frac{2}{r} - 2})}. 
\end{equation*}
Let 
\begin{equation*}
  \Omega_{0} = \max\{\Omega_{1}, \Omega_{2}\}. 
\end{equation*}
Then for any $\Omega\in \mathbb{R}\setminus\{0\}$ with $|\Omega|\geq \Omega_{0}$, we obtain 
\begin{align*}
  {\|\Phi_{C_{0}}({\bf u})\|}_{\widetilde{L}_{x_{3}}^{r}(0,\infty;\dot{B}_{p,q}^{\frac{2}{p} + \frac{1}{r} - 1})} 
  \leq& C \varepsilon_{1}
  + C \varepsilon_{2} 
  + C \varepsilon_{3}
  + C {\left\|{\bf u}\otimes{\bf u}\right\|}_{\widetilde{L}_{x_{3}}^{\frac{r}{2}}(0,\infty;\dot{B}_{p,q}^{\frac{2}{p} + \frac{2}{r} - 2})}\\ 
  &+ \frac{\widetilde{C}}{C_{0}^{2}} {\left\|{\bf u}\right\|}_{\widetilde{L}_{x_{3}}^{r}(0,\infty;\dot{B}_{p,q}^{\frac{2}{p} + \frac{1}{r} - 1})}. 
\end{align*}
By Lemma~\ref{0619-9} and ${\bf u}\in X$, there exists a constants $C_{3}>0$ such that we have 
\begin{equation*}
  {\|\Phi_{C_{0}}({\bf u})\|}_{\widetilde{L}_{x_{3}}^{r}(0,\infty;\dot{B}_{p,q}^{\frac{2}{p} + \frac{1}{r} - 1})}\\ 
  \leq C_{3} \varepsilon_{1}
  + C_{3} \varepsilon_{2} 
  + C_{3} \varepsilon_{3} 
  + C_{3} \varepsilon_{0}^{2} 
  + \frac{\widetilde{C}}{C_{0}^{2}}\varepsilon_{0}. 
\end{equation*}
We take $C_{0}$ sufficiently large with 
\begin{equation*}
  \frac{\widetilde{C}}{C_{0}^{2}} \leq \frac{1}{2}. 
\end{equation*}
Also, let 
\begin{equation*}
  \varepsilon_{0} 
  = \frac{1}{4C_{3}}. 
\end{equation*}
Since $\varepsilon_{1}$, $\varepsilon_{2}$ and $\varepsilon_{3}$ are arbitrary, we choose 
\begin{equation*}
  \varepsilon_{1} 
  = \varepsilon_{2} 
  = \varepsilon_{3} 
  = \frac{1}{48C_{3}^{2}}. 
\end{equation*}
Then it follows that 
\begin{equation*}
  {\|\Phi_{C_{0}}({\bf u})\|}_{\widetilde{L}_{x_{3}}^{r}(0,\infty;\dot{B}_{p,q}^{\frac{2}{p} + \frac{1}{r} - 1})}\\ 
  \leq \varepsilon_{0} 
\end{equation*}
for any ${\bf u} \in X$, 
which implies $\Phi_{C_{0}}({\bf u}) \in X$. 
Therefore, the remainder of the proof follows from the same argument as in the proof of Theorem~\ref{0609-3}. 
Consequently, we can construct a unique solution ${\bf u}\in X$ to \eqref{SNSC} 
satisfying ${\bf u} \in C_{x_{3}}(0,\infty ; \dot{B}_{p,q}^{\frac{2}{p} - 1}(\mathbb{R}^{2}))$ and \eqref{0609-1}. 

\vskip 7mm 

\noindent 
{\bf Acknowledgments.} The author would like to thank Professor Tsukasa Iwabuchi for valuable discussions and continuous encouragement. 
The author is grateful to Professor Keiichi Watanabe for valuable comments on an earlier version of this paper, 
which helped improve the presentation and clarify several points. 
The author was supported by the Grant-in-Aid for JSPS Fellows, Grant Number 26KJ0527.\\ 
{\bf Data availability statement.} This manuscript has no associated data.\\
{\bf Conflict of Interest.} The author declares that he has no conflict of interest.


\begin{thebibliography}{xx}

  \bib{Ba_Ma_Ni_1997}{article}{
   author={Babin, Anatoli},
   author={Mahalov, Alex},
   author={Nicolaenko, Basil},
   title={Regularity and integrability of $3$D Euler and Navier-Stokes
   equations for rotating fluids},
   journal={Asymptot. Anal.},
   volume={15},
   date={1997},
   number={2},
   pages={103--150},
}

 \bib{Ba_Ma_Ni_1999}{article}{
   author={Babin, Anatoli},
   author={Mahalov, Alex},
   author={Nicolaenko, Basil},
   title={Global regularity of 3D rotating Navier-Stokes equations for
   resonant domains},
   journal={Indiana Univ. Math. J.},
   volume={48},
   date={1999},
   number={3},
   pages={1133--1176},
}

 \bib{Ba_Ma_Ni_2001}{article}{
   author={Babin, Anatoli},
   author={Mahalov, Alex},
   author={Nicolaenko, Basil},
   title={3D Navier-Stokes and Euler equations with initial data
   characterized by uniformly large vorticity},
   note={Dedicated to Professors Ciprian Foias and Roger Temam (Bloomington,
   IN, 2000)},
   journal={Indiana Univ. Math. J.},
   volume={50},
   date={2001},
   pages={1--35},
}

  \bib{Ba_Ch_Da_2011}{book}{
    author={Bahouri, Hajer},
    author={Chemin, Jean-Yves},
    author={Danchin, Rapha\"el},
    title={Fourier analysis and nonlinear partial differential equations},
    series={Grundlehren der mathematischen Wissenschaften [Fundamental
    Principles of Mathematical Sciences]},
    volume={343},
    publisher={Springer, Heidelberg},
    date={2011},
    pages={xvi+523},
 }

\bib{Ch_De_Ga_Gr_2002_2}{article}{
   author={Chemin, Jean-Yves},
   author={Desjardins, Beno\^it},
   author={Gallagher, Isabelle},
   author={Grenier, Emmanuel},
   title={Ekman boundary layers in rotating fluids},
   note={A tribute to J. L. Lions},
   journal={ESAIM Control Optim. Calc. Var.},
   volume={8},
   date={2002},
   pages={441--466},
}

 \bib{Ch_De_Ga_Gr_2002}{article}{
   author={Chemin, Jean-Yves},
   author={Desjardins, Beno\^it},
   author={Gallagher, Isabelle},
   author={Grenier, Emmanuel},
   title={Anisotropy and dispersion in rotating fluids},
   conference={
      title={Nonlinear partial differential equations and their
      applications. Coll\`ege de France Seminar, Vol. XIV},
      address={Paris},
      date={1997/1998},
   },
   book={
      series={Stud. Math. Appl.},
      volume={31},
      publisher={North-Holland, Amsterdam},
   },
   date={2002},
   pages={171--192},
}

 \bib{Ch_De_Ga_Gr_2006}{book}{
   author={Chemin, Jean-Yves},
   author={Desjardins, Beno\^it},
   author={Gallagher, Isabelle},
   author={Grenier, Emmanuel},
   title={Mathematical geophysics},
   series={Oxford Lecture Series in Mathematics and its Applications},
   volume={32},
   note={An introduction to rotating fluids and the Navier-Stokes equations},
   publisher={The Clarendon Press, Oxford University Press, Oxford},
   date={2006},
   pages={xii+250},
}

 \bib{Ch_Fa_Ga_2026}{article}{
   author={Chemin, Jean-Yves},
   author={Fanelli, Francesco},
   author={Gallagher, Isabelle},
   title={Ekman boundary layers in a domain with topography},
   journal={Trans. Amer. Math. Soc.},
   volume={379},
   date={2026},
   number={5},
   pages={3063--3110},
}

 \bib{Ch_Le_1995}{article}{
   author={Chemin, J.-Y.},
   author={Lerner, N.},
   title={Flot de champs de vecteurs non lipschitziens et \'equations de
   Navier-Stokes},
   language={French},
   journal={J. Differential Equations},
   volume={121},
   date={1995},
   number={2},
   pages={314--328},
}

 \bib{Ch_1993}{article}{
   author={Chen, Zhi Min},
   title={$L^n$ solutions of the stationary and nonstationary Navier-Stokes
   equations in ${\bf R}^n$},
   journal={Pacific J. Math.},
   volume={158},
   date={1993},
   number={2},
   pages={293--303},
}

 \bib{Da_Ge_2017}{article}{
   author={Dalibard, Anne-Laure},
   author={G\'erard-Varet, David},
   title={Nonlinear boundary layers for rotating fluids},
   journal={Anal. PDE},
   volume={10},
   date={2017},
   number={1},
   pages={1--42},
}

 \bib{Da_Pr_2014}{article}{
   author={Dalibard, Anne-Laure},
   author={Prange, Christophe},
   title={Well-posedness of the Stokes-Coriolis system in the half-space
   over a rough surface},
   journal={Anal. PDE},
   volume={7},
   date={2014},
   number={6},
   pages={1253--1315},
}

 \bib{Da_2000}{article}{
   author={Danchin, R.},
   title={Global existence in critical spaces for compressible Navier-Stokes
   equations},
   journal={Invent. Math.},
   volume={141},
   date={2000},
   number={3},
   pages={579--614},
}

 \bib{Du_2005}{article}{
   author={Dutrifoy, Alexandre},
   title={Examples of dispersive effects in non-viscous rotating fluids},
   language={English, with English and French summaries},
   journal={J. Math. Pures Appl. (9)},
   volume={84},
   date={2005},
   number={3},
   pages={331--356},
}

 \bib{Ek_1905}{article}{
   author={Ekman, Vagn Walfrid},
   title={On the influence of the earth's rotation on ocean currents},
   journal={Arkiv. Matem., Astr. Fysik},
   volume={2},
   date={1905},
   number={11},
}

  \bib{Fu_2025}{article}{
   author={Fujii, Mikihiro},
   title={Stationary Navier-Stokes equations on the half spaces in the
   scaling critical framework},
   journal={J. Differential Equations},
   volume={435},
   date={2025},
   pages={Paper No. 113298, 25},
}

  \bib{Fu_2026}{article}{
   author={Fujii, Mikihiro},
   title={Sharp non-uniqueness for the Navier-Stokes equations in scaling critical spaces},
   journal={arXiv:2602.19846},
   date={2026},
}

 \bib{Ge_2005}{article}{
   author={Gerard-Varet, David},
   title={Formal derivation of boundary layers in fluid mechanics},
   journal={J. Math. Fluid Mech.},
   volume={7},
   date={2005},
   number={2},
   pages={179--200},
}

 \bib{Gi_I_Ma_Ma_2005}{article}{
   author={Giga, Yoshikazu},
   author={Inui, Katsuya},
   author={Mahalov, Alex},
   author={Matsui, Shin'ya},
   title={Uniform local solvability for the Navier-Stokes equations with the
   Coriolis force},
   journal={Methods Appl. Anal.},
   volume={12},
   date={2005},
   number={4},
   pages={381--393},
}

 \bib{Gi_I_Ma_Ma_2006}{article}{
   author={Giga, Yoshikazu},
   author={Inui, Katsuya},
   author={Mahalov, Alex},
   author={Matsui, Shin'ya},
   title={Navier-Stokes equations in a rotating frame in $\Bbb R^3$ with
   initial data nondecreasing at infinity},
   journal={Hokkaido Math. J.},
   volume={35},
   date={2006},
   number={2},
   pages={321--364},
}

 \bib{Gi_In_Ma_Ma_Sa_2007}{article}{
   author={Giga, Yoshikazu},
   author={Inui, Katsuya},
   author={Mahalov, Alex},
   author={Matsui, Shin'ya},
   author={Saal, J\"urgen},
   title={Rotating Navier-Stokes equations in $\Bbb R_+^3$ with initial data
   nondecreasing at infinity: the Ekman boundary layer problem},
   journal={Arch. Ration. Mech. Anal.},
   volume={186},
   date={2007},
   number={2},
   pages={177--224},
}

 \bib{Gi_I_Ma_Sa_2008}{article}{
   author={Giga, Yoshikazu},
   author={Inui, Katsuya},
   author={Mahalov, Alex},
   author={Saal, J\"urgen},
   title={Uniform global solvability of the rotating Navier-Stokes equations
   for nondecaying initial data},
   journal={Indiana Univ. Math. J.},
   volume={57},
   date={2008},
   number={6},
   pages={2775--2791},
}

 \bib{Gr_1980}{book}{
   author={Greenspan, Harvey Philip},
   title={The theory of rotating fluids},
   series={Cambridge Monographs on Mechanics and Applied Mathematics},
   note={Reprint of the 1968 original},
   publisher={Cambridge University Press, Cambridge-New York},
   date={1980},
   pages={xii+328},
}

 \bib{Gr_Ma_1997}{article}{
   author={Grenier, E.},
   author={Masmoudi, N.},
   title={Ekman layers of rotating fluids, the case of well prepared initial
   data},
   journal={Comm. Partial Differential Equations},
   volume={22},
   date={1997},
   number={5-6},
   pages={953--975},
}

 \bib{Hi_Sh_2010}{article}{
   author={Hieber, Matthias},
   author={Shibata, Yoshihiro},
   title={The Fujita-Kato approach to the Navier-Stokes equations in the
   rotational framework},
   journal={Math. Z.},
   volume={265},
   date={2010},
   number={2},
   pages={481--491},
}

 \bib{I_2015}{article}{
   author={Iwabuchi, Tsukasa},
   title={Global solutions for the critical Burgers equation in the Besov
   spaces and the large time behavior},
   journal={Ann. Inst. H. Poincar\'e{} C Anal. Non Lin\'eaire},
   volume={32},
   date={2015},
   number={3},
   pages={687--713},
}

 \bib{I_Ta_2014}{article}{
   author={Iwabuchi, Tsukasa},
   author={Takada, Ryo},
   title={Global well-posedness and ill-posedness for the Navier-Stokes
   equations with the Coriolis force in function spaces of Besov type},
   journal={J. Funct. Anal.},
   volume={267},
   date={2014},
   number={5},
   pages={1321--1337},
}

 \bib{I_Ta_2015}{article}{
   author={Iwabuchi, Tsukasa},
   author={Takada, Ryo},
   title={Dispersive effect of the Coriolis force and the local
   well-posedness for the Navier-Stokes equations in the rotational
   framework},
   journal={Funkcial. Ekvac.},
   volume={58},
   date={2015},
   number={3},
   pages={365--385},
}

 \bib{Ka_Ko_Sh_2019}{article}{
   author={Kaneko, Kenta},
   author={Kozono, Hideo},
   author={Shimizu, Senjo},
   title={Stationary solution to the Navier-Stokes equations in the scaling
   invariant Besov space and its regularity},
   journal={Indiana Univ. Math. J.},
   volume={68},
   date={2019},
   number={3},
   pages={857--880},
}

 \bib{Ka_1984}{article}{
   author={Kato, Tosio},
   title={Strong $L\sp{p}$-solutions of the Navier-Stokes equation in ${\bf
   R}\sp{m}$, with applications to weak solutions},
   journal={Math. Z.},
   volume={187},
   date={1984},
   number={4},
   pages={471--480},
}

 \bib{Ko_2015}{article}{
   author={Koba, Hajime},
   title={On stability of the spatially inhomogeneous
   Navier-Stokes-Boussinesq system with general nonlinearity},
   journal={Arch. Ration. Mech. Anal.},
   volume={215},
   date={2015},
   number={3},
   pages={907--965},
}

 \bib{Ko_Le_Ta_2014}{article}{
   author={Koh, Youngwoo},
   author={Lee, Sanghyuk},
   author={Takada, Ryo},
   title={Dispersive estimates for the Navier-Stokes equations in the
   rotational framework},
   journal={Adv. Differential Equations},
   volume={19},
   date={2014},
   number={9-10},
   pages={857--878},
}

  \bib{Ko_Ta_2001}{article}{
   author={Koch, Herbert},
   author={Tataru, Daniel},
   title={Well-posedness for the Navier-Stokes equations},
   journal={Adv. Math.},
   volume={157},
   date={2001},
   number={1},
   pages={22--35},
}

 \bib{Ko_Yo_2011}{article}{
   author={Konieczny, Pawe\l},
   author={Yoneda, Tsuyoshi},
   title={On dispersive effect of the Coriolis force for the stationary
   Navier-Stokes equations},
   journal={J. Differential Equations},
   volume={250},
   date={2011},
   number={10},
   pages={3859--3873},
}

 \bib{Ko_Ya_1994}{article}{
   author={Kozono, Hideo},
   author={Yamazaki, Masao},
   title={Semilinear heat equations and the Navier-Stokes equation with
   distributions in new function spaces as initial data},
   journal={Comm. Partial Differential Equations},
   volume={19},
   date={1994},
   number={5-6},
   pages={959--1014},
}

 \bib{Ma_2000}{article}{
   author={Masmoudi, Nader},
   title={Ekman layers of rotating fluids: the case of general initial data},
   journal={Comm. Pure Appl. Math.},
   volume={53},
   date={2000},
   number={4},
   pages={432--483},
}

 \bib{Oh_Ta_2021}{article}{
   author={Ohyama, Hiroki},
   author={Takada, Ryo},
   title={Asymptotic limit of fast rotation for the incompressible
   Navier-Stokes equations in a 3D layer},
   journal={J. Evol. Equ.},
   volume={21},
   date={2021},
   number={2},
   pages={2591--2629},
}

  \bib{Tr_1983}{book}{
   author={Triebel, Hans},
   title={Theory of function spaces},
   series={Monographs in Mathematics},
   volume={78},
   publisher={Birkh\"auser Verlag, Basel},
   date={1983},
   pages={284},
}

 \bib{Tsu_2019}{article}{
   author={Tsurumi, Hiroyuki},
   title={Ill-posedness of the stationary Navier-Stokes equations in Besov
   spaces},
   journal={J. Math. Anal. Appl.},
   volume={475},
   date={2019},
   number={2},
   pages={1732--1743},
}

 \bib{Tsu_2019_2}{article}{
   author={Tsurumi, Hiroyuki},
   title={Well-posedness and ill-posedness problems of the stationary
   Navier-Stokes equations in scaling invariant Besov spaces},
   journal={Arch. Ration. Mech. Anal.},
   volume={234},
   date={2019},
   number={2},
   pages={911--923},
}

  \bib{Pe_1979}{book}{
   author={J. Pedlosky},
   title={Geophysical Fluid Dynamics},
   publisher={Springer-Verlag New York},
   date={1979},
}

\end{thebibliography}
\end{document}